\documentclass[a4paper,11pt]{article}

\usepackage[utf8]{inputenc}
\usepackage[T1]{fontenc}
\usepackage[english]{babel}
\usepackage{amsmath, amsfonts, amsthm,amssymb,here,dsfont, hyperref}
\usepackage{mathtools}
\usepackage{graphicx,color,subfigure,multirow, here}
\usepackage{natbib}
\usepackage{enumitem}
\usepackage{comment}

\newtheorem{theo}{Theorem}[section]

\newtheorem{prop}[theo]{Proposition}

\newtheorem{lemme}[theo]{Lemma}
\newtheorem{remark}[theo]{Remark}

\newtheorem{assumption}[theo]{Assumption}

\newcommand{\w}{\widehat}

\usepackage[textwidth=1.8cm, textsize=tiny]{todonotes}

\title{Minimax convergence rates of a binary plug-in type classification procedure for time-homogeneous SDE paths under low-noise conditions}

\author{Eddy-Michel Ella-Mintsa}

\begin{document}

\maketitle
\begin{center}
 Institut de Recherche Technologique, CENAREST, BP 9154 Libreville, Gabon
\end{center}

\date{}

\begin{abstract}
The study of minimax convergence rates for classification procedures adapted to SDE paths remains relatively scarce in the literature. Existing results are limited to the white noise model and, more recently, to time-homogeneous SDEs with a focus on derivation of standard minimax rates. In this paper, we consider a diffusion model characterized by a time-homogeneous SDE with a space-dependent drift coefficient depending on the class and a diffusion coefficient that is common to the two classes. We establish, under low-noise conditions on the regression function, a faster convergence rate of order $\log^4(N)N^{-2\beta/(2\beta+1)}$ over a H\"older space of smoothness parameter $\beta \geq 1$. This result will require the establishment of an exponential inequality, which is essential to obtain the expected rate. We then prove that it is not possible to achieve a convergence rate that is faster than $N^{-2\beta/(2\beta+1)}$.
\end{abstract}

\vspace*{0.25cm} 



\section{Introduction}
\label{sec:intro}

We study some key theoretical properties of supervised classification adapted to time-homogeneous diffusion processes. The classification procedure is binary and plug-in type and is built from $N$ independent copies of a random pair $(X, Y)$ that belongs to a probability space $\left(\Omega, \mathcal{F}, \mathbb{P}_{X,Y}\right)$, where the characteristic $X = (X_t)_{t \in [0,T]}$ is a diffusion process whose drift coefficient $b_Y^*$ is unknown and depends on the label $Y \in \mathcal{Y} = \{0,1\}$, and whose diffusion coefficient is known and common to the two classes. The label $Y$ follows an unknown discrete law ${\bf p^*} = \left(p_0^*, p_1^*\right) \in (0,1)^2$, and the characteristic $X$ belongs to the probability space $\left(\mathcal{X}, \mathfrak{F}, \mathbb{P}_X\right)$ where $\mathcal{X} = \mathcal{C}([0,T], \mathbb{R})$ is the space of continuous functions on the compact interval $[0,T]$ with $T>0$, and $\mathfrak{F}$ its $\sigma$-field. Any measurable function $g$ that maps the set of trajectories $\mathcal{X}$ into the set of labels $\mathcal{Y}$ is called a classifier or a classification rule. Its performance is measured by $\mathbb{P}_{X,Y}(g(X) \neq Y)$ called the risk of misclassification or classification error. A Bayes classifier $g^*$ is an optimal classification rule in the sense of minimizing the classification error. As a result, the performance of any classifier $g: \mathcal{X} \rightarrow \mathcal{Y}$ is assessed via its excess risk $\mathbb{P}_{X,Y}(g(X) \neq Y) - \mathbb{P}_{X,Y}(g^{*}(X) \neq Y)$ with respect to the Bayes classifier. Since the joint distribution of the random couple $(X,Y)$ is unknown, we suppose to have a learning sample $\left\{(\bar{X}^j, Y_j), ~~ j = 1, \ldots, N\right\}$ constituted of $N$ independent copies of $(\bar{X},Y)$, where $\bar{X} = (X_{t_k^n})_{0 \leq k \leq n}$ is a discrete observation of $X$ with $\Delta_n = T/n$ and $n \rightarrow \infty$ (high frequency observations). We then build an empirical classification procedure of plug-in type $\w{g}$ whose excess risk tends to zero as the size $N$ of the learning sample tends to infinity. This paper addresses the establishment of a faster convergence rate of the empirical classifier $\w{g}$, that is, a rate that is faster than $N^{-1/2}$ as $N$ tends to infinity.    

\subsection{Generality and related works}

Functional data analysis is attracting growing interest, driven by the massive and increasing availability of this type of data in various fields of application such as finance (see, \textit{e.g.}, \cite{lamberton2011introduction}, \cite{el1997backward}), biology (see \cite{crow2017introduction}) or ecology (see, \textit{e.g.}, \cite{nagai1983asymptotic}). The work of Ramsay and Silverman is one of the pioneering references in functional data analysis and its applications in various fields (see \cite{ramsay2005fitting}).\\
One of the analyzes of great interest is the discriminant analysis whose aim is to build, from labeled data, a classification procedure to predict a predefined class for any new observation or feature. There is a large literature on this statistical technique that leads to the creation of multiple algorithms that are at the core of machine learning. This statistical method was first emphasized by \cite{fisher1936use}, \cite{rao1948utilization} or \cite{anderson1958introduction}. General approaches are studied in \cite{mclachlan2005discriminant} and \cite{devroye2013probabilistic}, in which particular attention is paid to the Bayesian rule that leads to more efficient classifiers. Increasingly sophisticated algorithms are developed according to the type of data being considered. Some of the algorithms are based on the plug-in principle (see, \textit{e.g.} \cite{devroye1980distribution}, \cite{audibert2007fast}), or the empirical minimization principle (see, \textit{e.g.}, \cite{mammen1999smooth}, \cite{bartlett2006convexity}). One can also find decision rules based on the k-Nearest Neighbors principle (see, \textit{e.g.} \cite{cover1967nearest}, \cite{gyorfi2006distribution},  \cite{devroye2013probabilistic}), or depth-based classification rules (see, \cite{tukey1975mathematics}, \textit{e.g.} \cite{cuevas2007robust}).\\ 
For the specific case of functional data analysis, significant contributions have been made in classification problems. One can cite, for example, classification procedures based on dimension reduction (see, \textit{e.g.} \cite{james2001functional}, \cite{hyndman2009forecasting},  \cite{delaigle2012achieving}), those based on nonparametric and kernel methods (see, \textit{e.g.} \cite{ferraty2003curves}, \cite{baillo2011classification}), depth-based classification rules (see, \textit{e.g.} \cite{cuevas2007robust}), or classification rules constructed using neural networks (see, \textit{e.g.} \cite{wang2023deep}, \cite{wang2024functional}). Not far from our framework, in the context of supervised classification of functional data modeled by diffusion processes, the first classification procedure to be constructed is based on minimization of the empirical classification error and proposed in \cite{cadre2013supervised}. This contribution is followed by the construction of plug-in type classifiers based on parametric or nonparametric estimators of the drift and diffusion coefficients (see \cite{denis2020classif}, \cite{gadat2020optimal}, \cite{denis2024nonparametric}, \cite{ella2026minimax}). In a nonparametric setting, plug-in type classifiers imply the construction of nonparametric estimators of drift and diffusion coefficients, which is another statistical challenge, especially when the diffusion process is unbounded, leading to nonparametric estimations on non compact intervals. Fortunately, there are multiple contributions on nonparametric estimation of coefficients of stochastic differential equations, particularly from independent and identically distributed diffusion processes (see \cite{comte2020nonparametric}, \cite{denis2021ridge}, \cite{marie2023nadaraya}, \cite{denis2024nonparametric}, \cite{ella2024nonparametric}, \cite{ella2025minimax}). However, for the statistical problem considered in this paper, projection estimators of the coefficients of diffusion processes do not appear to be appropriate for establishing our main results. In fact, a suitable nonparametric estimator of the drift or diffusion coefficient should be ideally defined as an empirical mean of independent random variables since it allows to establish exponential inequalities crucial for the purpose of this paper. Then, to our knowledge, only \cite{marie2023nadaraya} proposed  nonparametric drift estimators that satisfy the required properties that can lead to the establishment of exponential inequalities.\\
Focusing on the study of minimax convergence rates for classification rules, we have, for example, \cite{Yang99} which established that an empirical classification rule cannot reach a convergence rate faster than $N^{-1/2}$ under complexity assumptions on the regression function. However, \cite{bartlett2006convexity} and \cite{audibert2007fast} proved that faster convergence rates can be reached under a Margin Assumption (MA) or low-noise conditions on the regression function, that is, a condition in which the regression function $\Phi^*(X) = \mathbb{P}_{X,Y}(Y = 1|X)$ is unlikely to be in the neighborhood of $1/2$. To the best of our knowledge, only \cite{gadat2020optimal} established both standard and fast minimax-optimal convergence rates for binary classification based on trajectories generated by Gaussian processes. More recently, \cite{ella2026minimax} derived the standard minimax-optimal rates in the context of binary classification of trajectories generated by time-homogeneous stochastic differential equations (SDEs).

The present paper extends the analysis of \cite{gadat2020optimal} to mixtures of stochastic differential equations with space-dependent drift and diffusion coefficients. Moreover, under low-noise conditions, it significantly improves the optimal minimax rates established in \cite{ella2026minimax}. These extensions are far from straightforward and raise substantial new challenges, both in the theoretical analysis of nonparametric estimators for the coefficients of diffusion processes and in the study of the statistical properties of the resulting plug-in classifier.

\subsection{Main contributions}

We suppose to have at our disposal a learning sample $\left\{(\bar{X}^j, Y_j), ~ j = 1, \ldots, N\right\}$, where the couples $(\bar{X}^j, Y_j)$ are independent copies of the random pair $(\bar{X},Y)$ with $Y \in \mathcal{Y} = \{0,1\}$ and $\bar{X}$ a discrete observation of the diffusion process $X$ whose drift coefficient $b_Y^*$ is unknown and depends on the label $Y$ and whose diffusion coefficient is known and common to all classes. The objective is to establish a minimax convergence rate faster than $N^{-1/2}$ for the excess risk $\mathbb{P}_{X,Y}(\w{g}(X) \neq Y) - \mathbb{P}_{X,Y}(g^*(X) \neq Y)$ of the plug-in classifier $\w{g}$. The establishment of a faster convergence rate is possible under the low-noise condition on the regression function $\Phi^*: X \in (\mathcal{X}, \mathfrak{F}) \mapsto \mathbb{P}_{X,Y}(Y=1|X)$ (see \cite{audibert2007fast}). More precisely, the regression function $\Phi^*$ satisfies the following
\begin{equation*}
    \forall \varepsilon > 0, ~~\mathbb{P}_X\left(0 < \left|\Phi^*(X) - \dfrac{1}{2}\right| \leq \varepsilon\right) = \mathrm{O}\left(\varepsilon^{\alpha}\right),
\end{equation*}
where $\alpha > 0$, the case $\alpha = 0$ being not interesting for the purpose of this paper. The above result is established in \cite{denis2025empirical}, \textit{Proposition 4.3}, for $\alpha = 1$ assuming that the random variable $Z_T = \int_{0}^{T}(b_1^* - b_0^*)(X_s)dW_s$ admits a smooth transition density, essential to obtain the expected result. However, the existence of a smooth probability density for $Z_T$ is generally established under strong assumptions on the diffusion model. In this article, we prove that $Z_T$ admits a smooth density keeping assumptions on $b_i^*, ~ i \in \mathcal{Y}$ as weak as possible. Once the low-noise condition is established, the main results of this paper are as follows. 

\begin{enumerate}
    \item We consider Nadaraya-Watson estimators $\w{b}_{i, N, h_i, h_i^{\prime}}, ~ i \in \mathcal{Y}$ of drift coefficients $b_i^*, ~ i \in \mathcal{Y}$ proposed in \cite{marie2023nadaraya} from sub-samples of the sample $\{\bar{X}^1, \ldots, \bar{X}^N\}$, each coefficient $b_i^*$ being estimated from diffusion paths belonging to the class $i$. Recall that $N$ is the size of the learning sample and $h_i, h_i^{\prime} > 0$ are the bandwidths. We establish, under suitable assumptions on the diffusion model, the following exponential inequality.
    \begin{align*}
        \mathbb{P}_i^{\otimes N}\left(\left\|\widehat{b}_{i,N,h,h^{\prime}} - b_i^*\right\|_{\infty} \geq \delta\right) \leq &~ 2\exp\left(-\mathbf{C}N_i\delta^2h\right) + N_i\exp\left(-\dfrac{(T-t_0)\log^2(N)}{2\|K\|_{\infty}^2}\right)\\
        & + 4\exp\left(-\mathbf{C}N_i\delta^2h^{\prime}\right) + \dfrac{6\|b_i^*\|_{\infty}}{\delta}\exp\left(-\mathbf{C}^{\prime}N_ih\right),
    \end{align*}
    where $\delta \in (0,1)$, $\mathbf{C}, \mathbf{C}^{\prime} > 0$ are constants that depend on $T, b_0^*$ and $b_1^*$, $N_i := \sum_{j=1}^{N}\mathds{1}_{Y_j = i} \sim \mathrm{Binomial}(N, p_i^*)$ is the size of the sub-sample made up of diffusion paths of class $i$, $t_0 \in [0,T]$ is fixed, $K$ is a kernel, and $\mathbb{P}_i^{\otimes N}$ is the conditional joint probability distribution of the learning sample given $\{\mathds{1}_{Y_1=i}, \ldots, \mathds{1}_{Y_N = i}\}$. This exponential inequality is essential for establishing a faster convergence rate for the excess risk of the plug-in classifier. 
    \item From the low-noise condition, together with the above exponential inequality, we show that the excess risk $\mathbb{P}_{X,Y}(\w{g}(X) \neq Y) - \mathbb{P}_{X,Y}(g^*(X) \neq Y)$ of the plug-in classifier $\w{g}$ converges to zero with a rate of order $\log^4(N)N^{-2\beta/(2\beta+1)}$ over the H\"older class of smoothness parameter $\beta \geq 1$. The logarithmic factor is the result of two main complications. The first is the complexity of the Nadaraya-Watson estimator, which is a ratio of two estimators. The second is the fact that, multiple times, we deal with unbounded random variables. This leads us to consider random events in which these variables are bounded for the application of concentration inequalities such as Bernstein's inequality or the inequality established \cite{van1995exponential}, \textit{Lemma 2.1}. 
    \item We establish a lower bound on the average excess risk $\mathbb{E}_{\mathbb{P}^{\otimes N}}\left[\mathbb{P}_{X,Y}(\w{g}(X) \neq Y)\right] - \mathbb{P}_{X,Y}(g^*(X) \neq Y)$ of order $N^{-2\beta/(2\beta+1)}$, where $\mathbb{P}^{\otimes N}$ is the joint probability distribution of the learning sample and $\mathbb{E}_{\mathbb{P}^{\otimes N}}$ its corresponding expectation. Note that the study of the lower bound of the excess risk requires the use of the explicit formula of the transition density of the diffusion process $X$ provided in \cite{dacunha1986estimation}, and the equivalence relation between the image probability distribution of $X$ and the Wiener measure. In fact, the proof of the lower bound relies on Assouad's lemma adapted to the classification problem provided in \cite{Audibert2004ClassificationUP}. We then have to build a hypercube which includes a partition $\left\{\mathcal{X}_1, \ldots, \mathcal{X}_m\right\}$ of the infinite-dimensional space $\mathcal{X} = \mathcal{C}([0,T], \mathbb{R})$ with $m \in \mathbb{N}^*$ such that for all $i \in \{1, \ldots, m\}$ $\mathbb{P}_X(X \in \mathcal{X}_i) = w>0$, where $w$ is independent of $i$. To this end, the use of the exact formula for the transition density of $X$ is crucial for the construction of the hypercube. Moreover, we will need to ensure that for each $i \in \{1, \ldots, m\}$,  conditional on $\{X \in \mathcal{X}_i\}$, the diffusion process $X$ does not take values in a countable subset of $\mathbb{R}$, which is satisfied if the image measure of $X$ is equivalent to the Wiener measure.
\end{enumerate}

\subsection{Outline of the paper}

In Section~\ref{sec:statSetting}, we present the statistical setting of the paper, which includes the definition of the diffusion model, the notations adopted for this work, the assumptions about our diffusion model, the classification procedure, and low-noise conditions. Sections~\ref{sec:main-results} and \ref{sec:conclusion} are devoted respectively to the main results of the paper and the conclusion. We provide the proofs of the main results in Section~\ref{sec:proofs}. The proofs of intermediate results are provided in the appendix.

\section{Statistical setting}
\label{sec:statSetting}

We consider a classification model whose feature $X = (X_t)_{t \in [0,T]}$ is a short-time diffusion process defined in a filtered probability space $\left(\Omega, \mathcal{F}, \left(\mathcal{F}_t\right)_{t \in [0,T]}, \mathbb{P}_X\right)$, with $T > 0$ the time horizon and $\left(\mathcal{F}_t\right)_{t \in [0,T]}$ the natural filtration, and whose label $Y \in \mathcal{Y} = \{0,1\}$ is a binary random variable. The feature $X$ is solution of a mixture model characterized by the following stochastic differential equation
\begin{equation}\label{eq:classif-model}
    dX_t = b_Y^*(X_t)dt + dW_t, ~~ t \in [0, T], ~~ X_0 = x_0 \in \mathbb{R},
\end{equation}
where $W = (W_t)_{t \in [0, T]}$ is the standard Brownian motion independent of the label $Y$, the drift function 
 $b_Y^*$ is unknown and depends on the label $Y \in \mathcal{Y} = \{0,1\}$, and $b_0^* \neq b_1^*$. Denote by $\mathbf{p}^* = (p_0^*, p_1^*) \in (0,1)^2$ the law of the label $Y$ which is assumed to be unknown.\\
 A classifier is a measurable function $g$ that maps the set of features $\mathcal{X} = \mathcal{C}([0,T], \mathbb{R})$ into the set of labels $\mathcal{Y} = \{0,1\}$. More precisely, for any feature $X \in \mathcal{X}$, $g(X) \in \mathcal{Y}$ is the predicted label of $X$. Thus, the performance of $g$ is assessed through the classification error characterized by the function $R : \mathcal{G} \rightarrow [0,1]$ given for all $g \in \mathcal{G}$ by $R(g) := \mathbb{P}_{X,Y}(g(X) \neq Y)$, where $\mathcal{G}$ is a carefully chosen set of classifiers to be specified later. A Bayes classifier is a classification rule $g^*$ that minimizes the classification error $R$ on the set $\mathcal{G}$, that is:
 \begin{equation}\label{eq:bayes1}
     g^* \in \underset{g \in \mathcal{G}}{\arg\min} ~R(g).
 \end{equation}
 The function $g^*$ is given by $g^*(X) = \mathds{1}_{\Phi^*(X) \geq 1/2}$ where $\Phi^*$ is a regression function defined by $\Phi^*(X) = \mathbb{P}_{X,Y}(Y = 1 | X)$ (see \cite{devroye2013probabilistic} for more details). Since the law of the random pair $(X,Y)$ is unknown, this paper is devoted to the study of the minimax convergence rates of excess risk $R(\widehat{g}) - R(g^*)$ of an empirical classifier $\widehat{g}$ of the plug-in type built from a learning sample $\mathcal{Z}^N$. 

\subsection{Notations and definitions}
\label{subsec:not.def.}

We suppose to have a learning sample $\mathcal{Z}^N = \left\{(\bar{X}^j, Y_j), ~~ j = 1,\ldots, N\right\}$ constituted of $N$ independent copies of the random couple $(\bar{X},Y)$, where $\bar{X} = (X_{t_k})_{0 \leq k \leq n}$ is a discrete observation of the unique strong solution $X$ of Equation~\eqref{eq:classif-model} and its label $Y$. Recall that the distribution of the random pair $(X,Y)$ is denoted by $\mathbb{P}_{X,Y}$, and $\mathbb{E}_{X,Y}$ is the corresponding expectation. We denote by $\mathbb{P}_X$ and $\mathbb{E}_X$, respectively, the probability distribution and the expectation of $X$. The joint distribution of independent copies $(X^1, Y_1), \ldots, (X^N, Y_N)$ of the random pair $(X,Y)$ is indicated by $\mathbb{P}^{\otimes N}$, $\mathbb{E}_{\mathbb{P}^{\otimes N}}$ and $\mathrm{Var}_{\mathbb{P}^{\otimes N}}$ being the corresponding expectation and variance, respectively. Moreover, from the learning sample $\mathcal{Z}^N$, we define the following subsamples:
\begin{equation}\label{eq:Sample}
    \mathcal{Z}_i^N := \left\{\bar{X}^{ji}, ~~ j \in \mathcal{J}_i\right\}, ~~ \mathrm{where} ~~ \mathcal{J}_i := \left\{j \in \{1, \ldots, N\}: ~ (\bar{X}^j, i) \in \mathcal{Z}^N\right\}, ~~ i \in \mathcal{Y} = \{0,1\}.
\end{equation}
More precisely, for each $i \in \mathcal{Y}$, the subsample $\mathcal{Z}_i^N$ contains diffusion paths $\bar{X}^{1i}, \ldots, \bar{X}^{Ni}$ that belong to the class $i$. The subsample $\mathcal{Z}_i^N$ is used to build a nonparametric estimator of the drift function $b_i^*$. Its random size $N_i$ is given by $N_i := \sum_{j \in \mathcal{J}_i}{1} =  \sum_{j=1}^{N}\mathds{1}_{Y_j = i} \sim \mathrm{Binomial}(N, p_i^*)$. For each $i \in \mathcal{Y}$, we define the following conditional probability,
\begin{align*}
    \mathbb{P}_i^{\otimes N}\left(.\right) := \mathbb{P}^{\otimes N}\left(. \biggm\vert \mathds{1}_{Y_1 = i}, \ldots, \mathds{1}_{Y_N = i}\right),
\end{align*}
and its corresponding expectation $\mathbb{E}_{\mathbb{P}_i^{\otimes N}}$ and variance $\mathrm{Var}_{\mathbb{P}_i^{\otimes N}}$ are defined by
\begin{align*}
    &~ \mathbb{E}_{\mathbb{P}_i^{\otimes N}}\left[.\right] := \mathbb{E}_{\mathbb{P}^{\otimes N}}\left[. \biggm\vert \mathds{1}_{Y_1 = i}, \ldots, \mathds{1}_{Y_N = i}\right], ~~  \mathrm{Var}_{\mathbb{P}_i^{\otimes N}}\left[.\right] := \mathrm{Var}_{\mathbb{P}^{\otimes N}}\left[. \biggm\vert \mathds{1}_{Y_1 = i}, \ldots, \mathds{1}_{Y_N = i}\right].
\end{align*}
Finally, in the sequel, we adopt the following notations:
\begin{itemize}
    \item $\mathcal{W}$ is the Wiener measure in space $\mathcal{X} = \mathcal{C}\left([0,T], \mathbb{R}\right)$.
    \item $\mu$ is the Lebesgue measure on $\mathbb{R}$ and $\mu^{(n)}$ the Lebesgue measure on $\mathbb{R}^n, ~ n \in \mathbb{N} \setminus\{0,1\}$. 
    \item When two measures $\mu_1$ and $\mu_2$ are equivalent, we denote $\mu_1 \sim \mu_2$.
    \item For all $p,q \in \mathbb{N}^*$ such that $q > p$, $[\![p,q]\!] = \left\{p, p+1, \ldots, q\right\}$.
    \item For all $p \in \mathbb{R}_{+*}$, $\lfloor p \rfloor$ is the largest integer strictly smaller than $p$, and $\lceil p \rceil$ is the smallest integer strictly greater than $p$.
    \item For all $\phi \in L^2(\mathbb{R}, \mathbb{R})$, $\|\phi\|$ denotes the $L^2-$norm of $\phi$ and $\mathrm{Supp}(\phi)$ its support.
    \item For all $\phi \in \mathcal{C}_b(\mathbb{R}, \mathbb{R})$, $\|\phi\|_{\infty}$ denotes the supremum norm of $\phi$.
    \item For all $\psi \in \mathcal{C}^{\infty}(\mathbb{R}, \mathbb{R})$ and for all $\ell \in \mathbb{N}^*$, $\psi^{(\ell)}$ is the $\ell^{\mathrm{th}}$ derivative of $\psi$.
    \item For all $I \subset \mathbb{R}$, $\mathrm{Int}(I)$ denotes the topological interior of $I$, and $|I| = \sum_{x \in I}1$ its cardinality.
    \item For any subset $A \subset \mathbb{R}$, $x \mapsto \mathds{1}_{x \in A}$ denotes the indicator function of $A$.
\end{itemize}

\subsection{Assumptions}
\label{subsec:ass}

The establishment of the main results of this paper requires one to impose some key assumptions on the diffusion model under study. To this end, the following assumptions are made on the drift coefficients.

\begin{assumption}
\label{ass:Reg}
$b_0^*$ and $b_1^*$ are compactly supported and $\mathrm{Supp}(b_i^*) \subset I$ for each $i \in \mathcal{Y}$, where $I \subset \mathbb{R}$ is a non-empty compact interval independent of ${\bf b}^* = (b_0^*, b_1^*)$. In addition, $b_0^*$ and $b_1^*$ belong to the H\"older class $\Sigma(\beta, R)$ given by
\begin{equation*}
    \Sigma(\beta, R) := \left\{f \in \mathcal{C}^{\lfloor \beta \rfloor}(\mathbb{R}, \mathbb{R}), ~~ \left|f^{(\lfloor \beta \rfloor)}(x) - f^{(\lfloor \beta \rfloor)}(y)\right| \leq R|x-y|^{\beta - \lfloor \beta \rfloor}, ~~ x,y \in \mathbb{R}\right\},
\end{equation*}
where $\beta \geq 1$ is the smoothness parameter of the space and $R>0$.
\end{assumption}

\begin{assumption}\label{ass:hormander}
    $b_0^*$ and $b_1^*$ satisfy $b_0^*(x_0) \neq b_1^*(x_0)$ and $\mu(\{x \in \mathbb{R} : b_0^*(x) \neq b_1^*(x)\})>0$.
\end{assumption}
The above assumptions on our diffusion model provide an adequate framework for establishing the key results of this paper. Assumption~\ref{ass:Reg} states that for each $i \in \mathcal{Y}$ there exist $A^i, B^i \in \mathbb{R}$ such that $A^i < B^i$ and $\mathrm{Supp}(b_i^*) = [A^i, B^i] \subset I$. This assumption is required to ensure that the resulting Nadaraya-Watson estimators of the drift functions are consistent and reach a convergence rate of order $N^{-\beta/(2\beta+1)}$, a crucial point for the study of faster minimax rates for the plug-in classifier. We give more details on this specific point in Section~\ref{sec:main-results}. Moreover, since $b_0^*, b_1^* \in \mathcal{C}^{\lfloor \beta\rfloor}\left(\mathbb{R}, \mathbb{R}\right)$, $b_0^*$ and $b_1^*$ and their derivatives should be smooth on the boundaries of $\mathrm{Supp}(b_0^*)$ and $\mathrm{Supp}(b_1^*)$ respectively. Functions of this kind are largely studied in the literature and are particularly used to build bases of compactly supported functions like the spline basis (see, \textit{e.g.} \cite{gyorfi2006distribution}) or bases of compactly supported wavelet functions (see, \textit{e.g.} \cite{hardle2012wavelets}).\\
The immediate implications of Assumption~\ref{ass:Reg} are the following.
\begin{itemize}
    \item[(i)] The drift coefficients $b_0^*$ and $b_1^*$ are Lipschitz functions, which implies that our diffusion model~\eqref{eq:classif-model} admits a unique strong solution $X = (X_t)_{t \in [0,T]}$ (see, \textit{e.g.} \cite{karatzas2014brownian}, \textit{Chapter 5, Theorem 2.5, p.287}). Moreover, the diffusion process $X$ admits a transition density $(s,t,x,y) \in [0,T]^2 \times \mathbb{R}^2 \mapsto \Gamma_X(s,t,x,y)$ given by $\Gamma_X(s,t,x,y) := p_0^*\Gamma_{X|Y=0}(s,t,x,y) + p_1^*\Gamma_{X|Y=1}(s,t,x,y)$, where $\Gamma_{X|Y=0}$ and $\Gamma_{X|Y=1}$ are, respectively, the transition densities of $X$ on events $\{Y = 0\}$ and $\{Y = 1\}$. The transition density $\Gamma_X$ of $X$ plays a significant role in the study of the lower bound on the excess risk of the plug-in type classifier. 
    \item[(ii)] Since the diffusion coefficient of our model is known and equal to $1$, Assumption~\ref{ass:Reg} implies the following Novikov's condition
    \begin{equation*}
        \mathbb{E}_X\left[\exp\left(\dfrac{1}{2}\int_{0}^{T}b_Y^*(X_s)ds\right)\right] < \infty.
    \end{equation*}
    It follows from Girsanov's theorem (see \textit{e.g.} \cite{revuzyor1999}, Chapter VIII, p.325-333) that 
    \begin{equation*}
        \forall ~ t \in [0,T], ~~ \dfrac{d\mathbb{P}_X}{d\mathbb{P}_W}{\biggm\vert\mathcal{F}_t} := \exp\left(\int_{0}^{t}b_Y^*(X_s)dX_s -\dfrac{1}{2}\int_{0}^{t}b_Y^{*2}(X_s)ds\right),
    \end{equation*}
    where $\mathbb{P}_W$ is the probability distribution of the standard Brownian motion $W = (W_t)_{t \geq 0}$. We deduce that $\mathbb{P}_X \circ X^{-1} \sim \mathcal{W}$. This equivalence relation is extensively used in the proof of Theorem~\ref{thm:lower-bound}. In fact, the proof method, based on the key result of Assouad's lemma adapted to the classification problem, requires the construction of a partition of the space $\mathcal{X} = \mathcal{C}\left([0,T], \mathbb{R}\right)$ of diffusion paths (see \cite{Audibert2004ClassificationUP}). Then, it is crucial to ensure that for any $S \subset \mathcal{X}$ and for any $t \in [0,T]$, $\mu\left(\{f(t), ~ f \in S\}\right) > 0$, which is derived from the fact that $\mathcal{W}(S) > 0$. As a result, conditional on $\{X \in S\}$, the diffusion process $X$ takes values in a subset of $\mathbb{R}$ that contains a non-empty and continuous subset.
\end{itemize}
Assumption~\ref{ass:hormander} provides a framework in which low-noise conditions can be established. In fact, the proof technique for establishing the result of Proposition~\ref{prop:margin-condition} that characterizes low-noise conditions requires the continuity of the probability density of the following random variable
\begin{equation*}
    Z_T := \int_{0}^{T}(b_1^* - b_0^*)(X_t)dW_t.
\end{equation*}
In the literature, one can find adequate assumptions on the drift coefficients that ensure the existence of a continuous density of $Z_T$. For example, it suffices to add to Assumptions~\ref{ass:Reg} and \ref{ass:hormander} that $b_0^*, b_1^* \in \mathcal{C}^{\infty}(\mathbb{R}, \mathbb{R})$ (see \cite{nualart2006malliavin}, \textit{Chapter 2, Theorem 2.3.3, p.133}), or add to Assumption~\ref{ass:Reg} the ellipticity condition on $b_1^* - b_0^*$ (see, \textit{e.g.} \cite{gobet2002lan}). In each of these two cases, the assumptions on the coefficients $b_0^*$ and $b_1^*$ will be too restrictive for our diffusion model, as these functions are already compactly supported. The goal is to keep the assumptions on $b_0^*$ and $b_1^*$ as weak as possible. The first point of Assumption~\ref{ass:hormander} is analogous to H\"ormander's condition considered in \cite{nualart2006malliavin}, \textit{Chapter 2}, in the context of the study of random variables with smooth density functions. The second point, for its part, ensures the existence of a continuous subset of $\mathrm{Supp}(b_0^*) \cup \mathrm{Supp}(b_0^*)$ in which the two coefficients do not intersect. This condition is essential to ensure that $\mathbb{E}_X\left[Z_T^2\right] > 0$, the probability distribution of $X$ being non-atomic. It also provides a non-trivial statistical setting avoiding a reduction to a one-class model through a near-total overlap between the two classes. To this end, we derive the following result.    

\begin{lemme}\label{lm:bounded-density}
Under Assumption~\ref{ass:Reg} and \ref{ass:hormander}, we have $\mathbb{E}_{X}\left[Z_T^2\right] > 0$, and the random variable $Z_T$ has a continuous and bounded density. 
\end{lemme}
The above result is crucial to prove the low-noise conditions provided by Proposition~\ref{prop:margin-condition}, a key result that allows us to derive a rate faster than $N^{-1/2}$ for the excess risk of the plug-in classifier. This result is proven using the Malliavin calculus, relying on Proposition 2.1.1 in \cite{nualart2006malliavin}, \textit{Chapter 2, p.86}. The proof of Lemma~\ref{lm:bounded-density} is provided in the appendix.

In the next section, we give a more explicit description of the Bayes classifier, its empirical counterpart, and the set $\mathcal{G}$ of classification rules.   

\subsection{Classification procedure}
\label{subsec:BayesClassifier}

The Bayes classifier $g^*$ defined in Equation~\ref{eq:bayes1} is the most effective classification rule in the sense of minimizing the risk of misclassification. $g^*$ is given for each $X \in \mathcal{X}$ by $g^*(X) = \mathds{1}_{\Phi^*(X) \geq 1/2}$ where $\Phi^*: \psi \in \mathcal{X} \mapsto \mathbb{P}_{X,Y}(Y=1 | X = \psi)$ is the regression function. Under Assumption~\ref{ass:Reg} and from \cite{denis2020classif}, \textit{Proposition 1}, the following holds:
\begin{equation*}
    \Phi^*(X) = \Phi_{\bf b^*}^*(X) = p_1^*\exp\left(F_{\bf b^*}^1(X)\right)/\left[p_0^*\exp\left(F_{\bf b^*}^0(X)\right) + p_1^*\exp\left(F_{\bf b^*}^1(X)\right)\right],
\end{equation*}
where for each $i \in \mathcal{Y}$,
\begin{equation*}
    F_{\bf b^*}^i(X) := \int_{0}^{T}b_i^*(X_t)dX_t - \dfrac{1}{2}\int_{0}^{T}b_i^{*2}(X_t)dt.
\end{equation*}
The above result is obtained using the Bayes rule combined with Girsanov's theorem. Consequently, the Bayes classifier $g^*$ is fully determined by the unknown elements ${\bf b^*} = (b_0^*, b_1^*)$ and ${\bf p^*} = (p_0^*, p_1^*)$, which makes it a classifier that is computationally untractable in practice. A common strategy is to propose , from a learning sample $\mathcal{Z}^N =  \left\{(\bar{X}^j, Y_j), ~~ j = 1,\ldots, N\right\}$ made up of $N$ independent copies of $(\bar{X}, Y)$ where $\bar{X} = (X_{t_k^n})_{0\leq k \leq n} \in \mathbb{R}^{n+1}$ is a discrete observation of $X$, an empirical classification rule $\widehat{g}$ that mimics the Bayes classifier $g^*$ and given for all $\bar{X} \in \mathbb{R}^{n+1}$ by $\widehat{g}(\bar{X}) = \mathds{1}_{\widehat{\Phi}_{\w{\bf b}}(\bar{X}) \geq 1/2}$, where
\begin{equation*}
    \widehat{\Phi}_{\w{\bf b}}(\bar{X}) := \widehat{p}_1\exp(\bar{F}_{\w{\bf b}}^1(\bar{X}))/\left[\widehat{p}_0\exp(\bar{F}_{\w{\bf b}}^0(\bar{X})) + \widehat{p}_1\exp(\bar{F}_{\w{\bf b}}^1(\bar{X}))\right],
\end{equation*}
and, considering $\xi: [0,T] \mapsto \{t_0^n, t_1^n, \ldots, t_n^n\}$ such that $\xi(t) = t_k^n$ for all $t \in [t_k^n, t_{k+1}^n)$,
\begin{equation}\label{eq:discrete}
    \begin{aligned}
        \bar{F}_{\w{\bf b}}^i(X) = &~ \int_{0}^{T}\w{b}_i(X_{\xi(t)})dX_t - \dfrac{1}{2}\int_{0}^{T}\w{b}_i^{2}(X_{\xi(t)})dt\\
        = &~ \sum_{k=0}^{n-1}\w{b}_i(X_{t_k^n})(X_{t_{k+1}^n} - X_{t_k^n}) - \dfrac{1}{2}\sum_{k=0}^{n-1}(t_{k+1}^n - t_k^n)\w{b}_i^2(X_{t_k^n}),
    \end{aligned}
\end{equation}
and $\bar{F}_{\w{\bf b}}^i(X) = \bar{F}_{\w{\bf b}}^i(\bar{X}) = F_{\w{\bf b}}^i(\bar{X})$. The estimator $\w{\bf p} = \left(\w{p}_0, \w{p}_1\right)$ is built from the learning sample $\mathcal{Z}^N$ as follows:
\begin{equation}\label{eq:estimation-p}
    \w{p}_i := \dfrac{1}{N}\sum_{j=1}^{N}\mathds{1}_{Y_j = i}, ~~ i \in \mathcal{Y}.
\end{equation}
In addition, each drift function $b_i^*$ is estimated from the sample paths $\mathcal{Z}_i^N$ given by Equation~\eqref{eq:Sample}. We give more details on nonparametric estimators of coefficients $b_0^*$ and $b_1^*$ in Section~\ref{sec:main-results}. As we can see, the empirical classifier $\widehat{g}$ is deduced from $g^*$ simply by replacing the unknown elements by their respective estimators constructed from $\mathcal{Z}^{N}$. This strategy is known as the plug-in principle and $\widehat{g}$ is then called a plug-in classifier. Focusing on the set $\mathcal{G}$ of classification rules, we define the following sets.
\begin{align*}
    \mathbf{B}(\beta, R) := &~ \left\{{\bf b^*} = (b_0^*, b_1^*) \in \Sigma(\beta, R) \times \Sigma(\beta, R): b_0^*(x_0) \neq b_1^*(x_0), ~ \mu(\mathfrak{B}_{\bf b^*}) > 0
    \right\},\\
    {\bf P} := &~ \left\{{\bf p^*} = (p_0^*, p_1^*) \in (0,1)^2:
    p_0^* + p_1^* = 1 \right\},\\
    {\bf F}(\beta, R) := &~ {\bf B}(\beta, R) \times {\bf P},
\end{align*}
where $\mathfrak{B}_{\bf b^*} = \left\{x \in \mathbb{R} : b_0^*(x) \neq b_1^*(x)\right\}$. The function space ${\bf B}(\beta, R)$ gathers all possible drift functions $b_0^*$ and $b_1^*$ that satisfy Assumptions~\ref{ass:hormander} and belong to the H\"older class $\Sigma(\beta, R)$. ${\bf P}$ is the set of all possible discrete laws of the label $Y$. We set the condition $(p_0^*, p_1^*) \in (0,1)$ to ensure, for $N$ large enough, that the two classes appear in the learning sample. Then, the set ${\bf F}(\beta, R)$ is the one that contains the real value of the model parameter ${\bf f^*} = ({\bf b^*}, {\bf p^*})$. From now on, in order to adapt our notations to the classification problem, the following new notation is adopted:
\begin{equation*}
    \Phi_{\bf b^*}^* = \Phi_{\bf f^*}, ~~ \widehat{\Phi}_{\w{\bf b}} = \Phi_{\widehat{\bf f}}, ~~ g^* = g_{\bf f^*}, ~~ \widehat{g} = g_{\widehat{\bf f}}.
\end{equation*}
Finally, the set $\mathcal{G}$ of classifiers is given by
\begin{equation*}
    \mathcal{G} := \left\{g_{\bf f}: X \in \mathcal{X} \mapsto \mathds{1}_{\Phi_{\bf f}(X) \geq 1/2} \in \mathcal{Y}, ~ {\bf f} \in {\bf F}(\beta, R) \right\}.
\end{equation*}
In the sequel, we rather rely on the space ${\bf F}(\beta, R)$ as the model parameter ${\bf f^*} = ({\bf b^*}, {\bf p^*})$ fully determines the Bayes classifier $g^* = g_{\bf f^*}$ and satisfies 
\begin{equation*}
    {\bf f}^* \in \underset{{\bf f} \in {\bf F}(\beta, R)}{\arg\min} R(g_{\bf f}). 
\end{equation*}
In the next section, we briefly present the low-noise condition and discuss its effectiveness for our diffusion model.

\subsection{Low-noise conditions}
\label{subsec:low-noise}

The prediction principle for any classifier $g_{\bf f} \in \mathcal{G}$ consists, for any feature $X \in \mathcal{X}$, of returning one of the two classes $0$ and $1$ based on the information provided by $\Phi_{\bf f}(X)$. This information is generally too noisy when the value of $\Phi_{\bf f}(X)$ is in the vicinity of $1/2$, leading to a higher risk of misclassification. In these conditions, the highest possible convergence rate that can be reached by any empirical classification rule is $N^{-1/2}$ (see, \textit{e.g.}, \cite{Yang99}). This convergence rate can be improved when the probability of $\Phi_{\bf f}(X)$ being in the vicinity of $1/2$ is sufficiently small; in this case, we say that we are in low-noise conditions. When these conditions are assumed to hold, it is referred to as a margin assumption on the regression function (see, \textit{e.g.} \cite{audibert2007fast}). The following result was proved in \cite{denis2025empirical}, \textit{Proposition 4.3} in a multiclass setting.

\begin{prop}\label{prop:margin-condition}
    There exists a constant $C > 0$ depending on ${\bf b^*} = (b_0^*, b_1^*) \in {\bf B}(\beta, R)$ such that for all $\varepsilon \in (0, 1/8)$,
\begin{equation*}
    \mathbb{P}_{X}\left(0 < \left|\Phi_{\bf f^*}(X) - \dfrac{1}{2}\right| \leq \varepsilon\right) \leq C\varepsilon.
\end{equation*}
\end{prop}

The result of Proposition~\ref{prop:margin-condition} shows that the probability that $\Phi_{\bf f^*}(X)$ is in the neighborhood of $1/2$ is as small as the vicinity is narrow. This result could not have been deduced from \cite{gadat2020optimal}, \textit{Proposition 1}, as the diffusion model studied is a mixture of Gaussian processes. Dealing with diffusion processes with space-dependent coefficients is more challenging and raises issues about the existence of smooth probability densities for random variables of the form $\int_{0}^{T}\phi(X_t)dW_t$ where $\phi$ is not necessarily an elliptic function or infinitely differentiable. In \cite{denis2025empirical}, the result of Proposition~\ref{prop:margin-condition} was proved under restrictive regularity assumptions on the drift and diffusion coefficients that guarantee the existence of a smooth density function for $Z_T$ based on Theorem 2.3.3 in \cite{nualart2006malliavin}, Chapter 2, p.128. In the present paper, we establish that $Z_T$ admits a continuous and bounded density function under weaker regularity assumptions of the drift coefficients using the Malliavin calculus. Since \cite{denis2025empirical} is a preprint, for completeness, we include a proof using Lemma~\ref{lm:bounded-density} in the appendix. \\
We show in the proof of Theorem~\ref{thm:upper-bound} that
\begin{align*}
    \mathbb{E}_{\mathbb{P}^{\otimes N}}\left[R(g_{\w{\bf f}}) - R(g_{\bf f^*})\right] \leq &~ 2\mathbb{P}_{X}\left(0 < \left|\Phi_{\bf f^*}(X) - \dfrac{1}{2}\right| \leq \varepsilon\right) + \mathbb{E}_{\mathbb{P}^{\otimes N}}\left[\mathbb{P}_{X}\left(\left|\Phi_{\w{\bf f}}(X) - \Phi_{\bf f^*}(X)\right| \geq \varepsilon\right)\right].
\end{align*}
Thus, from Proposition~\ref{prop:margin-condition}, we obtain
\begin{equation}\label{eq:upbound}
    \mathbb{E}_{\mathbb{P}^{\otimes N}}\left[R(g_{\w{\bf f}}) - R(g_{\bf f^*})\right] \leq 2C\varepsilon^2 + \mathbb{E}_{\mathbb{P}^{\otimes N}}\left[\mathbb{P}_{X}\left(\left|\Phi_{\w{\bf f}}(X) - \Phi_{\bf f^*}(X)\right| \geq \varepsilon\right)\right].
\end{equation}
Note that the two terms on the right-hand side of Equation~\eqref{eq:upbound} behave antagonistically as $\varepsilon \rightarrow 0$. As a result, the idea of a  trade-off between the two terms is highlighted. However, the best way out is to prove that the second term on the right-hand side of Equation~\eqref{eq:upbound} converges to zero with a rate faster than a polynomial growth, as described in \cite{audibert2007fast}. In the next section, we establish an exponential inequality for the nonparametric estimator of each drift coefficient $b_i^*, ~ i \in \mathcal{Y}$, leading, for the second term, to a rate faster than $N^{-2\beta/(2\beta+1)}$ while the first term $2C\varepsilon^2$ is of order $N^{-2\beta/(2\beta+1)}$ (up to a logarithmic factor).

Now that the statistical setting has been clearly defined, the next section outlines the main results of the paper. 

\section{Main results}
\label{sec:main-results}

This section is devoted to the establishment of minimax convergence rates for the worst excess risk of the plug-in classifier $\w{g} = g_{\w{\bf f}}$. In Section~\ref{subsec:np-estimation}, we establish an exponential inequality for the nonparametric estimators of the two drift coefficients $b_i^*, ~ i \in \mathcal{Y} = \{0,1\}$. Sections~\ref{subsec:upper-bound} and~\ref{subsec:lower-bound} focus, respectively, on the upper bound and the lower bound of the worst excess risk of $\w{g} = g_{\w{\bf f}}$.   

\subsection{Nonparametric estimation of the drift coefficients}
\label{subsec:np-estimation}

We consider the Nadaraya-Watson estimators of the drift functions $b_0^*$ and $b_1^*$ proposed in \cite{marie2023nadaraya}. These estimators are best suited to the problem considered in this paper. Formally, let $t_0 \in (0,T)$. From subdivision $I_n = \{0=t_0^n, t_1^n, \ldots, t_n^n=T\}$ of the time interval $[0,T]$ with time step $\Delta_n=T/n$, set $k_0 = \lceil t_0/\Delta_n\rceil - 1$ and consider subdivision $\{s_{k_0}^{n}, \ldots, s_n^{n}\}$ of the time interval $[t_0, T]$ where $s_{k_0}^{n} = t_0$ and for all $k \in [\![k_0+1, n]\!], ~ s_{k}^{n} = t_k^n \in I_n$. Then, for each $i \in \mathcal{Y}$ and from the sample paths $\mathcal{Z}_i^N$, the Nadaraya-Watson estimator of $b_i^*$ is given as follows:
\begin{equation}\label{eq:kernel-estimator1}
    \widetilde{b}_{i,N,h_i, h_i^{\prime}}(x) := \dfrac{\widehat{(b\zeta)}_{i,N,h_i}(x)}{\widehat{\zeta}_{i,N,h_i^{\prime}}(x)}, ~~ x \in I_i,
\end{equation}
where $I_i \subset \mathbb{R}$ is the estimation interval,
\begin{equation*}
    \w{\zeta}_{i,N,h_i}(x) := \dfrac{\mathds{1}_{N_i > 1}}{N_i(T-t_0)}\sum_{j=1}^{N_i}\sum_{k=k_0}^{n-1}(s_{k+1}^{n} - s_k^{n})K_{h_i}(X_{s_k^{n}}^{ji} - x)
\end{equation*}
is a kernel estimator of the density function $f_i$ given for all $x \in \mathbb{R}$ by
\begin{align*}
    \zeta_i^*(x) = \dfrac{1}{T-t_0}\int_{t_0}^{T}\Gamma_{X|Y=i}(0, t, x_0, x)dt,
\end{align*}
and 
\begin{equation*}
    \widehat{(b\zeta)}_{i,N,h_i^{\prime}}(x) := \dfrac{\mathds{1}_{N_i > 1}}{N_i(T-t_0)}\sum_{j=1}^{N_i}\sum_{k=k_0}^{n-1}K_{h_i^{\prime}}(X_{s_k^n}^{ji} - x)(X_{s_{k+1}^n}^{ji} - X_{s_k^n}^{ji})
\end{equation*}
is a kernel estimator of $(b\zeta)_i^* = b_i^*\zeta_i^*$, with $h_i, h_i^{\prime} \in (0,1)$ the bandwidths and the function $x \mapsto K_{h_i}(x) = h_i^{-1}K(x/h_{i})$ defined from the kernel $K$. The Nadaraya-Watson estimator $\widetilde{b}_{i,N,h_i,h_i^{\prime}}$ is not necessarily well defined on the real line. Moreover, even if it was, the fact that the denominator $x \mapsto \w{\zeta}_{i,N,h_i^{\prime}}(x)$ vanishes at infinity can be a serious drawback in establishing an upper bound of the estimation risk of $\widetilde{b}_{i,N,h_i,h_i^{\prime}}$ of order $N^{-\beta/(2\beta+1)}$, a result crucial to achieving the objectives outlined in this article, since the resulting rate directly influences the choice of the bandwidth (see \cite{marie2023nadaraya}, \textit{Propositions 1 and 2}). Therefore, note that under Assumption~\ref{ass:Reg} and from \cite{denis2024nonparametric}, Lemma 5 with $\sigma^* = 1$ and $q=3/2$, the transition density $(t,x) \in (0,T] \times \mathbb{R} \mapsto \Gamma_{X|Y=i}(0,t,x_0,x)$ satisfies the following result:
\begin{equation*}
    \Gamma_{X|Y=i}(0,t,x_0,x) \geq \dfrac{1}{\mathfrak{C}\sqrt{t}}\exp\left(-\dfrac{2x^2}{3t}\right) \geq \dfrac{1}{\mathfrak{C}\sqrt{T}}\exp\left(-\dfrac{2x^2}{3t_0}\right), 
\end{equation*}
where the constant $\mathfrak{C} > 1$ is independent of ${\bf b^*} = (b_0^*, b_1^*)$. Then, there exists a constant $m > 0$ that is independent of ${\bf b^*}$ such that for all $i \in \mathcal{Y}$ and for all $x \in I, ~ \zeta_i^*(x) \geq m$. Therefore, we now consider, as in \cite{marie2023nadaraya}, the following Nadaraya-Watson estimator of drift $b_i^*$:
\begin{equation*}
    \widehat{b}_{i,N,h_i, h_i^{\prime}}(x) := \dfrac{\widehat{(b\zeta)}_{i,N,h_i^{\prime}}(x)}{\widehat{\zeta}_{i,N,h_i}(x)}\mathds{1}_{\w{\zeta}_{i,N,h_i}(x) \geq m}, ~~ x \in \mathrm{Supp}(b_i^*) \subset I.
\end{equation*}
This situation justifies the assumption on the respective supports of the drift coefficients $b_0^*$ and $b_1^*$. In addition, we make the following assumptions on the kernel $K$.
\begin{assumption}\label{ass:reg-kernel}
    The kernel $K$ belongs to $\mathbb{L}^2(\mathbb{R}, dx)$, and there exists a constant $C_K > 0$ such that
    \begin{equation*}
        \left|K(x) - K(y)\right| \leq C_K|x-y|, ~~ x,y \in \mathbb{R}.
    \end{equation*}
\end{assumption}
\begin{assumption}\label{ass:prop-kernel}
    There exists $\gamma \in \mathbb{N} \setminus\{0\}$ such that the functions $x \mapsto x^{k}K(x)$, $k \in [\![0, \gamma+1]\!]$ are integrable on $\mathbb{R}$ and satisfy
    \begin{align*}
        \int_{\mathbb{R}}K(x)dx = 1, ~~ \int_{\mathbb{R}}x^{k}K(x)dx = 0, ~~ k \in [\![1, \gamma]\!].
    \end{align*}
\end{assumption}
One can construct multiple kernels $K$ that fully satisfy Assumptions~\ref{ass:reg-kernel} and \ref{ass:prop-kernel}. A classical example of a kernel of order $\ell = \gamma$ is the function $K: \mathbb{R} \rightarrow \mathbb{R}$ given by
\begin{equation*}
    K(x) = \sum_{k=0}^{\gamma}\phi_m(0)\phi_m(x)\mathds{1}_{|x| \leq 1},
\end{equation*}
built from the orthonormal basis $\left\{\phi_m, m \in \mathbb{N}\right\}$ of Legendre polynomials in $\mathbb{L}^2([-1, 1], dx)$ (see \cite{tsybakov2008introduction}, \textit{Chapter 1, Proposition 1.3, p.10} for more details). Moreover, the same kernel also satisfies Assumption~\ref{ass:reg-kernel}, being continuously differentiable and compactly supported. 

\begin{theo}\label{thm:exp-bound-drift}
  Suppose $\Delta_n = \mathcal{O}(N^{-2}), ~ N \rightarrow \infty$ and for each $i \in \mathcal{Y}$, consider a strictly positive sequence $(\delta_{i,N})_{N}$ such that
    \begin{align*}
        &~ \delta_{i,N} \rightarrow 0, ~ N\delta_{i,N}^2 \rightarrow \infty, ~~ h_{i,N}^{\beta}, h_{i,N}^{\prime \beta} \underset{N \rightarrow \infty}{=} o(\delta_{i,N}) ~~ \mathrm{and} ~~ \delta_{i,N} \underset{N \rightarrow \infty}{=} o(\log^{-1}(N)).
    \end{align*}
    Under Assumptions~\ref{ass:Reg}, \ref{ass:reg-kernel} and \ref{ass:prop-kernel} with $\gamma = \left\lfloor \beta\right\rfloor + 1$, for each $i \in \mathcal{Y}$ and conditional on event $\{N_i > 1\}$, the following holds:
    \begin{align*}
       \mathbb{P}_i^{\otimes N}\left(\left\|\widehat{b}_{i,N,h_{i,N},h_{i,N}^{\prime}} - b_i^*\right\|_{\infty} \geq \delta_{i,N}\right) \leq &~ 2\exp\left[-\mathbf{C}N_i\delta_{i,N}^2(h_{i,N} \land h_{i,N}^{\prime})\right] + N_i\exp\left(-\dfrac{(T-t_0)\log^2(N)}{2\|K\|_{\infty}^2}\right)\\
       & + \dfrac{6\|b_i^*\|_{\infty}}{\delta_{i,N}}\exp\left(-\mathbf{C}^{\prime}N_ih_{i,N}\right),
    \end{align*}
    where 
    \begin{equation*}
        \begin{aligned}
            \mathbf{C}:= &~ \min\left\{\dfrac{m^2}{576C_{t_0,T}\left\|b_i^*\right\|_{\infty}^2\|K\|^2}, \dfrac{m^2}{1296C_{t_0,T}^{\prime}\left\|b_i^*\right\|_{\infty}^2\|K\|^2}, \dfrac{m^2}{1296C_{t_0,T}^{\prime\prime}\|K\|^2}, ~ i \in \mathcal{Y}\right\} > 0,\\
            \mathbf{C}^{\prime} := &~ \frac{3m^2}{96C_{t_0,T}\|K\|^2 + 16mT\|K\|_{\infty}} > 0
        \end{aligned}
    \end{equation*}
    and $C_{t_0,T}, C_{t_0,T}^{\prime}, C_{t_0,T}^{\prime\prime} > 0$ are constants depending on $t_0$ and $T$.
\end{theo}
Theorem~\ref{thm:exp-bound-drift} provides an exponential inequality that is essential to establish an upper bound of the excess risk of the plug-in classifier $\w{g}$ that is of the same order as $N^{-2\beta/(2\beta+1)}$ (up to a logarithmic factor). Note that the shape of Nadaraya-Watson estimators of drift coefficients $b_i^*, ~ i \in \mathcal{Y}$ has been crucial for establishing the result of the above theorem, being a ratio of empirical means of independent random variables with a lower bounded denominator by a strictly positive constant independent of $N$ due to the truncation of the estimator. That is why this estimator is viewed as more suitable compared to projection estimators widely studied in the literature (see, \textit{e.g.} \cite{denis2021ridge}, \cite{comte2020nonparametric}). 

\begin{remark}
 For practical situations, there exist numerical methods for the selection of bandwidths $h_{i,N}$ and $h_{i,N}^{\prime}$, and the hyper-parameter $m$ from the learning sample $\mathcal{Z}_N$. Focusing on the selection of bandwidths assuming that $h_{i,N} = h_{i,N}^{\prime}$, we have the Leave-one-out Cross Validation method described in \cite{marie2023nadaraya} and consisting in choosing $\w{h}_{i,N}$ so that
 \begin{equation*}
     \w{h}_{i,N} \in \underset{h \in \mathcal{H}}{\arg\min}{~\mathrm{CV}(h)},
 \end{equation*}
 where $\mathcal{H}$ is a finite subset of the interval $(0,1)$ (for example $\mathcal{H} = \{k/100, ~ k \in [\![1, 20]\!]\}$), and
$$\mathrm{CV}(h) := \sum_{i=1}^{N}\left[\sum_{j=0}^{n-1}\widehat{b}_{N,n,h}^{-i}(X_{t_j}^i)^2(t_{j+1} - t_j) - 2\sum_{j=0}^{n-1}\widehat{b}_{N,n,h}^{-i}(X_{t_j}^i)(X_{t_{j+1}}^i - X_{t_j}^i) \right],$$
with
$$\widehat{b}_{N,n,h}^{-i}(x) := \sum_{k \in [\![1,N]\!] \setminus \{i\}}\sum_{j=0}^{n-1}w_j^k(x)\left(X_{t_{j+1}}^i - X_{t_j}^i\right), ~~ i \in [\![1, N]\!],$$
and 
$$w_j^k(x) := \dfrac{K_h(X_{t_j}^k - x)}{\sum_{i=1}^{N}\sum_{\ell = 0}^{n-1}K_h(X_{t_{\ell}}^i - x)(t_{\ell+1} - t_{\ell})}, ~~ (k,j) \in [\![1, N]\!] \times [\![0,n-1]\!].$$
For the hyperparameter $m$, one can use $\w{m}_{N} = \min\{\w{\zeta}_{0,N,\w{h}_N}(x) \land \w{\zeta}_{1,N,\w{h}_N}(x), ~ x \in I\}$ with, for example, $I = [-1, 1]$ (see \cite{marie2023nadaraya}, \cite{comte2017nonparametric}). 
\end{remark}

The next section is devoted to the establishment of the upper bound on the excess risk of the plug-in classifier under low-noise conditions.

\subsection{Upper bound on the worst excess risk of the plug-in classifier}
\label{subsec:upper-bound}

Once the low-noise condition is established together with the exponential inequality provided by Theorem~\ref{thm:exp-bound-drift}, we derive below an upper bound on the worst excess risk of the plug-in-type classifier $\w{g}$.
\begin{theo}\label{thm:upper-bound}
 Suppose that $\Delta_n = \mathcal{O}(N^{-2})$, $N \rightarrow \infty$ and for each $i \in \mathcal{Y}$, $h_{i,N} = h_{i,N}^{\prime} = N^{-1/(2\beta+1)}$. Under Assumptions~\ref{ass:Reg}, \ref{ass:reg-kernel} and \ref{ass:prop-kernel}, the following holds:
    \begin{equation*}
        \mathbb{E}_{\mathbb{P}^{\otimes N}}\left[R(g_{\w{\bf f}}) - R(g_{\bf f^*})\right] \leq C\log^4(N)N^{-2\beta/(2\beta+1)},
    \end{equation*}
    where $C>0$ is a constant.
\end{theo}
The result of Theorem~\ref{thm:upper-bound} shows that under the low-noise condition, it is indeed possible to construct a more efficient classification procedure. 
In the context of supervised classification for trajectories generated by stochastic differential equations, the above result extends to diffusion processes with space-dependent coefficients, the investigation carried out in \cite{gadat2020optimal} on binary classification for diffusion paths generated by a Gaussian process, solution of the white noise model whose drift depends on the label $Y$, resulting, under low-noise condition, in an optimal rate of order $N^{-2s/(2s+1)}$ over a Sobolev space of smoothness parameter $s > 0$. As we already know, Diffusion models with space-dependent coefficients bring additional non trivial complications compared to the white noise model. The main difficulties related to these models are generally related to the study of transition densities and their estimates, the use of standard norms and the corresponding scalar products, or the existence of smooth probability densities. In this paper, we were able to use standard norms thanks to the compact support of the drift coefficients of the studied mixture model, which is not obvious when considering non-compactly supported drift and diffusion coefficients. In the context of supervised classification for multivariate data of dimension $d$, \cite{audibert2007fast} established, under the strong density assumption, an optimal convergence rate of order $N^{-(1+\alpha)\beta/(2\beta+d)}$ over the H\"older class of smoothness parameter $\beta \geq 1$, where $\alpha > 0$ comes from the Margin Assumption stated as follows:
\begin{equation*}
    P_X\left(0 < \left|\eta(X) - \dfrac{1}{2}\right| \leq t\right) \leq Ct^{\alpha} ~~~ \forall~ t > 0,
\end{equation*}
where $\eta(X) = P(Y=1|X)$ is the regression function. This rate is of order $N^{-2\beta/(2\beta+1)}$ for $\alpha = 1$ and $d = 1$. The result of Proposition~\ref{prop:margin-condition} corresponds to case $\alpha = 1$. Then, the rate provided by Theorem~\ref{thm:upper-bound}
is of the same order as theirs for $d=1$, and faster for $d > 1$.\\
In the next section, we focus on the study of the lower bound on the worst excess risk of the plug-in classifier. 
\subsection{Lower bound on the excess risk of the plug-in classifier}
\label{subsec:lower-bound}
 
We show that under the low-noise condition, the convergence rate of order $N^{-2\beta/(2\beta+1)}$ cannot be improved. To this end, we derive the following result.
\begin{theo}\label{thm:lower-bound}
   There exists a constant $c>0$ such that
\begin{equation*}
    \underset{\w{\bf f}}{\inf}\underset{{\bf f^*} \in {\bf F}(\beta, R)}{\sup}\mathbb{E}_{\mathbb{P}^{\otimes N}}\left[R(g_{\w{\bf f}}) - R(g_{\bf f^*})\right] \geq cN^{-2\beta/(2\beta+1)}.
\end{equation*}
\end{theo}
The above result holds for any time-homogeneous diffusion model, particularly those whose drift coefficients are non-compactly supported. Moreover, this result is not specific to plug-in classifiers, since it holds for any binary supervised classification procedure for time-homogeneous diffusion paths. The result of Theorem\ref{thm:lower-bound} is derived following the principle of Assouad's Lemma adapted to the classification problem established in \cite{Audibert2004ClassificationUP}. The lower bound of the excess risk of the classification procedure is established from a family of possible distributions of the random pair $(X,Y)$ where $X$ is the characteristic and $Y \in \{0,1\}$ the label, the probability distribution of $(X,Y)$ being assumed to be unknown. Then, this method does not rely on a particular classification principle. Note that the lower bound obtained is not of the same order as the upper bound provided by Theorem~\ref{thm:upper-bound}. The extra-factor $\log^4(N)$ on the upper bound of the excess risk of $\w{g}$ is mainly due to the nature of the considered diffusion model together with the nonparametric estimators of the drift functions as described in the previous sections. 

\section{Conclusion}
\label{sec:conclusion}

This paper has tackled the study of minimax convergence rates of a classification procedure of the plug-in type for trajectories generated by time-homogeneous Stochastic Differential Equations with space-dependent coefficients. This problem is first highlighted in \cite{audibert2007fast} in the context of supervised classification of multivariate data, resulting in the establishment of optimal rates in the H\"older class.  The classification model considered in this paper is distinguished by its complexity. In fact, the classification model is based on a mixture of Stochastic Differential Equations with space-dependent coefficients, inducing challenges such as the existence of smooth density functions, the construction of best suited nonparametric estimators of the drift coefficients or the study of exponential inequalities. These difficulties lead to a strong assumption on the support of the drift coefficients and required the diffusion coefficient $\sigma$ to be known. In fact, we do not find in the literature a nonparametric estimator of $\sigma^2$ from i.i.d. SDE paths that is considered best suited for the study carried out in this paper. \\
The immediate perspectives for future investigations are the extension of the present study to a mixture of diffusion processes whose space-dependent drift and diffusion coefficients are non-compactly supported and unknown. This new extension will require the construction of nonparametric estimators of the  drift and diffusion coefficients that are adapted to the study of an upper bound on the excess risk under low-noise conditions, implying the establishment of an exponential inequality. In fact, projection estimators of the drift function from i.i.d SDE paths on the whole line $\mathbb{R}$ proposed, for instance, in \cite{comte2020nonparametric}, or the projection estimators of the square of the diffusion coefficient from i.i.d SDE paths studied in \cite{ella2024nonparametric} and \cite{ella2025minimax} seem not to be well suited to solve this problem. As a result, one should consider nonparametric estimators of the kNN-type or the kernel-type such as the Nadaraya-Watson estimator from repeated observations of a time-homogeneous diffusion process (see \cite{marie2023nadaraya}). We can also think of extending the study to a mixture of time-inhomogeneous diffusion processes. However, this new model brings an additional complication as the coefficients of the diffusion process depend both on the space and the time. To be more precise, one can notice that the construction of the Nadaraya-Watson estimator $\w{b}_{i,N,h}$ of the drift coefficient $b_i^*$ was built from portions $X^{t_0, j} = (X_t^j)_{t_0 \leq t \leq T}, ~ j \in [\![1, N]\!]$ of the $N$ independent copies of the solution $X$ of model~\eqref{eq:classif-model}. The estimators considered would no longer be suitable if the classification model was a mixture of time-inhomogeneous diffusion processes.

\section{Proofs}
\label{sec:proofs}

\subsection{Proof of Theorem~\ref{thm:exp-bound-drift}}

The proof of Theorem~\ref{thm:exp-bound-drift} relies on the following lemmas.

\begin{lemme}\label{lm:bias-1}
Under Assumptions~\ref{ass:Reg}, \ref{ass:reg-kernel} and \ref{ass:prop-kernel}, there exists a constant $C>0$ depending on $\beta, t_0, T$ and $K$ such that for all $j \in [\![1,N]\!]$ and for all $k \in [\![0,n-1]\!]$,
    \begin{equation*}
        \left|\mathbb{E}_{\mathbb{P}_i^{\otimes N}}\left[K_{h_{i,N}}(X_{t_k}^{ji} - x)\right] - \Gamma_{X|Y=i}(0,t_k,x_0,x)\right| \leq Ch_{i,N}^{\beta}, ~ i \in \mathcal{Y}, ~ x \in \mathbb{R}.
    \end{equation*}
\end{lemme}

\begin{lemme}\label{lm:bias-2}
Under Assumptions~\ref{ass:Reg}, \ref{ass:reg-kernel} and \ref{ass:prop-kernel}, there exists a constant $C>0$ depending on $\beta, t_0, T, \bf{b}^*$ and $K$ such that for all $j \in [\![1,N]\!]$,
    \begin{equation*}
        \left|\sum_{k=0}^{n-1}\mathbb{E}_{\mathbb{P}_i^{\otimes N}}\left[K_{h_{i,N}^{\prime}}(X_{t_k}^{ji} - x)\int_{t_k}^{t_{k+1}}b_i^*(X_{s}^{ji})ds\right] - (b^*\zeta)_{i,\Delta_n}(x)\right| \leq C\left(h_{i,N}^{\prime\beta} + \sqrt{\Delta_n}\right), ~~ i \in \mathcal{Y}, ~ x \in \mathbb{R}.
    \end{equation*}
\end{lemme}
Lemma~\ref{lm:bias-1} and Lemma~\ref{lm:bias-2} give the upper bounds on the bias terms for the respective kernel estimators $\w{\zeta}_{i,N,h_{i,N}}$ and $\w{(b\zeta)}_{i,N,h_{i,N}}$. The proofs are provided in the appendix.

\begin{lemme}[\textbf{Bernstein's inequality}]\label{lm:bernstein}
    Let $X_1, \ldots, X_n$ be independent and square integrable random variables such that for some nonnegative constants $v$ and $b$, we have $\mathrm{Var}(X_j) \leq v$ and $|X_j| \leq b$ almost surely for all $j \in [\![1,N]\!]$. The following holds: 
    \begin{equation*}
        \mathbb{P}\left(\sum_{j=1}^{N}\left(X_j - \mathbb{E}\left[X_i\right]\right) \geq x\right) \leq \exp\left(-\dfrac{x^2}{2Nv + \frac{2}{3}bx}\right).
    \end{equation*}
\end{lemme}
The result of Lemma~\ref{lm:bernstein} is established in \cite{massart2007concentration}, Proposition 2.8 and Equation (2.16), p.23-24. 

\begin{proof}[Proof of Theorem~\ref{thm:exp-bound-drift}]
    Fix $i \in \mathcal{Y}$ and consider any drift function $b_i^* \in \Sigma(\beta, R)$ that satisfies Assumption~\ref{ass:Reg}. From \cite{marie2023nadaraya}, \textit{proof of Proposition 3}, we have for each $x \in \mathrm{Supp}(b_i^*)$,
     \begin{multline*}
        (\widehat{b}_{i,N,h_{i,N},h_{i,N}^{\prime}} - b_i^*)(x)\\
        = \left[\left(\dfrac{\widehat{(b\zeta)}_{i,N,h_{i,N}^{\prime}} - (b\zeta)_i^*}{\widehat{\zeta}_{i,N,h_{i,N}}}\right)(x) - \left(\dfrac{1}{\widehat{\zeta}_{i,N,h_{i,N}}} - \dfrac{1}{\zeta_i^*}\right)(x)(b\zeta)_i^*(x)\right]\mathds{1}_{\widehat{\zeta}_{i,N,h_{i,N}}(x) \geq m/2}\\
        - b_i^*\mathds{1}_{\left|\zeta_i^*(x) - \widehat{\zeta}_{i,N,h_{i,N}}(x)\right| \geq m/2},
     \end{multline*}
     where $m > 0$ is a lower bound of the density functions $f_0$ and $f_1$ on the compact interval $I \subset \mathbb{R}$, and $(b\zeta)_i^* = b_i^* \zeta_i^*$, which implies that
     \begin{equation*}\label{eq:T1i-8}
         \begin{aligned}
             \left\|\widehat{b}_{i,N,h_{i,N},h_{i,N}^{\prime}} - b_i^*\right\|_{\infty} \leq &~ \dfrac{2}{m}\left\|\widehat{(b\zeta)}_{i,N,h_{i,N}^{\prime}} - (b\zeta)_i^*\right\|_{\infty} + \dfrac{2\left\|b_i^*\right\|_{\infty}}{m}\left\|\widehat{\zeta}_{i,N,h_{i,N}} - \zeta_i^*\right\|_{\infty}\\
             & + \left\|b_i^*\right\|_{\infty}\mathds{1}_{\left\|\zeta_i^* - \widehat{\zeta}_{i,N,h_{i,N}}\right\|_{\infty} \geq m/2}.
         \end{aligned}
     \end{equation*}
     Then, under Assumption~\ref{ass:Reg} and conditional on the event $\{N_i > 1\}$, we have
     \begin{multline*}
          \mathbb{P}_i^{\otimes N}\left(\left\|\widehat{b}_{i,N,h_{i,N},h_{i,N}^{\prime}} - b_i^*\right\|_{\infty} \geq \delta_{i,N}\right) \leq \mathbb{P}_i^{\otimes N}\left(\left\|\widehat{(b\zeta)}_{i,N,h_{i,N}^{\prime}} - (b\zeta)_i^*\right\|_{\infty} \geq \dfrac{m\delta_{i,N}}{6}\right)\\
           + \mathbb{P}_i^{\otimes N}\left(\left\|\widehat{\zeta}_{i,N,h_{i,N}} - \zeta_i^*\right\|_{\infty} \geq \dfrac{m\delta_{i,N}}{6\|b_i^*\|_{\infty}}\right) + \mathbb{P}_i^{\otimes N}\left(\mathds{1}_{\left\|\zeta_i^* - \widehat{\zeta}_{i,N,h_{i,N}}\right\|_{\infty} \geq m/2} \geq \dfrac{\delta_{i,N}}{3\left\|b_i^*\right\|_{\infty}}\right).
     \end{multline*}
     From the Markov inequality, we have
     \begin{equation*}
         \mathbb{P}_i^{\otimes N}\left(\mathds{1}_{\left\|\zeta_i^* - \widehat{\zeta}_{i,N,h_{i,N}}\right\|_{\infty} \geq m/2} \geq \dfrac{\delta_{i,N}}{3\left\|b_i^*\right\|_{\infty}}\right) \leq \dfrac{3\|b_i^*\|_{\infty}}{\delta_{i,N}}\mathbb{P}_i^{\otimes N}\left(\left\|\zeta_i^* - \widehat{\zeta}_{i,N,h_{i,N}}\right\|_{\infty} \geq \dfrac{m}{2}\right).
     \end{equation*}
     We deduce that
     \begin{multline}\label{eq:T1i-9}
              \mathbb{P}_i^{\otimes N}\left(\left\|\widehat{b}_{i,N,h_{i,N}, h_{i,N}^{\prime}} - b_i^*\right\|_{\infty} \geq \delta_{i,N}\right) \leq \dfrac{3\|b_i^*\|_{\infty}}{\delta_{i,N}}\mathbb{P}_i^{\otimes N}\left(\left\|\widehat{\zeta}_{i,N,h_{i,N}} - \zeta_i^*\right\|_{\infty} \geq \dfrac{m}{2}\right)\\
              + \mathbb{P}_i^{\otimes N}\left(\left\|\widehat{(b\zeta)}_{i,N,h_{i,N}^{\prime}} - (b\zeta)_i^*\right\|_{\infty} \geq \dfrac{m\delta_{i,N}}{6}\right) + \mathbb{P}_i^{\otimes N}\left(\left\|\widehat{\zeta}_{i,N,h_{i,N}} - \zeta_i^*\right\|_{\infty} \geq \dfrac{m\delta_{i,N}}{6\|b_i^*\|_{\infty}}\right).
     \end{multline}
    For all $s \in [t_0,T]$, set $\eta(s) = s_k^n$ for all $s \in [s_k^n, s_{k+1}^n)$ and $\zeta_{i,\Delta_n}(x) = \sum_{k=k_0}^{n-1}\frac{s_{k+1}^n - s_k^n}{T-t_0}\Gamma_{X|Y=i}(0,s_k^n,x_0,x)$. Then, for all $x \in \mathrm{Supp}(b_i^*)$, we have 
    \begin{equation*}
        \begin{aligned}
            \left|\zeta_{i,\Delta_n}(x) - \zeta_i^*(x)\right| \leq &~ \dfrac{1}{T-t_0}\int_{t_0}^{T}\left|\Gamma_{X|Y=i}(0,\eta(t), x_0, x) - \Gamma_{X|Y=i}(0,t,x_0,x)\right|dt.  
        \end{aligned}
    \end{equation*}
    Since, for any $x \in I$, the function $t \mapsto \Gamma_{X|Y=i}(0,t,x_0,x)$ is Lipschitz on $[t_0, T]$, there exists a constant $C>0$ depending on $t_0$ and $T$ such that for all $x \in I$,
    \begin{equation}\label{eq:f-error}
        \left|\zeta_{i,\Delta_n}(x) - \zeta_i^*(x)\right| \leq \dfrac{C}{T-t_0}\int_{t_0}^{T}|\eta(t) - t|dt \leq C\Delta_n.
    \end{equation}
    Conditional on $\{N_i > 1\}$, we have the following.
    \begin{multline*}
        \left|\widehat{\zeta}_{i,N,h_{i,N}}(x) - \zeta_{i}^*(x)\right| = \left|\widehat{\zeta}_{i,N,h_{i,N}}(x) - \zeta_{i,\Delta_n}(x)\right| + \left|\zeta_{i,\Delta_n}(x) - \zeta_i^*(x)\right|\\
        \leq \dfrac{1}{N_i(T-t_0)}\left|\sum_{j=1}^{N_i}\sum_{k=k_0}^{n-1}\left(s_{k+1}^n - s_k^n\right)\left(K_{h_{i,N}}(X_{s_k^n}^{ji} - x) - \Gamma_{X|Y=i}(0, s_k^n, x_0, x)\right)\right| + C\Delta_n\\
        \leq \underset{x \in \mathrm{Supp}(b_i^*)}{\sup}\left|\dfrac{1}{N_i(T-t_0)}\sum_{j=1}^{N_i}\sum_{k=k_0}^{n-1}\left(s_{k+1}^n - s_k^n\right)\left[K_{h_{i,N}}(X_{t_k}^{ji} - x) - \Gamma_{X|Y=i}(0,t_k,x_0,x)\right]\right| + C\Delta_n.
    \end{multline*}
    Since the function $x \mapsto \dfrac{1}{N_i(T-t_0)}\sum_{j=1}^{N_i}\sum_{k=k_0}^{n-1}(s_{k+1}^n - s_k^n)\left[K_{h_{i,N}}(X_{s_k^n}^{ji} - x) - \Gamma_{X|Y=i}(0,s_k^n,x_0,x)\right]$ is continuous, there exists $x^* \in \mathrm{Supp}(b_i^*)$ such that
    \begin{multline*}
        \underset{x \in \mathrm{Supp}(b_i^*)}{\sup}\left|\dfrac{1}{N_i(T-t_0)}\sum_{j=1}^{N_i}\sum_{k=k_0}^{n-1}\left(s_{k+1}^n - s_k^n\right)\left[K_{h_{i,N}}(X_{s_k^n}^{ji} - x) - \Gamma_{X|Y=i}(0,s_k^n,x_0,x)\right]\right|\\
        = \left|\dfrac{1}{N_i(T-t_0)}\sum_{j=1}^{N_i}\sum_{k=k_0}^{n-1}\left(s_{k+1}^n - s_k^n\right)\left[K_{h_{i,N}}(X_{s_k^n}^{ji} - x^*) - \Gamma_{X|Y=i}(0,s_k^n,x_0,x^*)\right]\right|.
    \end{multline*}
    We deduce that conditional on $\{N_i > 1\}$,
    \begin{equation*}
        \begin{aligned}
            \left\|\widehat{\zeta}_{i,N,h_N} - \zeta_i^*\right\|_{\infty} \leq &~ \left|\dfrac{1}{N_i(T-t_0)}\sum_{j=1}^{N_i}\sum_{k=k_0}^{n-1}\left(s_{k+1}^n - s_k^n\right)V_{i,j,k}^{\zeta}(x^*)\right|\\
            & + \dfrac{1}{N_i(T-t_0)}\sum_{j=1}^{N_i}\sum_{k=k_0}^{n-1}\left(s_{k+1}^n - s_k^n\right)\left|B_{i,j,k}^{\zeta}(x^*)\right| + C\Delta_n,
        \end{aligned}
    \end{equation*}
    where for all $x \in \mathbb{R}$,
    \begin{equation}\label{eq:V-f}
        \begin{aligned}
            V_{i,j,k}^{\zeta}(x) := &~ K_{h_{i,N}}(X_{s_k^n}^{ji} - x) - \mathbb{E}_{\mathbb{P}_i^{\otimes N}}\left[K_{h_{i,N}}(X_{s_k^n}^{ji} - x)\right], ~~ (i,j,k) \in \mathcal{Y} \times [\![1,N_i]\!] \times [\![k_0,n-1]\!]\\ 
            B_{i,j,k}^{\zeta}(x) := &~ \mathbb{E}_{\mathbb{P}_i^{\otimes N}}\left[K_{h_{i,N}}(X_{s_k^n}^{ji} - x)\right] - \Gamma_{X|Y=i}(0,s_k^n,x_0,x), ~~ (i,j,k) \in \mathcal{Y} \times [\![1,N_i]\!] \times [\![k_0,n-1]\!].
        \end{aligned}
    \end{equation}
    By Lemma~\ref{lm:bias-1} and the assumptions therein, we have $\left|B_{i,j,k}^{\zeta}(x^*)\right| = \mathcal{O}(h_{i,N}^{\beta})$ and we obtain the following:
    \begin{equation}\label{eq:T1i-10bis}
        \left\|\widehat{\zeta}_{i,N,h_{i,N}} - \zeta_i^*\right\|_{\infty} \leq \left|\dfrac{1}{N_i(T-t_0)}\sum_{j=1}^{N_i}\sum_{k=k_0}^{n-1}\left(s_{k+1}^n - s_k^n\right)V_{i,j,k}^{\zeta}(x)\right| + C(h_{i,N}^{\beta} + \Delta_n).
    \end{equation}
    For all $x \in \mathrm{Supp}(b_i^*)$, in event $\{N_i > 1\}$ and by Equation~\eqref{eq:f-error}, $\left|(b\zeta)_{i,\Delta_n}(x) - (b\zeta)_i(x)\right| = \mathcal{O}(\Delta_n)$ and
    \begin{multline*}
        \left|\widehat{(b\zeta)}_{i,N,h_{i,N}^{\prime}}(x) - (b\zeta)_i^*(x)\right| \leq \left|\dfrac{1}{N_i(T-t_0)}\sum_{j=1}^{N_i}\sum_{k=k_0}^{n-1}K_{h_{i,N}^{\prime}}(X_{s_k^n}^{ji} - x)(X_{s_{k+1}^n}^{ji} - X_{s_k^n}^{ji}) - (b\zeta)_i^*(x)\right|\\
        \leq \dfrac{1}{N_i(T-t_0)}\left(\left|\sum_{j=1}^{N_i}\sum_{k=k_0}^{n-1}V_{i,j,k}^{b,\zeta}(x)\right| + \left|\sum_{j=1}^{N_i}B_{i,j}^{b,\zeta}(x)\right| + \left|\sum_{j=1}^{N_i}\sum_{k=k_0}^{n-1}\Psi_{i,j,k}(x)\right|\right) + C\Delta_n,
    \end{multline*}
    where $C>0$ is a constant and for each $i \in \mathcal{Y}$ and for all $x \in \mathrm{Supp}(b_i^*)$ and $(j,k) \in [\![1,N_i]\!] \times [\![k_0,n-1]\!]$,
    \begin{equation}\label{eq:V-bf}
        \begin{aligned}
            V_{i,j,k}^{b,\zeta}(x) := &~ K_{h_{i,N}^{\prime}}(X_{s_k^n}^{ji} - x)\int_{s_k^n}^{s_{k+1}^n}b_i^*(X_{u}^{ji})du - \mathbb{E}_{\mathbb{P}_i^{\otimes N}}\left[K_{h_{i,N}^{\prime}}(X_{s_k^n}^{ji} - x)\int_{s_k^n}^{s_{k+1}^n}b_i^*(X_{u}^{ji})du\right],\\ 
            B_{i,j}^{b,\zeta}(x) := &~ \sum_{k=k_0}^{n-1}\mathbb{E}_{\mathbb{P}_i^{\otimes N}}\left[K_{h_{i,N}^{\prime}}(X_{s_k^n}^{ji} - x)\int_{s_k^n}^{s_{k+1}^n}b_i^*(X_{u}^{ji})du\right] - (T-t_0)(b\zeta)_{i,\Delta_n}(x),\\
            \Psi_{i,j,k}(x) := &~ K_{h_{i,N}^{\prime}}(X_{s_k^n}^{ji} - x)(W_{s_{k+1}^n}^j - W_{s_k^n}^j).
        \end{aligned}
    \end{equation}
    By Lemma~\ref{lm:bias-2}, for all $i \in \mathcal{Y}$ and $x \in \mathrm{Supp}(b_i^*), ~ \left|B_{i,j}^{b,\zeta}(x)\right| = \mathcal{O}\left((h_{i,N}^{\prime \beta} + \sqrt{\Delta_n}\right)$, and since, under Assumptions~\ref{ass:Reg} and \ref{ass:reg-kernel}, the functions $x \mapsto V_{i,j,k}^{b,\zeta}(x)$ and $x \mapsto \sum_{j=1}^{N_i}\sum_{k=k_0}^{n-1}\Psi_{i,j,k}(x)$ are continuous, there exist $y^* \in \mathrm{Supp}(b_i^*)$ and a constant $C>0$ such that
    \begin{multline}\label{eq:T1i-14}
        \left\|\widehat{(b\zeta)}_{i,N,h_{i,N}^{\prime}} - (b\zeta)_i^*\right\|_{\infty}\\
        \leq \dfrac{1}{N_i(T-t_0)}\left(\left|\sum_{j=1}^{N_i}\sum_{k=k_0}^{n-1}V_{i,j,k}^{b,\zeta}(y^*)\right| + \left|\sum_{j=1}^{N_i}\sum_{k=k_0}^{n-1}\Psi_{i,j,k}(y^*)\right|\right) + C\left(h_{i,N}^{\prime \beta} + \sqrt{\Delta_n}\right)  
    \end{multline}
    From Equations~\eqref{eq:T1i-14}, \eqref{eq:T1i-10bis}, \eqref{eq:T1i-9} since $\Delta_n = \mathcal{O}(N^{-2})$, $\max(h_{i,N}^{\beta}, h_{i,N}^{\prime\beta}) = o(\delta_{i,N})$, $N\delta_{i,N}^2 \rightarrow \infty$ as $N \rightarrow \infty$ and $\delta_{i,N} = o(\log^{-1}(N))$, for $N$ large enough, we obtain:
    \begin{equation}\label{eq:T1i-15}
    \begin{aligned}
        \mathbb{P}_i^{\otimes N}\left(\left\|\widehat{b}_{i,N,h_{i,N}, h_{i,N}^{\prime}} - b_i^*\right\|_{\infty} \geq \delta_{i,N}\right) \leq &~ \Lambda_{i1} + \Lambda_{i2} + \Lambda_{i3} + \dfrac{3\|b_i^*\|_{\infty}}{\delta_{i,N}}\Lambda_{i4},
    \end{aligned}
    \end{equation}
    where
    \begin{equation}\label{eq:lambda}
        \begin{aligned}
            \Lambda_{i1} = &~ \mathbb{P}_i^{\otimes N}\left(\left|\sum_{j=1}^{N_i}\sum_{k=k_0}^{n-1}(s_{k+1}^n - s_k^n) V_{i,j,k}^{\zeta}(x^*)\right| \geq \dfrac{mN_i\delta_N(T-t_0)}{12\|b_i^*\|_{\infty}}\right),\\
            \Lambda_{i2} = &~ \mathbb{P}_i^{\otimes N}\left(\left|\sum_{j=1}^{N_i}\sum_{k=k_0}^{n-1}V_{i,j,k}^{b,\zeta}(y^*)\right| \geq \dfrac{mN_i\delta_N(T-t_0)}{18\|b_i^*\|_{\infty}}\right),\\
            \Lambda_{i3} = &~ \mathbb{P}_i^{\otimes N}\left(\left|\sum_{j=1}^{N_i}\sum_{k=k_0}^{n-1}\Psi_{i,j,k}(y^*)\right| \geq \dfrac{mN_i\delta_N(T-t_0)}{18}\right),\\
            \Lambda_{i4} = &~ \mathbb{P}_i^{\otimes N}\left(\left|\sum_{j=1}^{N_i}\sum_{k=k_0}^{n-1}(s_{k+1}^n - s_k^n) V_{i,j,k}^{\zeta}(x^*)\right| \geq \dfrac{mN_i(T-t_0)}{4}\right).
        \end{aligned}
    \end{equation}

\paragraph{Focus on $\Lambda_{i1}$ and $\Lambda_{i4}$.}

From Equation~\eqref{eq:V-f} and for all $(j,k) \in [\![1,N_i]\!] \times [\![k_0, n-1]\!]$, conditional on $\{N_i > 1\}$, $\mathbb{E}_{\mathbb{P}_i^{\otimes N}}\left[V_{i,j,k}^{\zeta}(x^*)\right] = 0$. Under Assumption~\ref{ass:reg-kernel} using Proposition 1.2 in \cite{gobet2002lan}, there exists a constant $C_{t_0, T}>0$ depending on $t_0$ and $T$ such that
\begin{equation*}
    \begin{aligned}
        \mathrm{Var}_{\mathbb{P}_i^{\otimes N}}\left[\sum_{k=k_0}^{n-1}\left(s_{k+1}^n - s_k^n\right)V_{i,j,k}^{\zeta}(x^*)\right] \leq &~ n\Delta_n^2\sum_{k=k_0}^{n-1}\mathbb{E}_{\mathbb{P}_i^{\otimes N}}\left[K_{h_{i,N}}^2(X_{s_k^n}^{ji} - x^*)\right]\\
       = &~ n\Delta_n^2\sum_{k=k_0}^{n-1}\int_{\mathbb{R}}\dfrac{1}{h_{i,N}^2}K^2\left(\dfrac{z-x^*}{h_{i,N}}\right)\Gamma_{X|Y=i}(0, s_k^n, x_0, z)dz\\
       = &~ n\Delta_n^2\sum_{k=k_0}^{n-1}\int_{\mathbb{R}}\dfrac{1}{h_{i,N}}K^2(z)\Gamma_{X|Y=i}(0, s_k^n, x_0, x_0 + h_{i,N}z)dz\\
       \leq &~ C_{t_0,T}h_{i,N}^{-1}\|K\|^2 =: v. 
    \end{aligned}
\end{equation*}
Moreover, for all $j \in \{1, \ldots, N_i\}, ~ \left|\sum_{k=k_0}^{n-1}(s_{k+1}^n - s_k^n)V_{i,j,k}^{\zeta}(x^*)\right| \leq 2Th_{i,N}^{-1}\|K\|_{\infty} =: b$. Then, since $\delta_{i,N} \rightarrow 0$ as $N \rightarrow \infty$, from Equation~\eqref{eq:lambda}, Lemma~\ref{lm:bernstein} and for $N$ large enough, conditional on $\{N_i > 1\}$, we obtain: 
\begin{equation}\label{eq:bound-lambda-1}
\begin{aligned}
     \Lambda_{i1} \leq &~ 2\exp\left(-\dfrac{\left(\frac{mN_i\delta_{i,N}(T-t_0)}{12\|b_i^*\|_{\infty}}\right)^2}{2N_iC_{t_0,T}h_{i,N}^{-1}\|K\|^2 + \frac{4T\|K\|_{\infty}mN_i\delta_Nh_{i,N}^{-1}}{36\|b_i^*\|_{\infty}}}\right) \leq 2\exp\left(-\dfrac{m^2N_i\delta_{i,N}^2h_{i,N}}{576C_{t_0,T}\left\|b_i^*\right\|_{\infty}^2\|K\|^2}\right),\\ \Lambda_{i4} \leq &~ 2\exp\left(-\dfrac{\left(\frac{mN_i(T-t_0)}{4}\right)^2}{2N_iC_{t_0,T}h_{i,N}^{-1}\|K\|^2 + \frac{T\|K\|_{\infty}mh_{i,N}^{-1}N_i}{3}}\right) \leq 2\exp\left(-\dfrac{3m^2N_ih_{i,N}}{96C_{t_0,T}\|K\|^2 + 16mT\|K\|_{\infty}}\right).
\end{aligned}
\end{equation}

\paragraph{Focus on $\Lambda_{i2}$.}

In event $\{N_i > 1\}$ and from Equation~\eqref{eq:V-bf}, we obtain for all $j \in \{1, \ldots, N_i\}$, $\mathbb{E}_{\mathbb{P}_i^{\otimes N}}\left[\sum_{k=k_0}^{n-1}V_{i,j,k}^{b,f}(y^*)\right] = 0$. From Assumptions~\ref{ass:Reg} and \ref{ass:reg-kernel} and using Cauchy-Schwarz's inequality together with Proposition 1.2 in \cite{gobet2002lan}
\begin{equation*}
    \begin{aligned}
        \mathrm{Var}_{\mathbb{P}_i^{\otimes N}}\left[\sum_{k=k_0}^{n-1}V_{i,j,k}^{b,\zeta}(y^*)\right] \leq &~ \mathbb{E}_{\mathbb{P}_i^{\otimes N}}\left[\left(\sum_{k=k_0}^{n-1}K_{h_{i,N}^{\prime}}(X_{s_k^n}^{ji}-x)\int_{s_k^n}^{s_{k+1}^n}b_i^*(X_u^{ji})du\right)^2\right]\\
        \leq &~ n\Delta_n^2\|b_i^*\|_{\infty}^2\sum_{k=k_0}^{n-1}\mathbb{E}_{\mathbb{P}_i^{\otimes N}}\left[K_{h_{i,N}^{\prime}}^2(X_{s_k^n}^{ji} - x)\right]\\
        = &~ n\Delta_n^2\|b_i^*\|_{\infty}^2\sum_{k=k_0}^{n-1}\int_{\mathbb{R}}\dfrac{1}{h_{i,N}^{\prime 2}}K^2\left(\dfrac{z - y^*}{h_{i,N}^{\prime}}\right)\Gamma_{X|Y=i}(0,s_k^n,x_0,z)dz\\
        = &~ n\Delta_n^2\sum_{k=k_0}^{n-1}\int_{\mathbb{R}}\dfrac{1}{h_{i,N}^{\prime}}K^2(z)b_i^{*2}(z)\Gamma_{X|Y=i}(0,s_k^n,x_0,x_0+h_{i,N}^{\prime}z)dz\\
        \leq &~ C_{t_0,T}^{\prime}h_{i,N}^{\prime -1}\|K\|^2\|b_i^*\|_{\infty}^2 =: v,
    \end{aligned}
\end{equation*}
where $C_{t_0,T}^{\prime} > 0$ is a constant depending on $t_0$ and $T$. In addition, for all $j \in \{1, \ldots, N_i\}$, we have $\left|\sum_{k=k_0}^{n-1}V_{i,j,k}^{b,\zeta}(y^*)\right| \leq 2Th_{i,N}^{\prime -1}\left\|b_i^*\right\|_{\infty}\|K\|_{\infty} =: b$. For $N$ large enough, then $\delta_{i,N}$ close enough to 0 and applying Lemma~\ref{lm:bernstein}, we obtain from Equation~\eqref{eq:lambda},
\begin{equation}\label{eq:bound-lambda-2}
    \begin{aligned}
        \Lambda_{i2} \leq &~ \exp\left(-\dfrac{\left(\frac{mN_i\delta_{i,N}(T-t_0)}{18\|b_i^*\|_{\infty}}\right)^2}{2N_iC_{t_0,T}^{\prime}h_{i,N}^{\prime -1}\|K\|^2\|b_i^*\|_{\infty}^2 + \frac{4T\|K\|_{\infty}mN_i\delta_{i,N}h_N^{\prime -1}}{54}}\right)\\
        \leq &~ 2\exp\left(-\dfrac{m^2N_i\delta_{i,N}^2h_{i,N}^{\prime}}{1296C_{t_0,T}^{\prime}\left\|b_i^*\right\|_{\infty}^2\left\|K\right\|^2}\right).
    \end{aligned}
\end{equation}

\paragraph{Focus on $\Lambda_{i3}$.}

Fix $i \in \mathcal{Y}$ and set 
$$\mathcal{A}_i := \left\{\underset{j \in \{1, \ldots, N_i\}}{\sup}{\left|\sum_{k=k_0}^{n-1}\Psi_{i,j,k}(y^*)\right|} \leq (T-t_0)h_{i,N}^{\prime -1}\log(N)\right\}.$$ 
From Equation~\eqref{eq:lambda}, 
\begin{equation*}
    \Lambda_{i3} \leq \mathbb{P}_i^{\otimes N}\left(\left|\sum_{j=1}^{N_i}\sum_{k=k_0}^{n-1}\Psi_{i,j,k}(y^*)\right| \geq \dfrac{mN_i\delta_{i,N}(T-t_0)}{18} \biggm\vert \mathcal{A}_i\right) + \mathbb{P}_i^{\otimes N}\left(\mathcal{A}_i^{c}\right),
\end{equation*}
where $\mathcal{A}_i^{c}$ is the complementary of the random event $\mathcal{A}_i$. We have $\mathbb{E}_{\mathbb{P}_i^{\otimes N}}[\sum_{k=k_0}^{n-1}\Psi_{i,j,k}(y^*)|\mathcal{A}_i] = 0$ and on event $\mathcal{A}_i$, $\left|\sum_{k=k_0}^{n-1}\Psi_{i,j,k}(y^*)\right| \leq (T-t_0)h_{i,N}^{\prime -1}\log(N) =: b$. There exists a constant $C_{t_0,T}^{\prime\prime} > 0$ depending on $t_0$ such that
\begin{align*}
    \mathrm{Var}_{\mathbb{P}_i^{\otimes N}}\left[\sum_{k=k_0}^{n-1}\Psi_{i,j,k}(y^*) \biggm\vert \mathcal{A}_i\right] = &~ \mathbb{E}_{\mathbb{P}_i^{\otimes N}}\left[\sum_{k=k_0}^{n-1}K_{h_{i,N}^{\prime}}^2(X_{s_k^n}^{ji} - x)(W_{s_{k+1}^n}^j - W_{s_k^n}^j)^2\biggm\vert \mathcal{A}_i\right]\\
    \leq &~ \Delta_n\sum_{k=k_0}^{n-1}\mathbb{E}_{\mathbb{P}_i^{\otimes N}}\left[K_{h_{i,N}^{\prime}}^2(X_{s_k^n}^{ji} - x) \biggm\vert \mathcal{A}_i\right]\\
    \leq &~ \Delta_n\sum_{k=k_0}^{n-1}\int_{\mathbb{R}}\dfrac{1}{h_{i,N}^{\prime 2}}K^2\left(\dfrac{z-x}{h_{i,N}^{\prime}}\right)\Gamma_{X|Y=i}(0, s_k^n, x_0, z)dz\\
    \leq &~ C_{t_0,T}^{\prime\prime}h_{i,N}^{\prime -1}\|K\|^2 =: v,
\end{align*}
Then, since $\delta_{i,N} \underset{N \rightarrow \infty}{=} o(\log^{-1}(N))$, applying Lemma~\ref{lm:bernstein} with $N$ large enough, we obtain:
\begin{equation}\label{eq:bound-Lambda-31}
    \begin{aligned}
        \Lambda_{i3} \leq &~ \exp\left(-\dfrac{\left(\frac{mN_i\delta_N(T-t_0)}{18}\right)^2}{2N_iC_{t_0,T}^{\prime\prime}h_{i,N}^{\prime -1}\|K\|^2 + \frac{2(T-t_0)mN_i\delta_Nh_{i,N}^{\prime -1}\log(N)}{54}}\right) + \mathbb{P}_i^{\otimes N}\left(\mathcal{A}_i^{c}\right)\\
        \leq &~ 2\exp\left(-\dfrac{m^2N_i\delta_{i,N}^2h_{i,N}^{\prime}}{1296C_{t_0,T}^{\prime\prime}\|K\|^2}\right) + \mathbb{P}_i^{\otimes N}\left(\mathcal{A}_i^{c}\right).
    \end{aligned}
\end{equation}
From \cite{van1995exponential}, \textit{Lemma 2.1}, on event $\{N_i > 1\}$, we obtain
\begin{equation}\label{eq:bound-lambda-32}
    \begin{aligned}
        \mathbb{P}_i^{\otimes N}\left(\mathcal{A}_i^{c}\right) \leq &~ \sum_{j=1}^{N_i}\mathbb{P}\left(\left|\int_{t_0}^{T}K_{h_{i,N}^{\prime}}(X_{\eta(t)}^{ji}-x)dW_t^j\right| > (T-t_0)h_{i,N}^{\prime -1}\log(N)\right)\\
        \leq &~ N_i\exp\left(-\dfrac{(T-t_0)\log^2(N)}{2\|K\|_{\infty}^2}\right),
    \end{aligned}
\end{equation}
where $\eta(t)=t_k$ for all $t \in [t_k, t_{k+1})$. Thus, from Equations~\eqref{eq:bound-lambda-32} and \eqref{eq:bound-Lambda-31}, we obtain
\begin{equation}\label{eq:bound-Lambda-3}
    \Lambda_{i3} \leq 2\exp\left(-\dfrac{m^2N_i\delta_{i,N}^2h_{i,N}^{\prime}}{1296C_{t_0,T}^{\prime\prime}\|K\|^2}\right) + N_i\exp\left(-\dfrac{(T-t_0)\log^2(N)}{2\|K\|_{\infty}^2}\right).
\end{equation}

\paragraph{Conclusion.}

From Equations~\eqref{eq:bound-Lambda-3}, \eqref{eq:bound-lambda-2}, \eqref{eq:bound-lambda-1} and \eqref{eq:T1i-15}, conditional on $\{N_i > 1\}$, we obtain
\begin{align*}
    \mathbb{P}_i^{\otimes N}\left(\left\|\widehat{b}_{i,N,h_{i,N},h_{i,N}^{\prime}} - b_i^*\right\|_{\infty} \geq \delta_{i,N}\right) \leq &~ 6\exp\left[-\mathbf{C}N_i\delta_{i,N}^2(h_{i,N} \land h_{i,N}^{\prime})\right] + N_i\exp\left(-\dfrac{(T-t_0)\log^2(N)}{2\|K\|_{\infty}^2}\right)\\
    & + \dfrac{6\|b_i^*\|_{\infty}}{\delta_{i,N}}\exp\left(-\mathbf{C}^{\prime}N_ih_{i,N}\right),
\end{align*}
where 
\begin{equation*}
    \begin{aligned}
        \mathbf{C}:= &~ \min\left\{\dfrac{m^2}{576C_{t_0,T}\left\|b_i^*\right\|_{\infty}^2\|K\|^2}, \dfrac{m^2}{1296C_{t_0,T}^{\prime}\left\|b_i^*\right\|_{\infty}^2\|K\|^2}, \dfrac{m^2}{1296C_{t_0,T}^{\prime\prime}\|K\|^2}, ~ i \in \mathcal{Y}\right\},\\
        \mathbf{C}^{\prime} := &~ \frac{3m^2}{96C_{t_0,T}\|K\|^2 + 16mT\|K\|_{\infty}}.
    \end{aligned}
\end{equation*}
\end{proof}

\subsection{Proof of Theorem~\ref{thm:upper-bound}}

\begin{proof}
     The excess risk of plug-in classifier $\w{g}$ is given by
    \begin{align*}
        R(g_{\w{\bf f}}) - R(g_{{\bf f}^*}) = &~ \mathbb{E}_{X,Y}\left[|2\Phi_{\bf f^*}(X) - 1|\mathds{1}_{g_{\w{\bf f}}(X) \neq g_{\bf f^*}(X)}\right]\\
         \leq &~ 2\varepsilon\mathbb{P}_{X,Y}\left(\left\{0 < \left|\Phi_{\bf f^*}(X) - \dfrac{1}{2}\right| \leq \varepsilon\right\} \cap \left\{g_{\w{\bf f}}(X) \neq g_{\bf f^*}(X)\right\}\right)\\
         & + 2\mathbb{P}_{X,Y}\left(\left\{\left|\Phi_{\bf f^*}(X) - \dfrac{1}{2}\right| > \varepsilon\right\} \cap \left\{g_{\w{\bf f}}(X) \neq g_{\bf f^*}(X)\right\}\right).
    \end{align*}
    Since
    \begin{align*}
        & \left\{g_{\w{\bf f}}(X) \neq g_{{\bf f^*}}(X)\right\} = \left\{\Phi_{\w{\bf f}}(X) \geq \dfrac{1}{2}, \Phi_{\bf f^*}(X) < \dfrac{1}{2}\right\} \cup \left\{\Phi_{\w{\bf f}}(X) < \dfrac{1}{2}, \Phi_{\bf f^*}(X) \geq \dfrac{1}{2}\right\},\\
        & \left\{\left|\Phi_{\bf f^*}(X) - \dfrac{1}{2}\right| > \varepsilon\right\} = \left\{\Phi_{\bf f^*}(X) \geq \dfrac{1}{2} + \varepsilon\right\} \cup \left\{\Phi_{\bf f^*}(X) \leq \dfrac{1}{2} - \varepsilon\right\},
    \end{align*}
    we deduce that for all ${\bf f^*} \in {\bf F}(\beta, R)$, 
    \begin{align*}
        & \left\{\left|\Phi_{\bf f^*}(X) - \dfrac{1}{2}\right| > \varepsilon\right\} \cap \left\{g_{\w{\bf f}}(X) \neq g_{\bf f^*}(X)\right\} = \left\{\left|\Phi_{\w{\bf f}}(X) - \Phi_{\bf f^*}(X)\right| \geq \varepsilon\right\}.
    \end{align*}
    We then obtain the following: 
    \begin{equation}\label{eq:ExR1}
        \begin{aligned}
            \mathbb{E}_{\mathbb{P}^{\otimes N}}\left[R(g_{\w{\bf f}}) - R(g_{\bf f^*})\right] \leq &~ 2\varepsilon\mathbb{P}_{X}\left(0 < \left|\Phi_{\bf f^*}(X) - \dfrac{1}{2}\right| \leq \varepsilon\right)\\
            & + \mathbb{E}_{\mathbb{P}^{\otimes N}}\left[\mathbb{P}_{X}\left(\left|\Phi_{\w{\bf f}}(X) - \Phi_{\bf f^*}(X)\right| \geq \varepsilon\right)\right], ~~ .
        \end{aligned}
    \end{equation}

    \subsection*{Upper bound of $\mathbb{E}_{\mathbb{P}^{\otimes N}}\left[\mathbb{P}_{X}\left(\left|\Phi_{\w{\bf f}}(X) - \Phi_{\bf f^*}(X)\right| \geq \varepsilon\right)\right]$}

    From Equation~\eqref{eq:estimation-p} and the Strong Law of Large Numbers, there exists $M_0 \in \mathbb{N}$ such that for all $N \geq M_0$, $\min(\w{p}_0, \w{p}_1) \geq p_{\min}^*/3 ~~ a.s.$, where $p_{\min}^* = \min(p_0^*, p_1^*)$. We have the following:
    \begin{equation*}
        \begin{aligned}
            \left|\Phi_{\w{\bf f}}(X) - \Phi_{\bf f^*}(X)\right| \leq \left|\w{\Phi}_{\w{\bf b}}(\bar{X}) - \Phi_{\w{\bf b}}^*(\bar{X})\right| + \left|\Phi_{\w{\bf b}}^*(\bar{X}) - \Phi_{\bf b^*}^*(\bar{X})\right| + \left|\Phi_{\bf b^*}^*(\bar{X}) - \Phi_{\bf b^*}^*(X)\right|,
        \end{aligned}
    \end{equation*}
    and from the proof of Theorem 1 in \cite{denis2024nonparametric}, we have
    \begin{equation*}
        \begin{aligned}
            \left|\w{\Phi}_{\w{\bf b}}(\bar{X}) - \Phi_{\w{\bf b}}^*(\bar{X})\right| \leq C_{{\bf p}_{\min}^*}\left(\left|\w{p}_0 - p_0^*\right| + \left|\w{p}_1 - p_1^*\right|\right), ~~ \left|\Phi_{\w{\bf b}}^*(\bar{X}) - \Phi_{\bf b^*}^*(\bar{X})\right| \leq \sum_{i \in \mathcal{Y}}\left|\bar{F}_{\w{\bf b}}^i(\bar{X}) - \bar{F}_{\bf b^*}^i(\bar{X})\right|,
        \end{aligned}
    \end{equation*}
   and $\mathbb{E}_X\left(\left|\Phi_{\bf b^*}^*(\bar{X}) - \Phi_{\bf b^*}^*(X)\right|^2\right) = \mathcal{O}(\Delta_n)$, where $C_{{\bf p}_{\min}^*} > 0$ is a constant depending on ${\bf p}_{\min}^*$. Using Markov's inequality, for all $\varepsilon>0, ~ \mathbb{P}_X\left(\left|\Phi_{\bf b^*}^*(\bar{X}) - \Phi_{\bf b^*}^*(X)\right|>\varepsilon\right) = \mathcal{O}\left(\Delta_n\varepsilon^{-2}\right)$. Then we obtain:
    \begin{multline}\label{eq:UpperBoundPHI}
            \mathbb{E}_{\mathbb{P}^{\otimes N}}\left[\mathbb{P}_{X}\left(\left|\Phi_{\w{\bf f}}(X) - \Phi_{\bf f^*}(X)\right| \geq \varepsilon\right)\right] \leq \sum_{i \in \mathcal{Y}}\mathbb{E}_{\mathbb{P}^{\otimes N}}\left[\mathbb{P}_{X}\left(\left|\bar{F}_{\widehat{\mathbf{b}}}^i(\bar{X}) - \bar{F}_{\mathbf{b}^*}^i(\bar{X})\right| \geq \dfrac{\varepsilon}{6}\right)\right]\\
            + \sum_{i \in \mathcal{Y}}\mathbb{P}^{\otimes N}\left(\left|\widehat{p}_i - p_i^*\right| \geq \dfrac{\varepsilon}{6C_{{\bf p}_{\min}^*}}\right) + \mathcal{O}(\Delta_n\varepsilon^{-2}).
    \end{multline}

    \subsubsection*{Upper bound of $\sum_{i \in \mathcal{Y}}\mathbb{E}_{\mathbb{P}^{\otimes N}}\left[\mathbb{P}_{X}\left(\left|\bar{F}_{\widehat{\mathbf{b}}}^i(\bar{X}) - \bar{F}_{\mathbf{b}^*}^i(\bar{X})\right| \geq \dfrac{\varepsilon}{6}\right)\right]$}

    From Equation~\eqref{eq:discrete}, for each $i \in \mathcal{Y}$,
     \begin{equation}
         \begin{aligned}
             \bar{F}_{\widehat{\mathbf{b}}}^i(\bar{X}) - \bar{F}_{\mathbf{b}^*}^i(\bar{X}) = &~ \int_{0}^{T}(\widehat{b}_{i,N,h_{i,N}} - b_i^*)(X_{\eta(s)})dW_s + \int_{0}^{T}(\widehat{b}_{i,N,h_{i,N}} - b_i^*)(X_{\eta(s)})(b_Y^* - b_i^*)(X_{\eta(s)})ds \\
             & - \dfrac{1}{2}\int_{0}^{T}(\widehat{b}_{i,N,h_{i,N}} - b_i^*)^2(X_{\eta(s)})ds.
         \end{aligned} 
     \end{equation}
     Then, setting $C_{\bf b^*} = \max\left\{\left\|b_0^*\right\|_{\infty}, \left\|b_1^*\right\|_{\infty}\right\}$, for each $i \in \mathcal{Y}$ and for all $\varepsilon > 0$ close enough to $0$,
     \begin{multline}\label{eq:Fbar-i}
        \mathbb{P}_{X}\left(\left|\bar{F}_{\widehat{\mathbf{b}}}^i(\bar{X}) - \bar{F}_{\mathbf{b}^*}^i(\bar{X})\right| \geq \dfrac{\varepsilon}{6}\right) \leq \mathbb{P}_X\left(\left|\int_{0}^{T}(\w{b}_{i,N,h_{i,N},h_{i,N}^{\prime}} - b_i^*)(X_{\eta(s)})dW_s\right| \geq \dfrac{\varepsilon}{18}\right)\\
         + \mathbb{P}_X\left(\left\|\w{b}_{i,N,h_{i,N},h_{i,N}^{\prime}} - b_i^*\right\|_{\infty} \geq \dfrac{\varepsilon}{36TC_{\bf b^*}}\right),
     \end{multline}
    Remark that $\mathbb{P}_X\left(\left\|\w{b}_{i,N,h_{i,N},h_{i,N}^{\prime}} - b_i^*\right\|_{\infty} \geq \frac{\varepsilon}{36TC_{\bf b^*}}\right) = \mathds{1}_{\left\|\w{b}_{i,N,h_{i,N},h_{i,N}^{\prime}} - b_i^*\right\|_{\infty} \geq \varepsilon/(36TC_{\bf b^*})}$. Moreover, From Lemma 2.1 in \cite{van1995exponential}, we obtain
    \begin{equation*}
        \mathbb{P}_X\left(\left|\int_{0}^{T}(\w{b}_{i,N,h_{i,N},h_{i,N}^{\prime}} - b_i^*)(X_{\eta(s)})dW_s\right| \geq \dfrac{\varepsilon}{18}\right) \leq 2\exp\left(-\dfrac{\varepsilon^2}{36T\left\|\w{b}_{i,N,h_{i,N},h_{i,N}^{\prime}} - b_i^*\right\|_{\infty}}\right).
    \end{equation*}
    Then, for any $\delta > 0$ and from the above result and Equation~\ref{eq:Fbar-i}, 
     \begin{multline*}
        \sum_{i \in \mathcal{Y}}\mathbb{E}_{\mathbb{P}^{\otimes N}}\left[\mathbb{P}_{X}\left(\left|\bar{F}_{\widehat{\mathbf{b}}}^i(\bar{X}) - \bar{F}_{\mathbf{b}^*}^i(\bar{X})\right| \geq \dfrac{\varepsilon}{6}\right)\right] \leq \mathbb{P}^{\otimes N}(N_i \leq 1) + 4\exp\left(-\dfrac{\varepsilon^2}{\delta^2}\right) \\
        + \sum_{i \in \mathcal{Y}}\mathbb{E}_{\mathbb{P}^{\otimes N}}\left[\mathds{1}_{N_i>1}\mathbb{P}_i^{\otimes N}\left(\left\|\w{b}_{i,N,h_{i,N},h_{i,N}^{\prime}} -  b_i^*\right\|_{\infty} > \dfrac{\delta}{36TC_{\bf b^*}}\right)\right]\\
        + 2\sum_{i \in \mathcal{Y}}\mathbb{E}_{\mathbb{P}^{\otimes N}}\left[\mathds{1}_{N_i>1}\mathbb{P}_i^{\otimes N}\left(\left\|\w{b}_{i,N,h_{i,N},h_{i,N}^{\prime}} - b_i^*\right\|_{\infty} > \dfrac{\delta}{36T}\right)\right]. 
     \end{multline*}
     Thus, by Theorem~\ref{thm:exp-bound-drift} with $\delta = \frac{1}{2}\log^{3/2}(N)N^{-\beta/(2\beta+1)}$ and $h_{i,N} = h_{i,N}^{\prime} = N^{-1/(2\beta+1)}$
     \begin{multline}\label{eq:Fbar-i2}
         \sum_{i \in \mathcal{Y}}\mathbb{E}_{\mathbb{P}^{\otimes N}}\left[\mathbb{P}_{X}\left(\left|\bar{F}_{\widehat{\mathbf{b}}}^i(\bar{X}) - \bar{F}_{\mathbf{b}^*}^i(\bar{X})\right| \geq \dfrac{\varepsilon}{6}\right)\right] \\
         \leq 4\exp\left(-\dfrac{2\varepsilon^2N^{\frac{2\beta}{2\beta+1}}}{\log^3(N)}\right) + C^*\dfrac{N^{\frac{\beta}{2\beta+1})}}{\log^{3/2}(N)}\mathbb{E}_{\mathbb{P}^{\otimes N}}\left[\mathds{1}_{N_i>1}\exp\left(-\mathbf{C}^{\prime\prime\prime}\dfrac{N_i}{N^{\frac{1}{2\beta+1}}}\right)\right] + \mathbb{P}^{\otimes N}(N_i \leq 1)\\
          + 4N\exp\left(-\dfrac{(T-t_0)\log^2(N)}{2\|K\|_{\infty}^2}\right) + 24\mathbb{E}_{\mathbb{P}^{\otimes}}\left[\mathds{1}_{N_i>1}\exp\left(-\mathbf{C}^{\prime\prime}\dfrac{N_i}{N}\log^3(N)\right)\right],
     \end{multline}
     where the constant $C^* > 0$ depends on ${\bf b^*}$, $\mathbf{C}^{\prime\prime}>0$ is a constant depending on $\mathbf{C}$, $T$ and $C_{\bf b^*}$, and the constant $\mathbf{C}^{\prime\prime\prime} > 0$ depends on $\mathbf{C}^{\prime}$, $T$ and $C_{\bf b^*}$. Since for each $i \in \mathcal{Y}$, $N_i \sim \mathrm{Binomial}(N,p_i^*)$, using the Taylor expansion up to second order,
     \begin{equation}\label{eq:Fbar-i3}
         \begin{aligned}
             \mathbb{E}_{\mathbb{P}^{\otimes}}\left[\mathds{1}_{N_i>1}\exp\left(-\mathbf{C}^{\prime\prime}\dfrac{N_i}{N}\log^3(N)\right)\right] \leq &~ \sum_{j=0}^{N}\exp\left(-\mathbf{C}^{\prime\prime}j\dfrac{\log^3(N)}{N}\right)\binom{N}{j}p_i^{* j}(1-p_i^*)^{N-j}\\
             = &~ \left[1-p_i^* + p_i^{*}\exp(-\mathbf{C}^{\prime\prime}\log^3(N)/N)\right]^N\\
             = &~ \left[1 - \mathbf{C}^{\prime\prime}p_i^*\dfrac{\log^3(N)}{N} + \dfrac{\mathbf{C}^{\prime\prime 2}p_i^*\log^6(N)}{2N^2} + \mathcal{O}\left(\dfrac{\log^6(N)}{N^2}\right)\right]^N\\
             = &~ \mathcal{O}\left(\exp\left(-\mathbf{C}^{\prime\prime}p_i^*\log^3(N)\right)\right).
         \end{aligned}
     \end{equation}
    Similarly, for each $i \in \mathcal{Y}$, we obtain the following: 
    \begin{equation}\label{eq:Fbar-i4}
        \begin{aligned}
            \mathbb{E}_{\mathbb{P}^{\otimes N}}\left[\mathds{1}_{N_i>1}\exp\left(-\mathbf{C}^{\prime\prime\prime}\dfrac{N_i}{N^{\frac{1}{2\beta+1}}}\right)\right] \leq &~ \left[1 - \dfrac{\mathbf{C}^{\prime\prime\prime}p_i^*}{N^{\frac{1}{2\beta+1}}} + \dfrac{\mathbf{C}^{\prime\prime\prime 2}p_i^*}{2N^{\frac{2}{2\beta+1}}} + \mathcal{O}\left(\dfrac{1}{N^{\frac{2}{2\beta+1}}}\right)\right]^N \\
            = &~ \mathcal{O}\left(\exp\left(-\mathbf{C}^{\prime\prime\prime}p_i^*N^{\frac{2\beta}{2\beta+1}}\right)\right).
        \end{aligned}
    \end{equation}
    Finally, since $\mathbb{P}^{\otimes N}(N_i \leq 1) = \mathrm{O}\left(N(1-p_i^*)^{N}\right)$, for $\varepsilon = \log^2(N)N^{-\beta/(2\beta+1)}$ and from Equations~\eqref{eq:Fbar-i4}, \eqref{eq:Fbar-i3} and \eqref{eq:Fbar-i2}, there exists a constant $C>0$ such that
     \begin{equation}\label{eq:Term-F}
     \begin{aligned}
          \sum_{i \in \mathcal{Y}}\mathbb{E}_{\mathbb{P}^{\otimes N}}\left[\mathbb{P}_{X}\left(\left|\bar{F}_{\widehat{\mathbf{b}}}^i(\bar{X}) - \bar{F}_{\mathbf{b}^*}^i(\bar{X})\right| \geq \dfrac{\varepsilon}{6}\right)\right] \leq &~ CN^{-2}.
     \end{aligned} 
     \end{equation}

\subsubsection*{Upper bound of $\sum_{i \in \mathcal{Y}}\mathbb{P}^{\otimes N}\left(\left|\widehat{p}_i - p_i^*\right| \geq \dfrac{\varepsilon}{6C_{{\bf p}_{\min}^*}}\right)$}

Using Bernstein's inequality with $\varepsilon = \log^2(N)N^{-\beta/(2\beta+1)}$, there exists a constant $c>0$ such that
 \begin{equation}\label{eq:Term-p}
     \begin{aligned}
         \mathbb{P}^{\otimes N}\left(\left|\widehat{p}_i - p_i^*\right| \geq \dfrac{\log^2(N)N^{-\frac{\beta}{2\beta+1}}}{6C_{{\bf p^*}_{\min}}}\right) = &~ \mathbb{P}^{\otimes N}\left(\left|\sum_{j=1}^{N}\mathds{1}_{Y_j = i} - Np_i^*\right| \geq \dfrac{\log^2(N)N^{\frac{\beta+1}{2\beta+1}}}{6C_{{\bf p^*}_{\min}}}\right)\\
        \leq &~ 2\exp\left(-c\log^4(N)N^{1/(2\beta+1)}\right).
     \end{aligned}
 \end{equation}
\subsubsection*{Conclusion of the proof}

Finally, since $\Delta_n = \mathcal{O}(N^{-2})$, by Proposition~\ref{prop:margin-condition} with $\varepsilon = \log^2(N)N^{-\beta/(2\beta+1)}$ and from Equations~\eqref{eq:Term-p}, \eqref{eq:Term-F}, \eqref{eq:UpperBoundPHI} and \eqref{eq:ExR1}, there exists a constant $C>0$ such that
\begin{equation*}
    \mathbb{E}_{\mathbb{P}^{\otimes N}}\left[R(g_{\w{\bf f}}) - R(g_{\bf f^*})\right] \leq C\log^4(N)N^{-2\beta/(2\beta+1)}.
\end{equation*}
\end{proof}

\subsection{Proof of Theorem~\ref{thm:lower-bound}}

\begin{proof}
We establish a lower bound of the worst excess risk of the nonparametric plug-in classifier $\widehat{g} = g_{\w{\bf f}}$ with respect to the Bayes classifier $g^{*} = g_{\bf f^*}$, where ${\bf f^*} = ({\bf b^*}, {\bf p^*}) \in {\bf F}(\beta, R) = {\bf B}(\beta, R) \times {\bf P}$. Recall that ${\bf p^*} \in {\bf P}$ represents the discrete law of the label $Y \in \mathcal{Y}$ while ${\bf b^*} = (b_0^*, b_1^*) \in {\bf B}(\beta, R)$ is the vector of drift functions that characterize the diffusion model \eqref{eq:classif-model}. More precisely, we want to prove that there exists a constant $c > 0$ such that 
\begin{equation*}
    \underset{\w{\bf f}}{\inf}\underset{{\bf f^*} \in {\bf F}(\beta, R)}{\sup}\mathbb{E}_{\mathbb{P}^{\otimes N}}\left[R(g_{\w{\bf f}}) - R(g_{\bf f^*})\right] \geq cN^{-2\beta/(2\beta+1)},
\end{equation*}
where the joint distribution $\mathbb{P}^{\otimes N}$ of the learning sample $\mathcal{Z}^N$ is fully characterized by the model parameter ${\bf f^*} \in {\bf F}(\beta, R)$. The common strategy consists of restricting the initial space ${\bf F}(\beta,R)$ to a well-chosen finite set, which is the set of hypotheses, and on which a lower bound of the excess risk is derived. Formally, consider a finite set ${\bf F}^{M} \subset {\bf F}(\beta, R)$ of $M+1$ hypotheses with $M \in \mathbb{N}^*$. We obtain
\begin{equation}\label{eq:lower-bound1}
    \underset{\w{\bf f}}{\inf}\underset{{\bf f^*} \in {\bf F}(\beta, R)}{\sup}\mathbb{E}_{\mathbb{P}^{\otimes N}}\left[R(g_{\w{\bf f}}) - R(g_{\bf f^*})\right] \geq \underset{\w{\bf f}}{\inf}\underset{{\bf f^*} \in {\bf F}^{M}}{\sup}\mathbb{E}_{\mathbb{P}^{\otimes N}}\left[R(g_{\w{\bf f}}) - R(g_{\bf f^*})\right].
\end{equation}
Then, it suffices to prove that there exists a constant $c > 0$ such that
\begin{equation*}
    \underset{\w{\bf f}}{\inf}\underset{{\bf f^*} \in {\bf F}^{M}}{\sup}\mathbb{E}_{\mathbb{P}^{\otimes N}}\left[R(g_{\w{\bf f}}) - R(g_{\bf f^*})\right] \geq cN^{-2\beta/(2\beta+1)}.
\end{equation*}
For this purpose, we consider the diffusion model
\begin{equation}\label{eq:model2}
    dX_t = Yf(X_t)dt + dW_t, ~~ t \in [0,T], ~~ X_0 = x_0 \in \mathbb{R},
\end{equation}
where $b_0^* \equiv 0$ and $b_1^* = f \in \Sigma^M$, with the finite set $\Sigma^M$ of size $M+1$ chosen carefully, the law of the label $Y \in \mathcal{Y} = \{0,1\}$ is set to ${\bf p}^* = \left(1/2, 1/2\right)$. 

\subsubsection{Construction of the set of hypotheses}

Let $M \in \mathbb{N}^*$ and $\beta > 1$. Define $D = \left\lfloor N^{1/(2\beta+1)} \right\rfloor$ and
\begin{align*}
  x_k = \dfrac{k - 1/2}{D}, ~~ \phi_k(x) := &~ RD^{-\beta}K\left(\dfrac{x-x_k}{D^{-1}}\right), ~~ x \in [0,1], ~~ k \in \{1, \ldots, D\},
\end{align*}
where the function $K : \mathbb{R} \longrightarrow [0, \infty)$ is given by $ K(x) = aK_0(2x), ~~ K_0(x) = \exp\left(-\frac{1}{1-x^2}\right)\mathds{1}_{(-1,1)}(x)$ with $a>0$ a sufficiently small real number. The function $K$ satisfies the following:
    \begin{equation}\label{eq:Fun-K}
        K \in \Sigma(\beta,1/2) \cap \mathcal{C}^{\infty}(\mathbb{R}), ~~ \mathrm{and} ~~ K(u) > 0 \iff u \in (-1/2, 1/2),
    \end{equation}
    and the functions $\phi_k$ belong to the H\"older class $\Sigma(\beta, R)$ in the interval $[0,1]$, and satisfy 
    \begin{equation}\label{eq:CompactSupp-Phi-k}
        \forall ~ k = 1, \ldots, D, ~~ \phi_k(x) > 0 \iff x \in \left(\dfrac{k-1}{D}, \dfrac{k}{D}\right),
    \end{equation}
(see \cite{tsybakov2008introduction}, page 92). Suppose that the function $f$ belongs to the finite set $\Sigma^M$ of size $M+1$ given by
\begin{equation}\label{eq:hypothesisSet}
    \Sigma^M := \left\{f_j = \kappa D^{-\beta} + \sum_{k=1}^{D}{\theta_k^j\phi_k}, ~~ \theta^j = \left(\theta_1^j, \ldots, \theta_D^j\right) \in (0,1)^D, ~~ \kappa > 0, ~~ j = 0, \ldots, M\right\},
\end{equation}
where, for $j = 0, ~~ \theta^j = (0,\ldots, 0)$, $f_0 = \kappa D^{-\beta}$. By construction, we have $\Sigma^M \subset \mathcal{C}^{\infty}(\mathbb{R}, \mathbb{R}) \cap \Sigma(\beta, R)$ (see \cite{tsybakov2008introduction}, \textit{chapter 2, page 93}). Moreover, since the function $f \in \Sigma^M \subset \mathcal{C}^{\infty}(\mathbb{R}, \mathbb{R}) \cap \Sigma(\beta, R)$ never vanishes, being strictly positive, we have $b_0^*(x_0) = 0 \neq b_1^*(x_0)$, the H\"ormander's consdition is satisfied, and the random variable  
\begin{equation*}
    \int_{0}^{T}(b_1^* - b_0^*)(X_s)dW_s = \int_{0}^{T}f(X_s)dW_s
\end{equation*}
admits a density function that belongs to the space $\mathcal{C}^{\infty}(\mathbb{R}, \mathbb{R})$ (see \cite{nualart2006malliavin}, \textit{Chapter 2, Theorem 2.3.3, p.133}). We deduce the following finite set:
\begin{equation*}
    {\bf B}^{M}(\beta, R) = \left\{{\bf b^*} = (0, f), ~ f \in \Sigma^M\right\}, ~~ {\bf P} = \{(1/2, 1/2)\} ~~ \mathrm{and} ~~ {\bf F}^M = {\bf B}^{M}(\beta, R) \times {\bf P}.
\end{equation*}
\\
The regression function $\Phi_{\bf f^*} = \Phi_f$ is given by
\begin{equation}\label{eq:Phib}
    \Phi_{f}(X) = \dfrac{\exp\left(F_{f}(X)\right)}{1 + \exp\left(F_{f}(X)\right)},
\end{equation}
where the diffusion process $X = (X_t)_{t \in [0,T]}$ is the unique strong solution of Equation~\eqref{eq:model2}, and
\begin{equation}\label{eq:Ff}
    F_{f}(X) = \int_{0}^{T}{f(X_s)dX_s} - \dfrac{1}{2}\int_{0}^{T}{f^{2}(X_s)ds}.
\end{equation}
Thus, for any estimator $\widehat{f}$ of $f \in \Sigma^M$, the plug-in classifier $\w{g} = g_{\w{f}}$ is given by $\w{g}(X) = \mathds{1}_{\Phi_{\w{f}}(X) \geq 1/2}$. Then, we have the following:
\begin{equation}\label{eq:lower-bound2}
    \underset{\w{\bf f}}{\inf}\underset{{\bf f^*} \in {\bf F}^{M}}{\sup}\mathbb{E}_{\mathbb{P}^{\otimes N}}\left[R(g_{\w{\bf f}}) - R(g_{\bf f^*})\right] = \underset{\w{f}}{\inf}\underset{f \in \Sigma^{M}}{\sup}\mathbb{E}_{\mathbb{P}^{\otimes N}}\left[R(g_{\w{f}}) - R(g^{*})\right].
\end{equation}

\subsubsection{General result on the lower bound of the excess risk}

The proof of the present theorem is based on the Assouad's lemma adapted to the classification problem (see \textit{e.g.} \cite{Audibert2004ClassificationUP}, \textit{Definition 5.1} and \textit{Lemma 5.1}). Let $m \in \mathbb{N}^*, ~ w \in (0,1), ~ \mathfrak{b} \in (0,1], ~ \mathfrak{b}^{\prime} \in (0,1]$. From Definition 5.1 in \cite{Audibert2004ClassificationUP}, a $(m, w, \mathfrak{b}, \mathfrak{b}^{\prime})-$hypercube of probability distributions is a family
\begin{equation*}
    \mathcal{P} := \left\{\mathbb{P}_{\vec{\sigma}}^{X,Y}, ~ \vec{\sigma}:=\left(\sigma_1, \ldots, \sigma_m\right) \in \{-1, +1\}^m\right\}
\end{equation*}
of $2^m$ probability distributions of the random pair $(X, Y)$ in the measurable space $\mathcal{X} \times \mathcal{Y}$ such that the marginal distribution $\mathbb{P}_{\vec{\sigma}}^X$ of $X$ does not depend on $\vec{\sigma}$, that is:
\begin{equation*}
    \forall ~ \vec{\sigma} \in \{-1, +1\}^m, ~~ \mathbb{P}_{\vec{\sigma}}^X(dX) = \mathbb{P}_{(+1, \ldots, +1)}^X(dX) = \mathbb{P}_X.
\end{equation*}
Moreover, there exists a partition $\mathcal{X}_0, \ldots, \mathcal{X}_m$ of $\mathcal{X}$ such that:
\begin{itemize}
    \item for any $j \in \{1, \ldots, m\}$, we have $\mathbb{P}_X(X \in \mathcal{X}_j) = w$
    \item for any $j \in \{0, \ldots, m\}$, for any $X \in \mathcal{X}_j$, we have
    \begin{equation*}
        \mathbb{P}_{\vec{\sigma}}^{X,Y}(Y = 1 | X) = \dfrac{1 + \sigma_j\xi(X)}{2} = 1 - \mathbb{P}_{\vec{\sigma}}^{X,Y}(Y = 0 | X),
    \end{equation*}
    where $\sigma_0 \equiv 1$ and $\xi: \mathcal{X} \longrightarrow [0,1]$ is such that for any $j \in \{1, \ldots, m\}$,
\begin{align*}
 \mathfrak{b} = \sqrt{1 - \left(\mathbb{E}_X\left[\sqrt{1 - \xi^2(X)} \biggm\vert X \in \mathcal{X}_j\right]\right)^2}, ~~~ \mathfrak{b}^{\prime} = \mathbb{E}_X\left[\xi(X) | X \in \mathcal{X}_j\right].
\end{align*}
\end{itemize}
Then, from Lemma 5.1 in \cite{Audibert2004ClassificationUP}, for any binary classification rule $\widehat{g}$ built from the learning sample $\mathcal{Z}^N$, we have the following:
\begin{equation}\label{eq:lower-bound3}
    \underset{\mathbb{P}_{X,Y} \in \mathcal{P}}{\sup}{\left\{\mathbb{E}_{\mathbb{P}^{\otimes N}}\left[R(\widehat{g}) - R(g^*)\right]\right\}} \geq \dfrac{1 - \mathfrak{b}\sqrt{Nw}}{2}mw\mathfrak{b}^{\prime}.
\end{equation}
In our framework, since the discrete law ${\bf p^*} = (1/2, 1/2)$ of the label $Y \in \mathcal{Y}$ is known, the joint distribution $\mathbb{P}_{X,Y}$ of the random pair $(X,Y) \in \mathcal{X} \times \mathcal{Y}$ is completely characterized by the drift function $f \in \Sigma^M$. Then, considering a $(m,w,\mathfrak{b}, \mathfrak{b}^{\prime})-$ hypercube of probability distributions, 
\begin{equation*}
    \mathcal{P}^M := \left\{\mathbb{P}_{f, \vec{\sigma}}, ~~ \vec{\sigma} := (\sigma_1, \ldots, \sigma_m) \in \{-1, +1\}^m, ~~ f \in \Sigma^M\right\}, 
\end{equation*}
with $M = 2^m-1$, we have the following:
\begin{equation}\label{eq:lower-bound4}
    \underset{\w{f}}{\inf}\underset{f \in \Sigma^{M}}{\sup}\mathbb{E}_{\mathbb{P}^{\otimes N}}\left[R(g_{\w{f}}) - R(g^{*})\right] = \underset{\widehat{f}}{\inf}\underset{\mathbb{P}_{X,Y} \in \mathcal{P}^M}{\sup}\mathbb{E}_{\mathbb{P}^{\otimes N}}\left[R(g_{\widehat{f}}) - R(g^*)\right],
\end{equation}
as each empirical classifier $\widehat{g} = g_{\widehat{f}}$ is completely identified by the nonparametric estimator $\widehat{f}$ of the drift function $f \in \Sigma^M$ built from the learning sample $\mathcal{Z}^N$. The next step is the construction of the $(m,w,\mathfrak{b}, \mathfrak{b}^{\prime}) -$ hypercube $\mathcal{P}^M$.

\subsubsection{Construction of the hypercube}

The diffusion process $X$ is the unique strong solution of the following stochastic differential equation:
\begin{equation*}
 X_t = Y\int_{0}^{t}{f(X_s)ds} + W_t, ~~ t \in [0,T], ~~ X_0 = 0,
\end{equation*}
and from \cite{dacunha1986estimation}, \textit{Lemma 2}, $X$ admits a transition density $(s, t, x, y) \mapsto \Gamma_X(s, t, x, y)$ given for all $(s, t, x, y) \in (0,T] \times (0,T] \times \mathbb{R} \times \mathbb{R}$ by
\begin{equation}\label{eq:transition-density-assouad}
    \begin{aligned}
        \Gamma_X(s, t, x, y) := &~ \dfrac{1}{2}\Gamma_{X|Y=1}(s, t, x, y) + \dfrac{1}{2}\Gamma_{X|Y=0}(s, t, x, y)\\
        = &~ \dfrac{1}{2\sqrt{2\pi(t-s)}}\exp\left(-\dfrac{(y-x)^2}{2(t-s)}\right) + \dfrac{\Lambda(s,t,x,y)}{2\sqrt{2\pi(t-s)}}\exp\left(-\dfrac{(y-x)^2}{2(t-s)} + \int_{0}^{x}{f(u)du}\right),
    \end{aligned}
\end{equation}
where $ \Lambda(s,t,x,y) = \widetilde{\mathbf{E}}\left[\exp\left((t-s)\int_{0}^{T}{G((1-u)x + uy + \sqrt{t-s}B_u)du}\right)\right], ~ G = -\dfrac{1}{2}(f^{2} + f^{\prime})$ and $B$ is a Brownian bridge with $\widetilde{\mathbf{E}}[B_t^2] = t(1-t)$ for all $t \in [0,T]$. From Equation~\eqref{eq:Phib}, we have
\begin{equation*}
    \Phi_f(X) = \dfrac{\exp\left(F_f(X)\right)}{\exp\left(F_f(X)\right) + 1} = \dfrac{1 + \widetilde{\xi}(X)}{2},
\end{equation*}
where the function $\widetilde{\xi}: \mathcal{X} \longrightarrow [-1,1]$ is given for any diffusion path $X \in \mathcal{X}$ by
\begin{equation}\label{eq:xitilde}
    \widetilde{\xi}(X) := \dfrac{\exp\left(F_f(X)\right) - 1}{\exp\left(F_f(X)\right) + 1},
\end{equation}
and $F_f(X)$ is given by Equation~\eqref{eq:Ff}. Denote by $\mathcal{X}^{+}$ and $\mathcal{X}^{-}$, two subsets of $\mathcal{X}$ defined as follows:
\begin{equation}\label{eq:xi-plus-less}
\mathcal{X}^{+} = \left\{X \in \mathcal{X}: \widetilde{\xi}(X) > 0\right\} ~~ \mathrm{and} ~~ \mathcal{X}^{-} = \left\{X \in \mathcal{X}: \widetilde{\xi}(X) \leq 0\right\}. 
\end{equation}
In a first step, we show that the two random events $\left\{X \in \mathcal{X}^{+}\right\}$ and $\{X \in \mathcal{X}^-\}$ are non-negligible. We have, on the one hand,
\begin{align*}
    \mathbb{P}_X\left(X \in \mathcal{X}^{+}\right) = &~ \mathbb{P}_X\left(\widetilde{\xi}(X) > 0\right) = \mathbb{P}_X\left(F_f(X) > 0\right) \geq \mathbb{P}_X\left(\int_{0}^{T}f(X_s)dW_s > 0\right).
\end{align*}
From \cite{dacunha1986estimation}, \textit{Lemma 2}, the random variable $\int_{0}^{T}f(X_s)dW_s$ admits the following density function
\begin{equation*}
    \phi(x) = \dfrac{1}{\sqrt{2\pi T}f(x)}\exp\left(-\dfrac{S(x)^2}{2T} + H(x)\right)\widetilde{\mathbf{E}}\left[\exp\left(T\int_{0}^{T}{G(ux + \sqrt{T}B_u)du}\right)\right],
\end{equation*}
where for all $x \in \mathbb{R}$,  $G(x) = -\dfrac{1}{8}(f^{\prime})^2\circ S^{-1}(x) + \dfrac{1}{4}(ff^{\prime\prime})\circ S^{-1}(x)$, $H(x) = -\dfrac{1}{2}\log\left(\frac{f(x)}{f(0)}\right)$ and $S(x) = \int_{0}^{x}\frac{1}{f(u)}du$. We deduce that
\begin{align*}
     \mathbb{P}_X\left(X \in \mathcal{X}^{+}\right) \geq &~ \int_{0}^{1}\phi(x)dx \leq \dfrac{\exp(-\|G\|_{\infty})}{\sqrt{2\pi T}\|f\|_{\infty}}\exp\left(-\|H\|_{\infty}\right)\exp\left(-\dfrac{\kappa^{-2}D^{2\beta}}{2T}\right)> 0.
\end{align*}
On the other hand, we have the following.
\begin{align*}
    \mathbb{P}_X\left(X \in \mathcal{X}^{-}\right) \geq &~ \mathbb{P}_X\left(\int_{0}^{T}f(X_s)dW_s \leq -\dfrac{T\|f\|_{\infty}^2}{2}\right)\\
    \geq &~ \dfrac{T\exp(-\|G\|_{\infty})\|f\|_{\infty}}{2\sqrt{2\pi T}}\exp\left(-\|H\|_{\infty}\right)\exp\left(-\dfrac{\kappa^{-2}D^{2\beta}}{2T}\right)> 0.
\end{align*}
We deduce that the two subsets $\mathcal{X}^{+}$ and $\mathcal{X}^{-}$ satisfy:
\begin{equation}\label{eq:criterion-xi}
    \mathcal{W}(\mathcal{X}^+) > 0, ~~ \mathcal{W}(\mathcal{X}^-) > 0, ~~ \mathcal{X}^+ \cap \mathcal{X}^- = \emptyset ~~ \mathrm{and} ~~ \mathcal{X}^+ \cup \mathcal{X}^- = \mathcal{X},
\end{equation}
where $\mathcal{W}$ is the Wiener measure. We derive the following lemma.

\begin{lemme}\label{lm:partition}
   Let $m, q \geq 1$ be two integers such that $q > m$, and $w \in (0,1)$ such that $wq \leq 1$. For each probability distribution $\mathbb{P}_X$ of the diffusion process $X$ characterized by the drift function $f \in \Sigma^{M}$, there exists a partition $\mathcal{X}_0, \mathcal{X}_1, \ldots, \mathcal{X}_m$ of $\mathcal{X}$ such that
\begin{itemize}
    \item For all $j \in \{1, \ldots, m\}, ~ \mathbb{P}_X(X \in \mathcal{X}_j) = w$,
    \item For all $j \in \{1, \ldots, m\}$, we have $\mathcal{X}_j \subset \mathcal{X}^{-}$ or $\mathcal{X}_j \subset \mathcal{X}^{+}$,
    \item For any $j \in \{1, \ldots, m\}$ and for any $X \in \mathcal{X}_j$, we have
    \begin{equation*}
        \Phi_{\vec{\sigma}, f}(X) = \mathbb{P}_{X,Y}(Y = 1 | X) = \dfrac{1 + \sigma_j\xi(X)}{2} = 1 - \mathbb{P}_{X,Y}(Y = 0 | X),
    \end{equation*}
    where $\xi = \left|\widetilde{\xi}\right|: \mathcal{X} \longrightarrow [0,1]$, and for each $j \in \{1, \ldots, m\}$, $\sigma_j = -1$ for $\mathcal{X}_j \subset \mathcal{X}^{-}$, and $\sigma_j = +1$ for $\mathcal{X}_j \subset \mathcal{X}^{+}$,
    \item For any $X \in \mathcal{X}_0$, $\Phi_{\vec{\sigma}, f}(X) = \mathbb{P}_{X,Y}(Y = 1 | X) = \dfrac{1 + \widetilde{\xi}(X)}{2}$.
\end{itemize}    
\end{lemme}

We then obtain from the above lemma the following $(m, w, \mathfrak{b}, \mathfrak{b}^{\prime})-$hypercube:
\begin{align*}
  \mathcal{P}^M := \left\{\mathbb{P}_{X,Y} = \mathbb{P}_{\vec{\sigma}, f}, ~ \vec{\sigma} \in \{-1, +1\}^m, ~ f \in \Sigma^M\right\}
\end{align*}
containing $2^m$ probability distributions of the random couple $(X,Y)$ in the measurable space $\mathcal{X} \times \{0,1\}$ with the marginal distribution $\mathbb{P}_X$ of $X$ independent of $\vec{\sigma}$. The proof of Lemma~\ref{lm:partition} is provided in the appendix.
 
\subsubsection{Conclusion}

We deduce from Equations~\eqref{eq:lower-bound1}, \eqref{eq:lower-bound2}, \eqref{eq:lower-bound3} and \eqref{eq:lower-bound4} that
\begin{equation}\label{eq:lw1}
 \underset{\w{\bf f}}{\inf}\underset{{\bf f^*} \in {\bf F}^{M}}{\sup}\mathbb{E}_{\mathbb{P}^{\otimes N}}\left[R(g_{\w{\bf f}}) - R(g_{\bf f^*})\right] \geq \dfrac{1 - \mathfrak{b}\sqrt{Nw}}{2}mw\mathfrak{b}^{\prime},
\end{equation}
where 
\begin{align*}
    \mathfrak{b} = \sqrt{1 - \left(\mathbb{E}_X\left[\sqrt{1 - \xi^2(X)} \biggm\vert X \in \mathcal{X}_1\right]\right)^2}, ~~~ \mathfrak{b}^{\prime} = \mathbb{E}_X\left[\xi(X) | X \in \mathcal{X}_1\right].
\end{align*}

\begin{enumerate}
\item[(i)] \textbf{Focus on $\mathfrak{b}$ and $\mathfrak{b}^{\prime}$}

The following lemma provides an upper bound on $\mathfrak{b}$ and a lower bound on $\mathfrak{b}^{\prime}$.
\begin{lemme}\label{lm:Bounds-b-bprime}
    There exist constants $c_{\xi}, C_{\xi}, c,c^{\prime} > 0$ such that
    \begin{equation*}
    \begin{aligned}
        & \forall ~ X \in \mathcal{X}, ~ c_{\xi}D^{-\beta}\left|Q_{\widetilde{f}}(X)\right| \leq \xi(X) \leq C_{\xi}D^{-\beta}\left|Q_{\widetilde{f}}(X)\right|,\\
        & \mathfrak{b} \leq cD^{-\beta} ~~ \mathrm{and} ~~ \mathfrak{b}^{\prime} \geq c^{\prime}D^{-\beta},
    \end{aligned}
    \end{equation*}
    where $ Q_{\widetilde{f}}(X) = \int_{0}^{T}{\widetilde{f}(X_s)dW_s} + \dfrac{D^{-\beta}}{2}\int_{0}^{T}{\widetilde{f}^2(X_s)ds}$ and $\widetilde{f}(x) = \kappa + R\sum_{k=1}^{D}{\theta_kK\left(\dfrac{x-x_k}{D^{-1}}\right)}$.
\end{lemme}

Next, we choose $mw$ so that $mw \rightarrow 0$ at the slowest possible rate. The proof of the above lemma is provided in the appendix.

\item[(ii)] \textbf{Low-noise conditions}

Let $\varepsilon \in (0, 1/8)$. For all $f \in \Sigma^M$, we have the following.
\begin{multline}\label{eq:margin1}
    \mathbb{P}_{X}\left(\left|\Phi_{\vec{\sigma}, f}(X) - \dfrac{1}{2}\right| \leq \varepsilon\right) = \sum_{j=0}^{m}{\mathbb{P}_{X}\left(\left|\Phi_{\vec{\sigma}, f}(X) - \dfrac{1}{2}\right| \leq \varepsilon \biggm\vert X \in \mathcal{X}_j\right)\mathbb{P}\left(X \in \mathcal{X}_j\right)}\\
    = w\sum_{j=1}^{m}{\mathbb{P}_{X}\left(\xi(X) \leq \varepsilon \biggm\vert X \in \mathcal{X}_j\right)} + \mathbb{P}_{X}\left(\left\{\left|\Phi_{\vec{\sigma}, f}(X) - \dfrac{1}{2}\right| \leq \varepsilon\right\}\cap\left\{X \in \mathcal{X}_0\right\}\right).
\end{multline}
By Proposition~\ref{prop:margin-condition}, there exists a constant $C>0$ such that for all $f \in \Sigma^M$,
\begin{equation}\label{eq:margin2}
   \mathbb{P}_{X}\left(\left\{\left|\Phi_{\vec{\sigma}, f}(X) - \dfrac{1}{2}\right| \leq \varepsilon\right\}\cap\left\{X \in \mathcal{X}_0\right\}\right) \leq C\varepsilon.
\end{equation}
On the other hand, from Lemma~\ref{lm:Bounds-b-bprime}, for all $j \in [\![1, m]\!]$, we have
\begin{align*}
    \mathbb{P}_{X}\left(\xi(X) \leq \varepsilon \biggm\vert X \in \mathcal{X}_j\right) \leq &~ \mathbb{P}_{X}\left(\left|Q_{\widetilde{f}}(X)\right| \leq c_{\xi}^{-1}\varepsilon D^{\beta} \biggm\vert X \in \mathcal{X}_j\right),
\end{align*}
where $Q_{\widetilde{f}}(X) = \int_{0}^{T}\widetilde{f}(X_t)dW_t + \frac{D^{-\beta}}{2}\int_{0}^{T}\widetilde{f}^2(X_t)dt$. We deduce that
\begin{equation}\label{eq:margin3}
    \mathbb{P}_{X}\left(\xi(X) \leq \varepsilon \biggm\vert X \in \mathcal{X}_j\right) \leq  \mathbb{P}_{X}\left(\left|\int_{0}^{T}\widetilde{f}(X_t)dW_t\right| \leq c_{\xi}^{-1}\varepsilon D^{\beta} + \dfrac{R\|K\|_{\infty}D^{-\beta}}{2} \biggm\vert X \in \mathcal{X}_j\right).
\end{equation}
Moreover, since $\widetilde{f}$ is an elliptic function, the random variable $\int_{0}^{T}\widetilde{f}(X_t)dW_t$ has a density function $\phi_T$ that can be approximated by Gaussian densities (see, \textit{e.g.} \cite{gobet2002lan}, \textit{Proposition 1.2}). More precisely, there exist constants $c_T, C_T>0$ depending on $T$ such that for all $x \in \mathbb{R}$, $\phi_T(x) \leq C_T\exp\left(-c_Tx^2\right)$. Setting $c_{\beta} = c_{\xi}^{-1}\varepsilon D^{\beta} + R\|K\|_{\infty}D^{-\beta}/2$. We deduce from Equation~\eqref{eq:margin3} that
\begin{equation}\label{eq:margin4}
    \mathbb{P}_{X}\left(\xi(X) \leq \varepsilon \biggm\vert X \in \mathcal{X}_j\right) \leq \int_{-c_{\beta}}^{c_{\beta}}\phi_T(x)dx \leq 2C_Tc_{\beta} = 2C_T\left(c_{\xi}^{-1}\varepsilon D^{\beta} + \dfrac{R\|K\|_{\infty}D^{-\beta}}{2} \right).
\end{equation}
From Equations~\eqref{eq:margin4}, \eqref{eq:margin2} and \eqref{eq:margin1}, we deduce that
\begin{equation*}
   \mathbb{P}_{X}\left(\left|\Phi_{\vec{\sigma}, f}(X) - \dfrac{1}{2}\right| \leq \varepsilon\right) \leq 2C_Tmw\left(c_{\xi}^{-1}\varepsilon D^{\beta} + \dfrac{R\|K\|_{\infty}D^{-\beta}}{2} \right) + C\varepsilon.
\end{equation*}
By Proposition~\ref{prop:margin-condition}, we have $\mathbb{P}_{X}\left(\left|\Phi_{\vec{\sigma}, f}(X) - \dfrac{1}{2}\right| \leq \varepsilon\right) = \mathrm{O}(\varepsilon)$, which implies that $mwD^{\beta} = \mathrm{O}\left(1\right)$ and $mwD^{-\beta} = \mathrm{O}(\varepsilon)$. Thus, low-noise conditions imply that $mw$ cannot tend to zero with a rate that is slower than $D^{-\beta}$.

\item[(iii)] \textbf{Lower bound on the excess risk}

We set $mw = D^{-\beta}$ and deduce from Equation~\eqref{eq:lw1} and Lemma~\ref{lm:Bounds-b-bprime} that 
\begin{equation*}
 \underset{\w{\bf f}}{\inf}\underset{{\bf f^*} \in {\bf F}^{M}}{\sup}\mathbb{E}_{\mathbb{P}^{\otimes N}}\left[R(g_{\w{\bf f}}) - R(g_{\bf f^*})\right] \geq c^{\prime}\dfrac{1 - cD^{-\beta}\sqrt{Nw}}{2}D^{-2\beta}.
\end{equation*}
\end{enumerate}
To conclude the proof, we remark that since $mw = D^{-\beta}$ and $D = \left\lfloor N^{1/(2\beta+1)} \right\rfloor$, we have $D^{-\beta}\sqrt{Nw} \rightarrow 0$ as $N$ tends to infinity. Finally, there exists a $c>0$ depending on $\beta$ such that
\begin{equation*}
 \underset{\w{\bf f}}{\inf}\underset{{\bf f^*} \in {\bf F}^{M}}{\sup}\mathbb{E}_{\mathbb{P}^{\otimes N}}\left[R(g_{\w{\bf f}}) - R(g_{\bf f^*})\right] \geq cN^{-2\beta/(2\beta+1)}.
\end{equation*}
\end{proof}

\section*{Declarations}

\paragraph{Conflict of interest} I have no conflict of interest to declare that is relevant to the content of this article. No
funding was received to assist with the preparation of this document.

\bibliographystyle{ScandJStat}
\bibliography{mabiblio.bib}

@book{anderson1958introduction,
  title={An introduction to multivariate statistical analysis},
  author={Anderson, Theodore Wilbur},
  volume={2},
  year={1958},
  publisher={Wiley New York}
}

@inproceedings{Audibert2004ClassificationUP,
  title={Classification under polynomial entropy and margin assump-tions and randomized estimators},
  author={Jean-Yves Audibert},
  booktitle={Preprint, Laboratoire de Probabilit\'es et Mod\`eles Al\'eatoires, Univ. Paris VI and VII},
  volume={908},
  year={2004},
  url={https://imagine.enpc.fr/publications/papers/04PMA-908.pdf}
}

@article{audibert2007fast,
  title={Fast learning rates for plug-in classifiers},
  author={Audibert, J.-Y. and Tsybakov, A.-B. and others},
  journal={Ann. Statist.},
  volume={35},
  number={2},
  pages={608--633},
  year={2007},
  publisher={Institute of Mathematical Statistics}
}

@article{baillo2011classification,
  title={Classification methods for functional data},
  author={Ba{\'\i}llo, A. and Cuevas, A. and Fraiman, R.},
  journal={The Oxford handbook of functional data analysis},
  year={2011}
}

@article{bartlett2006convexity,
  title={Convexity, classification, and risk bounds},
  author={Bartlett, Peter L and Jordan, Michael I and McAuliffe, Jon D},
  journal={Journal of the American Statistical Association},
  volume={101},
  number={473},
  pages={138--156},
  year={2006},
  publisher={Taylor \& Francis}
}

@article{cadre2013supervised,
  title={Supervised classification of diffusion paths},
  author={Cadre, B.},
  journal={Math. Methods Statist.},
  volume={22},
  number={3},
  pages={213--225},
  year={2013},
  publisher={Springer}
}

@book{comte2017nonparametric,
  title={Nonparametric Estimation},
  author={Comte, F.},
  year={2017},
  publisher={Spartacus IDH}
}

@article{comte2020nonparametric,
  title={Nonparametric drift estimation for i.i.d. paths of stochastic differential equations},
  author={Comte, F. and Genon-Catalot, V.},
  journal={Ann. Statist.},
  volume={48},
  number={6},
  pages={3336--3365},
  year={2020},
  publisher={Institute of Mathematical Statistics}
}

@article{cover1967nearest,
  title={Nearest neighbor pattern classification},
  author={Cover, Thomas and Hart, Peter},
  journal={IEEE transactions on information theory},
  volume={13},
  number={1},
  pages={21--27},
  year={1967},
  publisher={IEEE}
}

@book{crow2017introduction,
  title={An introduction to population genetics theory},
  author={Crow, James Franklin},
  year={2017},
  publisher={Scientific Publishers}
}

@article{cuevas2007robust,
  title={Robust estimation and classification for functional data via projection-based depth notions},
  author={Cuevas, A. and Febrero, M. and Fraiman, R.},
  journal={Comput. Statist.},
  volume={22},
  number={3},
  pages={481--496},
  year={2007},
  publisher={Springer}
}

@article{denis2020classif,
  title={Consistent procedures for multiclass classification of discrete diffusion paths},
  author={Denis, C. and Dion-Blanc, C. and Martinez, M.},
  journal={Scand. J. Stat.},
  volume={47},
  number={2},
  pages={516--554},
  year={2020}
}

@article{denis2021ridge,
  title={A ridge estimator of the drift from discrete repeated observations of the solution of a stochastic differential equation},
  author={Denis, Christophe and Dion-Blanc, Charlotte and Martinez, Miguel},
  journal={Bernoulli},
  volume={27},
  number={4},
  pages={2675--2713},
  year={2021},
  publisher={Bernoulli Society for Mathematical Statistics and Probability}
}

@article{dacunha1986estimation,
  title={Estimation of the coefficients of a diffusion from discrete observations},
  author={Dacunha-Castelle, Didier and Florens-Zmirou, Danielle},
  journal={Stochastics: An International Journal of Probability and Stochastic Processes},
  volume={19},
  number={4},
  pages={263--284},
  year={1986},
  publisher={Taylor \& Francis}
}

@article{delaigle2012achieving,
  title={Achieving near perfect classification for functional data},
  author={Delaigle, Aurore and Hall, Peter},
  journal={Journal of the Royal Statistical Society Series B: Statistical Methodology},
  volume={74},
  number={2},
  pages={267--286},
  year={2012},
  publisher={Oxford University Press}
}

@article{denis2024nonparametric,
  title={Nonparametric plug-in classifier for multiclass classification of {SDE} paths},
  author={Denis, Christophe and Dion-Blanc, Charlotte and Ella-Mintsa, Eddy and Tran, Viet Chi},
  journal={Scandinavian Journal of Statistics},
  volume={51},
  number={3},
  pages={1103--1160},
  year={2024},
  publisher={Wiley Online Library}
}

@article{denis2025empirical,
  title={Empirical risk minimization algorithm for multiclass classification of {SDE} paths},
  author={Denis, Christophe and Ella-Mintsa, Eddy-Michel},
  journal={arXiv preprint arXiv:2503.14045},
  year={2025}
}

@book{devroye2013probabilistic,
  title={A probabilistic theory of pattern recognition},
  author={Devroye, Luc and Gy{\"o}rfi, L{\'a}szl{\'o} and Lugosi, G{\'a}bor},
  volume={31},
  year={2013},
  publisher={Springer Science \& Business Media}
}

@article{devroye1980distribution,
  title={Distribution-free consistency results in nonparametric discrimination and regression function estimation},
  author={Devroye, Luc P and Wagner, TJ},
  journal={The Annals of Statistics},
  pages={231--239},
  year={1980},
  publisher={JSTOR}
}

@article{el1997backward,
  title={Backward stochastic differential equations in finance},
  author={El Karoui, N. and Peng, S. and Quenez, M. -C.},
  journal={Math. Finance},
  volume={7},
  number={1},
  pages={1--71},
  year={1997},
  publisher={Wiley Online Library}
}

@article{ella2024nonparametric,
  title={Nonparametric estimation of the diffusion coefficient from iid {SDE} paths},
  author={Ella-Mintsa, Eddy},
  journal={Statistical Inference for Stochastic Processes},
  volume={27},
  number={3},
  pages={585--640},
  year={2024},
  publisher={Springer}
}

@article{ella2025minimax,
  title={Minimax rates of convergence for the nonparametric estimation of the diffusion coefficient from time-homogeneous SDE paths},
  author={Ella-Mintsa, Eddy-Michel},
  journal={Statistical Inference for Stochastic Processes},
  volume={28},
  number={3},
  pages={17},
  year={2025},
  publisher={Springer}
}

@article{ella2026minimax,
  title={Minimax convergence rates of a binary classification procedure for time-homogeneous SDE paths},
  author={Ella-Mintsa, Eddy-Michel},
  journal={Metrika},
  pages={1--48},
  year={2026},
  publisher={Springer}
}

@article{ferraty2003curves,
  title={Curves discrimination: a nonparametric functional approach},
  author={Ferraty, Fr{\'e}d{\'e}ric and Vieu, Philippe},
  journal={Computational Statistics \& Data Analysis},
  volume={44},
  number={1-2},
  pages={161--173},
  year={2003},
  publisher={Elsevier}
}

@article{fisher1936use,
  title={The use of multiple measurements in taxonomic problems},
  author={Fisher, Ronald A},
  journal={Annals of eugenics},
  volume={7},
  number={2},
  pages={179--188},
  year={1936},
  publisher={Wiley Online Library}
}

@article{gadat2020optimal,
  title={Optimal functional supervised classification with separation condition},
  author={Gadat, S. and Gerchinovitz, S. and Marteau, C.},
  journal={Bernoulli},
  volume={26},
  number={3},
  pages={1797--1831},
  year={2020},
  publisher={Bernoulli Society for Mathematical Statistics and Probability}
}

@article{gobet2002lan,
  title={LAN property for ergodic diffusions with discrete observations},
  author={Gobet, E.},
  journal={Ann. Inst. Henri Poincar\'e Probab. Stat.},
  volume={38},
  number={5},
  pages={711--737},
  year={2002},
  publisher={Elsevier}
}

@book{gyorfi2006distribution,
  title={A distribution-free theory of nonparametric regression},
  author={Gy{\"o}rfi, L. and Kohler, M. and Krzyzak, A. and Walk, H.},
  year={2006},
  publisher={Springer Science \& Business Media}
}

@book{hardle2012wavelets,
  title={Wavelets, approximation, and statistical applications},
  author={H{\"a}rdle, W. and Kerkyacharian, G. and Picard, D. and Tsybakov, A.},
  volume={129},
  year={2012},
  publisher={Springer Science \& Business Media}
}

@article{hyndman2009forecasting,
  title={Forecasting functional time series},
  author={Hyndman, Rob J and Shang, Han Lin},
  journal={Journal of the Korean Statistical Society},
  volume={38},
  number={3},
  pages={199--211},
  year={2009},
  publisher={Elsevier}
}

@article{james2001functional,
  title={Functional linear discriminant analysis for irregularly sampled curves},
  author={James, Gareth M and Hastie, Trevor J},
  journal={Journal of the Royal Statistical Society Series B: Statistical Methodology},
  volume={63},
  number={3},
  pages={533--550},
  year={2001},
  publisher={Oxford University Press}
}

@book{karatzas2014brownian,
	title={Brownian motion and stochastic calculus},
	author={Karatzas, I. and Shreve, S.},
	volume={113},
	year={2014},
	publisher={springer}
}

@book{kazamaki2006continuous,
  title={Continuous exponential martingales and BMO},
  author={Kazamaki, Norihiko},
  year={2006},
  publisher={Springer}
}

@book{lamberton2011introduction,
  title={Introduction to stochastic calculus applied to finance},
  author={Lamberton, Damien and Lapeyre, Bernard},
  year={2011},
  publisher={Chapman and Hall/CRC}
}

@article{mammen1999smooth,
  title={Smooth discrimination analysis},
  author={Mammen, Enno and Tsybakov, Alexandre B},
  journal={The Annals of Statistics},
  volume={27},
  number={6},
  pages={1808--1829},
  year={1999},
  publisher={Institute of Mathematical Statistics}
}

@article{marie2023nadaraya,
  title={Nadaraya--Watson estimator for IID paths of diffusion processes},
  author={Marie, Nicolas and Rosier, Am{\'e}lie},
  journal={Scandinavian Journal of Statistics},
  volume={50},
  number={2},
  pages={589--637},
  year={2023},
  publisher={Wiley Online Library}
}

@book{massart2007concentration,
  title={Concentration inequalities and model selection: Ecole d'Et{\'e} de Probabilit{\'e}s de Saint-Flour XXXIII-2003},
  author={Massart, Pascal},
  year={2007},
  publisher={Springer}
}

@book{mclachlan2005discriminant,
  title={Discriminant analysis and statistical pattern recognition},
  author={McLachlan, Geoffrey J},
  year={2005},
  publisher={John Wiley \& Sons}
}

@article{nagai1983asymptotic,
  title={Asymptotic behavior for a nonlinear degenerate diffusion equation in population dynamics},
  author={Nagai, Toshitaka and Mimura, Masayasu},
  journal={SIAM Journal on Applied Mathematics},
  volume={43},
  number={3},
  pages={449--464},
  year={1983},
  publisher={SIAM}
}

@book{nualart2006malliavin,
  title={The Malliavin calculus and related topics},
  author={Nualart, David},
  year={2006},
  publisher={Springer}
}

@book{ramsay2005fitting,
  title={Fitting differential equations to functional data: Principal differential analysis},
  author={Ramsay, J.-O. and Silverman, B.-W.},
  year={2005},
  publisher={Springer}
}

@article{rao1948utilization,
  title={The utilization of multiple measurements in problems of biological classification},
  author={Rao, C Radhakrishna},
  journal={Journal of the Royal Statistical Society. Series B (Methodological)},
  volume={10},
  number={2},
  pages={159--203},
  year={1948},
  publisher={JSTOR}
}

@book {revuzyor1999,
    AUTHOR = {Revuz, D. and Yor, M.},
     TITLE = {Continuous martingales and {B}rownian motion},
    SERIES = {Grundlehren der mathematischen Wissenschaften},
    VOLUME = {293},
   EDITION = {Third},
 PUBLISHER = {Springer-Verlag, Berlin},
      YEAR = {1999}
}

@book{tsybakov2008introduction,
  title={Introduction to nonparametric estimation},
  author={Tsybakov, A.-B.},
  year={2008},
  publisher={Springer Science \& Business Media}
}

@inproceedings{tukey1975mathematics,
  title={Mathematics and the picturing of data},
  author={Tukey, John W},
  booktitle={Proceedings of the international congress of mathematicians},
  volume={2},
  pages={523--531},
  year={1975},
  organization={Vancouver}
}

@article{van1995exponential,
  title={Exponential inequalities for martingales, with application to maximum likelihood estimation for counting processes},
  author={Van-de-Geer, S.},
  journal={Ann. Statist.},
  pages={1779--1801},
  year={1995},
  publisher={JSTOR}
}

@article{wang2023deep,
  title={Deep neural network classifier for multidimensional functional data},
  author={Wang, Shuoyang and Cao, Guanqun and Shang, Zuofeng and Alzheimer's Disease Neuroimaging Initiative},
  journal={Scandinavian Journal of Statistics},
  volume={50},
  number={4},
  pages={1667--1686},
  year={2023},
  publisher={Wiley Online Library}
}

@article{wang2024functional,
  title={Functional data analysis using deep neural networks},
  author={Wang, Shuoyang and Zhang, Wanyu and Cao, Guanqun and Huang, Yuan},
  journal={Wiley Interdisciplinary Reviews: Computational Statistics},
  volume={16},
  number={4},
  pages={e70001},
  year={2024},
  publisher={Wiley Online Library}
}

@article{Yang99,
  Author = {Yang, Y.},
  Journal = {IEEE Transactions on Information Theory},
  Number = {7},
  Pages = {2271-2284},
  Title = {Minimax nonparametric classification: Rates of convergence},
  Volume = {45},
  Year = {1999}}

\section*{Appendix}

\subsection*{Proof of Lemma~\ref{lm:bounded-density}}

\begin{proof}
The proof of the Lemma relies on Malliavin calculus. First, we have, on the one hand
\begin{equation}\label{eq:case1}
    \mathbb{E}_X\left[Z_T^2\right] = \mathbb{E}_X\left[\int_{0}^{T}(b_1^* - b_0^*)^2(X_s)ds\right] \leq 2T\left(\left\|b_0^*\right\|_{\infty}^2 + \left\|b_1^*\right\|_{\infty}^2\right) < \infty.
\end{equation} 
 On the other hand, under Assumptions~\ref{ass:Reg} and \ref{ass:hormander}, since $\mu\left(\mathfrak{B}_{\bf b^*}\right) > 0$, there exists a compact interval $I \subset \mathfrak{B}_{\bf b^*}$ such that 
$\mu(I) > 0$ and $\mathrm{and} ~~ \underset{x \in I}{\inf}\left|b_1^*(x) - b_0^*(x)\right| > 0$, and we obtain
\begin{equation}\label{eq:case2}
    \begin{aligned}
    \mathbb{E}_X\left[Z_T^2\right] \geq &~ \mathbb{E}_X\left[\int_{0}^{T}(b_1^* - b_0^*)^2(X_s)\mathds{1}_{X_s \in I}ds\right] \geq \underset{x \in I}{\inf}\left|b_1^*(x) - b_0^*(x)\right|^2\int_{0}^{T}\mathbb{P}_X\left(X_s \in I\right)ds > 0.
\end{aligned}
\end{equation}
Second, under Assumption~\ref{ass:Reg}, using the Malliavin derivative operator $D$ on the random variable $Z_T$, we obtain from \cite{nualart2006malliavin}, Chapter 2, Exercise 2.2.1, p.124, that for all $t \in [0,T]$,
\begin{equation}\label{eq:case3}
    D_tZ_T = (b_1^* - b_0^*)(X_t)\exp\left(\int_{0}^{T}(b_1^{*\prime} - b_0^{*\prime})(X_s)dW_s - \dfrac{1}{2}\int_{0}^{T}(b_1^{*\prime} - b_0^{*\prime})^2(X_s)ds\right),
\end{equation}
and under Assumption~\ref{ass:Reg} and from \cite{kazamaki2006continuous}, Chapter 1, Theorems 1.2 and 1.6, p.3-9,  
\begin{equation}\label{eq:case4}
    \begin{aligned}
    & \mathbb{E}_X\left[\int_{0}^{T}|D_tZ_T|^2dt\right] \\
    & = \mathbb{E}_X\left[\int_{0}^{T}(b_1^* - b_0^*)^2(X_t)\exp\left(\int_{0}^{T}2(b_1^{*\prime} - b_0^{*\prime})(X_s)dW_s - \int_{0}^{T}(b_1^{*\prime} - b_0^{*\prime})^2(X_s)ds\right)dt\right]\\
    &\leq 2T\left(\left\|b_0^*\right\|_{\infty}^2 + \left\|b_1^*\right\|_{\infty}^2\right)\exp\left(2T\left[\left\|b_0^{*\prime}\right\|_{\infty}^2 + \left\|b_1^{*\prime}\right\|_{\infty}^2\right]\right) < \infty.
\end{aligned}
\end{equation}
It remains to verify that $\int_{0}^{T}|D_tZ_T|^2dt > 0 ~ a.s.$ To this end, note that from Equation~\eqref{eq:case3},
\begin{equation*}
    D_0Z_T = (b_1^* - b_0^*)(x_0)\exp\left(\int_{0}^{T}(b_1^{*\prime} - b_0^{*\prime})(X_s)dW_s - \dfrac{1}{2}\int_{0}^{T}(b_1^{*\prime} - b_0^{*\prime})^2(X_s)ds\right).
\end{equation*}
Thus, reasoning by contradiction, we have
\begin{align*}
    \int_{0}^{T}|D_tZ_T|^2dt = 0 ~~ a.s. \iff &~ \mathbb{E}_X\left[\int_{0}^{T}|D_tZ_T|^2dt\right] = 0\\
    \iff &~ \forall ~ t \in [0,T], ~~ \mathbb{E}_X(|D_tZ_T|) = 0\\
    \Longrightarrow &~  \mathbb{E}_X(|D_0Z_T|) = \left|(b_1^* - b_0^*)(x_0)\right| = 0\\
    \Longrightarrow &~ b_0^*(x_0) = b_1^*(x_0).
\end{align*}
The last equality is a contradiction under Assumption~\ref{ass:hormander}. We deduce that 
\begin{equation}\label{eq:case5}
    \int_{0}^{T}|D_tZ_T|^2dt > 0 ~~ a.s. 
\end{equation}
Then, from Equations~\eqref{eq:case1} and \eqref{eq:case3}, $Z_T$ belongs to the domain of the Malliavin derivative operator $D$ denoted by $\mathbb{D}^{1,2}$ (see \cite{nualart2006malliavin}, Chapter 1, p.27). In addition, setting $DZ_T = (D_tZ_T)_{t \in [0,T]}$, $DZ_T\left(\int_{0}^{T}\left|D_tZ_T\right|^2dt\right)^{-1}$ is well defined and belongs to the domain $\mathrm{Dom}(\delta)$ of the divergence operator $\delta$, adjoint operator of $D$, since from Equations~\eqref{eq:case2} and \eqref{eq:case5}, there exists a constant $c>0$ such that
\begin{equation*}
    \left|\int_{0}^{T}D_tZ_T \times D_tZ_T\left(\int_{0}^{T}\left|D_tZ_T\right|^2dt\right)^{-1}dt\right| = 1 \leq c\left\|Z_T\right\|_{2},
\end{equation*}
where $\|.\|_{2}$ is the norm on $\mathbb{D}^{1,2}$ and $\left\|Z_T\right\|_{2}^2 := \mathbb{E}_X\left[|Z_T|^2\right] + \mathbb{E}_X\left[\int_{0}^{T}\left|D_tZ_T\right|^2dt\right] > 0$ 
(see \cite{nualart2006malliavin}, Chapter 1, Definition 1.3.1, p.36-37). We finally conclude from \cite{nualart2006malliavin}, Chapter 2, Proposition 2.1.1, p.86, that the random variable $Z_T$ has a continuous and bounded density. 
\end{proof}

\subsection*{Proof of Proposition~\ref{prop:margin-condition}}

\begin{proof}
Recall that $\Phi_{\bf f^*}(X) = \dfrac{p_1^*\exp\left(F_{\bf b^*}^1(X)\right)}{p_0^*\exp\left(F_{\bf b^*}^0(X) + p_1^*\exp\left(F_{\bf b^*}^1(X)\right)\right)}$. For all $\varepsilon \in (0, 1/8)$,
\begin{equation}\label{eq:margin0}
    \begin{aligned}
        &~ \mathbb{P}_{X}\left(0 < \left|\Phi_{\bf f^*}(X) - \dfrac{1}{2}\right| \leq \varepsilon\right)\\
        &~ = \mathbb{P}_{X}\left(\left|p_1^*\exp\left(F_{\bf b^*}^1(X)\right) - p_0^*\exp\left(F_{\bf b^*}^0(X)\right)\right| \leq 2\varepsilon \left[p_1^*\exp\left(F_{\bf b^*}^1(X)\right) + p_0^*\exp\left(F_{\bf b^*}^0(X)\right)\right]\right)\\
        &~ = T_{1, \bf f^*} + T_{2, \bf f^*},
    \end{aligned}
\end{equation}
with,
\begin{align*}
    T_{1, \bf f^*} = &~ \mathbb{P}_{X}\left(\left\{\exp\left(F_{\bf b^*}^1(X) - F_{\bf b^*}^0(X)\right) \leq \dfrac{1+2\varepsilon}{1-2\varepsilon} \times \dfrac{p_0^*}{p_1^*}\right\} \cap \left\{\exp\left(F_{\bf b^*}^1(X) - F_{\bf b^*}^1(X)\right) \geq \dfrac{p_0^*}{p_1^*}\right\}\right),\\
    T_{2, \bf f^*} = &~ \mathbb{P}_{X}\left(\left\{\exp\left(F_{\bf b^*}^1(X) - F_{\bf b^*}^0(X)\right) \geq \dfrac{1-2\varepsilon}{1+2\varepsilon} \times \dfrac{p_0^*}{p_1^*}\right\} \cap \left\{\exp\left(F_{\bf b^*}^1(X) - F_{\bf b^*}^0(X)\right) < \dfrac{p_0^*}{p_1^*}\right\}\right).
\end{align*}

\subsection*{(1) Upper-bound on $T_{1, \bf f^*}$}

For all $\varepsilon \in (0,1/8)$, we have the following:
\begin{equation}\label{eq:T1-bound0}
    \begin{aligned}
        T_{1, \bf f^*} = &~ \mathbb{P}_{X}\left(\log\left(\dfrac{p_0^*}{p_1^*}\right) \leq F_{\bf b^*}^1(X) - F_{\bf b^*}^0(X) \leq \log\left(\dfrac{1+2\varepsilon}{1-2\varepsilon} \times \dfrac{p_0^*}{p_1^*}\right)\right),
    \end{aligned}
\end{equation}
Set $\phi_{1, \bf b^*} = b_1^* - b_0^*, ~~ \phi_{2, \bf b^*} = b_1^{*2} - b_0^{*2}$. We obtain 
$$F_{\bf b^*}^1(X) - F_{\bf b^*}^0(X) =  \int_{0}^{T}{\phi_{1,\bf b^*}^*(X_s)dW_s} + \frac{1}{2}\int_{0}^{T}{\phi_{2,\bf b^*}^*(X_s)ds}.$$ 
Under Assumption~\ref{ass:Reg},  there exist constants $\underline{C}$ and $\overline{C}$ such that $ \underline{C} \leq \frac{1}{2}\int_{0}^{T}{\phi_{2, \bf b^*}^*(X_s)ds}\leq \overline{C}$. Consider the subdivision $\{t_k = \underline{C} + k(\overline{C} - \underline{C})/m, ~ k = 0, \ldots, m\}$ of the compact interval $[\underline{C}, \overline{C}]$ with $m \rightarrow \infty$. Set $V_{\bf b^*}(X) = \frac{1}{2}\int_{0}^{T}{\phi_2^*(X_s)ds}$. We deduce from Equation~\eqref{eq:T1-bound0} that for all $\varepsilon \in (0, 1/8)$,
\begin{align*}
     & T_{1, \bf f^*} = \sum_{k=1}^{m}{\mathbb{P}_{X}\left(A_k\leq \int_{0}^{T}{\phi_{1, \bf b^*}(X_s)dW_s} \leq B_k \biggm\vert V_{\bf b^*}(X) \in [t_k, t_{k+1}]\right)\mathbb{P}_{X}\left(V_{\bf b^*}(X) \in [t_k, t_{k+1}]\right)},
\end{align*}
where for all $k \in \{1, \ldots, m\}$, $ A_k = \log\left(\frac{p_0^*}{p_1^*}\right) - t_{k+1}, ~~ B_k = \log\left(\frac{1+2\varepsilon}{1-2\varepsilon} \times \frac{p_0^*}{p_1^*}\right) - t_k, ~~ \varepsilon \in (0, 1/8)$.
By Lemma~\ref{lm:bounded-density} and the assumptions therein, the random variable $\int_{0}^{T}{\phi_{1, \bf b^*}(X_s)dW_s}$ has a continuous and bounded density function $\Gamma$. Then, for all ${\bf f^*} \in {\bf F} (\beta, R)$, we obtain the following.
\begin{equation}\label{eq:T1-bound2}
    \begin{aligned}
        T_{1, \bf f^*} = &~ \sum_{k=1}^{m}{\left(\int_{A_k}^{B_k}{\Gamma(x)dx}\right)\mathbb{P}_{X}\left(V_{\bf b^*}(X) \in [t_k, t_{k+1}]\right)} \leq \left\|\Gamma\right\|_{\infty}\sum_{k=1}^{m}{(B_k - A_k)\mathbb{P}_{X}\left(V_{\bf b^*}(X) \in [t_k, t_{k+1}]\right)},
    \end{aligned}
\end{equation}
For $k = 1, \ldots, m$ and for all $\varepsilon \in (0, 1/8)$, we have $B_k - A_k \geq 0$ and
\begin{align*}
    B_k - A_k = &~ t_{k+1} - t_k + \log\left(\dfrac{1 + 2\varepsilon}{1 - 2\varepsilon} \times \dfrac{p_0^*}{p_1^*}\right) - \log\left(\dfrac{p_0^*}{p_1^*}\right) \leq \dfrac{\overline{C} - \underline{C}}{m} + 8\varepsilon,
\end{align*}
and since $m \rightarrow \infty$, we obtain for all $m \geq \left\lceil \left(\overline{C} - \underline{C}\right)/\varepsilon \right\rceil$, $0 \leq B_k - A_k \leq 5\varepsilon, ~~ k = 1, \ldots, m$.
We finally, choosing $m$ such that $m \geq \lfloor 1/\varepsilon \rfloor$ obtain from Equation~\eqref{eq:T1-bound2} that 
\begin{equation}\label{eq:T1-f}
    \begin{aligned}
        T_{1, \bf f^*} = &~ \mathcal{O}(\varepsilon).
    \end{aligned}
\end{equation}

\subsection*{Upper-bound on $T_{2, \bf f^*}$}

Using a similar reasoning as in the previous case, for all $\varepsilon \in (0, 1/8)$, we obtain
\begin{align*}
     T_{2, \bf f^*} = \sum_{k=1}^{m}{\mathbb{P}_{X}\left(A_k^{\prime} \leq \int_{0}^{T}{\phi_{1, \bf b^*}(X_s)dW_s} \leq B_k^{\prime}\biggm\vert V_{\bf b^*}(X) \in [t_k, t_{k+1}]\right)\mathbb{P}_{X}\left(V_{\bf b^*}(X) \in [t_k, t_{k+1}]\right)},
\end{align*}
where for all $k \in \{1, \ldots, m\}, ~ A_k^{\prime} = \log\left(\frac{1-2\varepsilon}{1+2\varepsilon} \times \frac{p_0^*}{p_1^*}\right) - t_{k+1}, ~~ B_k^{\prime} = \log\left(\frac{p_0^*}{p_1^*}\right) - t_k.$
Since for all $k \in \{1, \ldots, m\}$ and for all $\varepsilon \in (0,1/8)$, we have $B_k^{\prime} - A_k^{\prime} \geq 0$ and $B_k^{\prime} - A_k^{\prime} \leq \frac{\overline{C} - \underline{C}}{m} + 8\varepsilon$.
Then for $m \geq \left\lceil \left(\overline{C} - \underline{C}\right)/\varepsilon \right\rceil$, we obtain:
\begin{equation}\label{eq:T2-f}
    \begin{aligned}
        T_{2, \bf f^*} = \sum_{k=1}^{m}{\left(\int_{A_k^{\prime}}^{B_k^{\prime}}{\Gamma(x)dx}\right)\mathbb{P}_{X}\left(V_{\bf b^*}(X) \in [t_k, t_{k+1}]\right)} \leq \left\|\Gamma\right\|_{\infty}\left(\dfrac{\overline{C} - \underline{C}}{m} + 8\varepsilon\right) = \mathcal{O}\left(\varepsilon\right)
    \end{aligned}
\end{equation}
The final result is deduced from Equations~\eqref{eq:T2-f}, \eqref{eq:T1-f} and \eqref{eq:margin0}.
\end{proof}

\subsection*{Proof of Lemma~\ref{lm:bias-1}}

\begin{proof}
Set $\pi_{i,k}(x) = \Gamma_{X|Y=i}(0,s_k^n,x_0,x), ~ \forall x \in \mathbb{R}$. For each $i \in \mathcal{Y}$, for all $j \in [\![1,N]\!]$ and $k \in [\![k_0,n-1]\!]$, 
    \begin{equation}\label{eq:bias-f-1}
        \mathbb{E}_{\mathbb{P}_i^{\otimes N}}\left[K_{h_{i,N}}(X_{s_k^n}^{ji} - x)\right] - \Gamma_{X|Y=i}(0,s_k^n,x_0,x) = \int_{-\infty}^{+\infty}K(v)\left[\pi_{i,k}(x + vh_{i,N}) - \pi_{i,k}(x)\right]dv.
    \end{equation}
    Since $\pi_{i,k} \in \mathcal{C}^{\infty}(\mathbb{R})$, from the Taylor-Lagrange formula up to order $\lfloor \beta \rfloor$, we obtain for all $x,v \in \mathbb{R}$,
    \begin{align*}
        \pi_{i,k}(x + vh_{i,N}) = \sum_{\ell = 0}^{\lfloor \beta \rfloor}\dfrac{(vh_{i,N})^{\ell}}{\ell!}\pi_{i,k}^{(\ell)}(x) + \dfrac{(vh_{i,N})^{\lfloor \beta \rfloor + 1}}{(\lfloor \beta \rfloor + 1)!}\pi_{i,k}^{(\lfloor \beta \rfloor + 1)}(x + vh_{i,N}\xi),
    \end{align*}
    where $\xi \in (0,1)$. Then, under Assumption~\ref{ass:prop-kernel} with $\gamma = \lfloor \beta \rfloor + 1$, we obtain
    \begin{equation*}
        \left|\int_{-\infty}^{+\infty}K(v)\left[\pi_{i,k}(x + vh_{i,N}) - \pi_{i,k}(x)\right]dv\right| = \left|\int_{-\infty}^{+\infty}K(v)\dfrac{(vh_{i,N})^{\lfloor \beta \rfloor + 1}}{(\lfloor \beta \rfloor + 1)!}\pi_{i,k}^{(\lfloor \beta \rfloor + 1)}(x + vh_{i,N}\xi)dv\right|.
    \end{equation*}
    From \cite{dacunha1986estimation}, \textit{Lemma 2}, for all $k \in [\![k_0,n-1]\!]$, we have
    \begin{equation*}
        \pi_{i,k}(x) = \Gamma_{X|Y=i}(0,s_k^n,x_0,x) = \dfrac{\Lambda(0,s_k^n,x_0,x)}{2\sqrt{2\pi t}}\exp\left(-\dfrac{(x - x_0)^2}{2t} + \int_{0}^{x}{b_i^*(u)du}\right),
    \end{equation*}
    where $\Lambda(0,s_k^n,x_0,x) = \widetilde{\mathbf{E}}\left[\exp\left(t_k\int_{0}^{T}{G((1-u)x + uy + \sqrt{s_k^n}B_u)du}\right)\right]$ and $G = -(b_i^{*2} + b_i^*)/2$. Since for each $t \in [t_0, T]$, the function $x \mapsto \Gamma_{X|Y=i}(0,t,x_0,x)$ is $\mathcal{C}^{\lfloor \beta \rfloor + 1}$, there exist constants $C_{\lfloor \beta \rfloor}, c_{\lfloor \beta \rfloor}, \alpha_{\lfloor \beta \rfloor} > 0$ depending on $\lfloor \beta \rfloor$ such that for all $k \in [\![k_0,n-1]\!]$,
    \begin{equation*}
        \forall ~ x \in \mathbb{R}, ~~ \left|\pi_{i,k}^{(\lfloor \beta\rfloor + 1)}(x)\right| = \left|\partial_{x}^{\lfloor \beta \rfloor + 1}\Gamma_{X|Y=i}(0,s_k^n,x_0,x)\right| \leq \dfrac{C_{\lfloor \beta \rfloor}}{(s_k^n)^{\alpha_{\lfloor \beta \rfloor}}}\exp\left(-c_{\lfloor \beta \rfloor}\dfrac{(x - x_0)^2}{s_k^n}\right) \leq \dfrac{C_{\lfloor \beta \rfloor}}{t_0^{\alpha_{\lfloor \beta \rfloor}}}.
    \end{equation*}
    We deduce that 
    \begin{equation}\label{eq:bias-f-2}
        \left|\int_{-\infty}^{+\infty}K(v)\left[\pi_{i,k}(x + vh_{i,N}) - \pi_{i,k}(x)\right]dv\right| \leq \dfrac{C_{\lfloor \beta \rfloor}}{t_0^{\alpha_{\lfloor \beta \rfloor}}}\dfrac{h_{i,N}^{\lfloor \beta \rfloor + 1}}{(\lfloor \beta \rfloor + 1)!}\int_{-\infty}^{+\infty}|v|^{\lfloor \beta \rfloor + 1}|K(v)|dv.
    \end{equation}
    The final result is deduced from Assumption~\ref{ass:prop-kernel} and Equations~\eqref{eq:bias-f-2} and \eqref{eq:bias-f-1}.
\end{proof}

\subsection*{Proof of Lemma~\ref{lm:bias-2}}

\begin{proof}
    For all $(i,j,k) \in \mathcal{Y} \times [\![1,N]\!] \times [\![k_0,n-1]\!]$, we have
    \begin{multline}\label{eq:bf-1}
        \left|\sum_{k=k_0}^{n-1}\mathbb{E}_{\mathbb{P}_i^{\otimes N}}\left[K_{h_{i,N}^{\prime}}(X_{s_k^n}^{ji} - x)\int_{s_k^n}^{s_{k+1}^n}b_i^*(X_{u}^{ji})du\right] - (T-t_0)(b\zeta)_{i,\Delta_n}^*(x)\right|\\
        \leq \left|\sum_{k=k_0}^{n-1}\left(s_{k+1}^n - s_k^n\right)\mathbb{E}_{\mathbb{P}_i^{\otimes N}}\left[K_{h_{i,N}^{\prime}}(X_{s_k^n}^{ji} - x)b_i^*(X_{s_k^n}^{ji})\right] - (T-t_0)(b\zeta)_{i,\Delta_n}^*(x)\right|\\
        + \sum_{k=k_0}^{n-1}\mathbb{E}_{\mathbb{P}_i^{\otimes N}}\left[\left|K_{h_{i,N}^{\prime}}(X_{s_k^n}^{ji} - x)\right|\int_{s_k^n}^{s_{k+1}^n}\left|b_i^*(X_{u}^{ji}) - b_i^*(X_{s_k^n}^{ji})\right|du\right].
    \end{multline}
    Using Cauchy-Schwarz's, from Assumptions~\ref{ass:Reg} and \ref{ass:prop-kernel} and \cite{denis2024nonparametric}, \textit{Lemma 2}, there exists a constant $C>0$ depending on $\|K\|$ such that
    \begin{multline}\label{eq:bf-2}
        \sum_{k=k_0}^{n-1}\mathbb{E}_{\mathbb{P}_i^{\otimes N}}\left[\left|K_{h_{i,N}^{\prime}}(X_{s_k^n}^{ji} - x)\right|\int_{s_k^n}^{s_{k+1}^n}\left|b_i^*(X_{u}^{ji}) - b_i^*(X_{s_k^n}^{ji})\right|ds\right]\\
        \leq \sum_{k=k_0}^{n-1}\left(\mathbb{E}_{\mathbb{P}_i^{\otimes N}}\left[\left|K_{h_{i,N}^{\prime}}(X_{s_k^n}^{ji} - x)\right|^2\right]\right)^{1/2}\left(\Delta_n\int_{s_k^n}^{s_{k+1}^n}\mathbb{E}\left[\left|b_i^*(X_{u}^{ji}) - b_i^*(X_{s_k^n}^{ji})\right|^2\right]du\right)^{1/2} \leq C\sqrt{\Delta_n}.
    \end{multline}
    On the other hand, for all $x \in \mathrm{Supp}(b_i^*)$, we have
    \begin{multline}\label{eq:bias-bf-1}
        \left|\sum_{k=k_0}^{n-1}\left(s_{k+1}^n - s_k^n\right)\mathbb{E}_{\mathbb{P}_i^{\otimes N}}\left[K_{h_{i,N}^{\prime}}(X_{s_k^n}^{ji} - x)b_i^*(X_{s_k^n}^{ji})\right] - (T-t_0)(b\zeta)_{i,\Delta_n}^*(x)\right|\\
        = (T-t_0)\int_{-\infty}^{+\infty}K(v)\left[(b^*f)_{i,\Delta_n}(x + vh_{i,N}^{\prime}) - (b\zeta)_{i,\Delta_n}^*(x)\right]dv.
    \end{multline}
    Since $b_i^* \in \Sigma(\beta, R)$ and $\zeta_i^* \in \mathcal{C}^{\infty}(\mathbb{R})$, the function $(b\zeta)_i^*$ is of class $\mathcal{C}^{\lfloor\beta\rfloor+1}$. Then, using the Taylor-Lagrange formula, we find that for all $x,v \in \mathbb{R}$, there exists $\xi \in (0,1)$ such that
    \begin{align*}
        (b\zeta)_{i,\Delta_n}^*(x + vh_{i,N}^{\prime}) = &~ \sum_{\ell = 0}^{\lfloor \beta \rfloor}\dfrac{(vh_{i,N}^{\prime})^{\ell}}{\ell!}(b\zeta)_{i,\Delta_n}^{* (\ell)}(x) + \dfrac{(vh_{i,N}^{\prime})^{\lfloor \beta \rfloor + 1}}{(\lfloor \beta \rfloor + 1)!}(b\zeta)_{i,\Delta_n}^{* (\lfloor \beta \rfloor + 1)}(x + vh_{i,N}^{\prime}\xi).
    \end{align*}
    Moreover, under Assumption~\ref{ass:Reg}, $(b\zeta)_{i,\Delta_n}^{* (\lfloor \beta \rfloor + 1)}$ is compactly supported, and $\left\|(b\zeta)_{i,\Delta_n}^{* (\lfloor \beta \rfloor + 1)}\right\|_{\infty} < \infty$.   Under Assumption~\ref{ass:prop-kernel} with $\gamma = \lfloor \beta \rfloor + 1$, we obtain for all $x \in \mathrm{Supp}(b_i^*)$,
    \begin{multline}\label{eq:bias-bf-3}
        \left|\int_{-\infty}^{+\infty}K(v)\left[(b\zeta)_i^*(x + vh_{i,N}^{\prime}) - (b\zeta)_i^*(x)\right]dv\right|\\
        = \int_{-\infty}^{+\infty}K(v)\dfrac{(vh_{i,N}^{\prime})^{\lfloor \beta \rfloor + 1}}{(\lfloor \beta \rfloor + 1)!}(b\zeta)_{i,\Delta_n}^{* (\lfloor \beta \rfloor + 1)}(x + vh_{i,N}^{\prime}\xi)dv = \mathcal{O}\left(h_{i,N}^{\prime\beta}\right).
    \end{multline}
    Equations~\eqref{eq:bias-bf-3}, \eqref{eq:bias-bf-1}, \eqref{eq:bf-2} and \eqref{eq:bf-1} lead to the expected result. 
\end{proof}

\subsection*{Proof of Lemma~\ref{lm:partition}}

\begin{proof}
We prove the result of Lemma~\ref{lm:partition} for any discrete-time version $X^{(n)} = (X_{t_k})_{k = 0 \leq k \leq n}$ of the diffusion process $X = (X_t)_{t \in [0,T]}$ from any subdivision $\{t_0 = 0, \ldots, t_n = T\}$ of the time interval $[0,T]$. The result is then extended to the continuous time process $X$ since for all $n \in \mathbb{N}^*$ and for any subdivision $\{t_0, \ldots, t_n\}$ of the time interval $[0,T]$, $\sigma\left(X_{t_k}, k \in [\![0,n]\!]\right) \subset \mathcal{F}_T = \sigma(X_t, t \in [0,T])$. Fix $n \in \mathbb{N}^*$ such that $n \rightarrow \infty$, and a subdivision $I_n = \{t_0^n, \ldots, t_n^n: 0=t_0^n<t_1^n<\ldots<t_n^n=T\}$ of the time interval $[0,T]$. We consider the discrete-time version $X^{(n)} = (X_{t_k^n})_{0\leq k \leq n}$ of the diffusion process $X$. In this context, the set $\mathcal{X}$ of diffusion paths becomes $\mathcal{X}^{(n)} := \mathbb{R}^n$. Since the diffusion process $X$ admits a transition density $(s,t,x,y) \mapsto \Gamma_X(s,t,x,y)$ given by Equation~\eqref{eq:transition-density-assouad}, the law of the random vector $X^{(n)} = (X_{t_k})_{0\leq k \leq n}$ is absolutely continuous with respect to the Lebesgue measure and its density function $\psi_{X^{(n)}}: x = (x_{1}, \ldots, x_{n}) \in \mathcal{X}^{(n)} \mapsto \psi_{X^{(n)}}(x)$ is given by
\begin{equation}\label{eq:Density-RandomVecttor}
    \psi_{X^{(n)}}(x) = \psi_{X^{(n)}}(x_{1}, \ldots, x_{n}) := \prod_{k=1}^{n}{\Gamma_X(t_{k-1}^n, t_k^n, x_{k-1}, x_{k})}.
\end{equation}
As we can see, the marginal distribution of $X^{(n)} = (X_{t_k^n})_{0\leq k \leq n}$ does not depend on $\vec{\sigma} \in \{-1, +1\}^m$ as the finite set $\Sigma^M$ is independent of $\vec{\sigma}$. Denote by $\widetilde{\xi}^{(n)}: \mathcal{X}^{(n)} \longrightarrow [-1,1]$, the discrete-time version of $\widetilde{\xi}$ (see Equation~\eqref{eq:xitilde}) given for all $x = (x_{t_1}, \ldots, x_{t_n}) \in \mathcal{X}^{(n)}$ by $\widetilde{\xi}^{(n)}(x) = \dfrac{\exp\left(F_f^{(n)}(x)\right) - 1}{\exp\left(F_f^{(n)}(x)\right) + 1}$, where $F_f^{(n)}(x) = \sum_{k=0}^{n-1}{f(x_{k})\left(x_{k+1} - x_{k}\right)} - \dfrac{1}{2}\sum_{k=0}^{n-1}{f^2(x_{k})\left(t_{k+1}^n - t_k^n\right)}$.
The function $\widetilde{\xi}^{(n)}$ is continuous on $\mathcal{X}^{(n)}$ and from Equations~\eqref{eq:xi-plus-less} and \eqref{eq:criterion-xi}, we define:
\begin{equation*}
   \mathcal{X}^{(n)+} = \left\{X \in \mathcal{X}^{(n)}: \widetilde{\xi}^{(n)}(X) > 0\right\} ~~ \mathrm{and} ~~ \mathcal{X}^{(n)-} = \left\{X \in \mathcal{X}^{(n)}: \widetilde{\xi}^{(n)}(X) \leq 0\right\},
\end{equation*}
and we have $\mathcal{X}^{(n)+} \cap \mathcal{X}^{(n)-} = \emptyset, ~~ \mathcal{X}^{(n)+} \cup \mathcal{X}^{(n)-} = \mathcal{X}^{(n)}, ~~ \mu^{(n)}(\mathcal{X}^{(n)+}) > 0 ~~ \mathrm{and} ~~ \mu^{(n)}(\mathcal{X}^{(n)-}) > 0$, where $\mu^{(n)}$ is the Lebesgue measure on $\mathcal{X}^{(n)} = \mathbb{R}^n$. Since $\mu^{(n)}(\mathcal{X}^{(n)+}) > 0$ and $\mu^{(n)}(\mathcal{X}^{(n)-}) > 0$, there exist sequences $(a_i)_{i \in [\![1,n]\!]}, (b_i)_{i \in [\![1,n]\!]}, (c_i)_{i \in [\![1,n]\!]}, (d_i)_{i \in [\![1,n]\!]}$ of values in $\bar{\mathbb{R}} = \mathbb{R} \cup \{-\infty, +\infty\}$ that satisfy the following conditions:
\begin{itemize}
    \item[(i)] 
    \begin{equation*}
        a_i < b_i ~~ \mathrm{and} ~~ c_i < d_i ~~ \forall ~ i \in [\![1,n]\!],
    \end{equation*}
    \item[(ii)] 
    \begin{equation*}
        \mathcal{I}^+ = \prod_{i = 1}^{n}(a_i, b_i) = (a_1, b_1) \times \ldots \times (a_n,b_n) \subset \mathcal{X}^{(n)+},
    \end{equation*}
    \item[(iii)]
    \begin{equation*}
        \mathcal{I}^- = \prod_{i = 1}^{n}(c_i, d_i) = (c_1, d_1) \times \ldots \times (c_n,d_n) \subset \mathcal{X}^{(n)-},
    \end{equation*}
    \item[(iv)]
    \begin{equation}\label{eq:Proba-I+I-}
        \mathbb{P}_X\left(X^{(n)} \in \mathcal{I}^+\right) = \mathbb{P}_X\left(X^{(n)} \in \mathcal{I}^-\right) > 0.
    \end{equation}
\end{itemize}
The fourth point comes from the simple fact that $\mathbb{P}_X\left(X^{(n)} \in \mathcal{X}^{(n)+}\right) > 0$, $\mathbb{P}_X\left(X^{(n)} \in \mathcal{X}^{(n)-}\right) > 0$ and the law of $X^{(n)}$ is atomless, being absolutely continuous with respect to the Lebesgue measure. Therefore, for any $\mathfrak{p} \in (0,1)$ such that $0 < \mathfrak{p} \leq \min\{\mathbb{P}_X(X^{(n)} \in \mathcal{X}^{(n)+}), \mathbb{P}_X(X^{(n)} \in \mathcal{X}^{(n)-})\} < 1$,
there exist $\mathcal{I}^+ \subset \mathcal{X}^{(n)+}$ and $\mathcal{I}^- \subset \mathcal{X}^{(n)-}$ respectively given by $(\mathrm{ii})$ and $(\mathrm{iii})$ such that
\begin{equation*}
    \mathbb{P}_X\left(X^{(n)} \in \mathcal{I}^+\right) = \mathbb{P}_X\left(X^{(n)} \in \mathcal{I}^-\right) = \mathfrak{p}.
\end{equation*}
Denote by $\psi_{X^{(n)}|\mathcal{I}^+}$ and $\psi_{X^{(n)}|\mathcal{I}^-}$ the conditional density functions respectively on events $\{X^{(n)} \in \mathcal{I}^+\}$ and $\{X^{(n)} \in \mathcal{I}^-\}$ and given for all $x = (x_{t_1}, \ldots, x_{t_n}) \in \mathcal{X}^{(n)}$ by
\begin{equation}\label{eq:conditional-densities}
    \psi_{X^{(n)}|\mathcal{I}^+}(x) = \dfrac{\psi_{X^{(n)}}(x)\mathds{1}_{x \in \mathcal{I}^+}}{\mathbb{P}_X\left(X^{(n)} \in \mathcal{I}^+\right)} ~~ \mathrm{and} ~~ \psi_{X^{(n)}|\mathcal{I}^-}(x) = \dfrac{\psi_{X^{(n)}}(x)\mathds{1}_{x \in \mathcal{I}^-}}{\mathbb{P}_X\left(X^{(n)} \in \mathcal{I}^-\right)}.
\end{equation}
Moreover, since $\mathcal{X}^{(n)+} \cap \mathcal{X}^{(n)-} = \emptyset$, there exists $i_0 \in [\![1, n]\!]$ such that $(a_{i_0}, b_{i_0}) \cap (c_{i_0}, d_{i_0}) = \emptyset$. Without loss of generality, we assume that $1 < i_0 < n$. Let $\psi_{X^{(n)}|\mathcal{I}^+}^{i_0}: (a_{i_0}, b_{i_0}) \rightarrow \mathbb{R}$ and $\psi_{X^{(n)}|\mathcal{I}^-}^{i_0}: (c_{i_0}, d_{i_0}) \rightarrow \mathbb{R}$ be the conditional density functions of the marginal distribution of the component $X_{t_{i_0}}$ of $X^{(n)}$ on events $\{X^{(n)} \in \mathcal{I}^+\}$ and $\{X^{(n)} \in \mathcal{I}^-\}$ respectively, given by:
\begin{align*}
    \psi_{X^{(n)}|\mathcal{I}^+}^{i_0}(x) = &~ \int_{a_1}^{b_1} \ldots \int_{a_{i_0-1}}^{b_{i_0-1}} \int_{a_{i_0+1}}^{b_{i_0+1}} \ldots \int_{a_n}^{b_n}\psi_{X^{(n)}|\mathcal{I}^+}^{i_0}(u_1, \ldots, u_{i_0-1}, x, u_{i_0+1}, \ldots, u_n)du^{(-i_0)},\\
    \psi_{X^{(n)}|\mathcal{I}^-}^{i_0}(x) = &~ \int_{c_1}^{d_1} \ldots \int_{c_{i_0-1}}^{d_{i_0-1}} \int_{c_{i_0+1}}^{d_{i_0+1}} \ldots \int_{c_n}^{d_n}\psi_{X^{(n)}|\mathcal{I}^+}^{i_0}(u_1, \ldots, u_{i_0-1}, x, u_{i_0+1}, \ldots, u_n)du^{(-i_0)},
\end{align*}
where $u^{(-i_0)} = \left(u_1, \ldots, u_{i_0-1}, u_{i_0+1}, \ldots, u_n\right) \in \mathbb{R}^{n-1}$ and $du^{(-i_0)} = du_n \ldots du_{i_0+1}du_{i_0-1} \ldots du_1$. The functions $x \mapsto \psi_{X^{(n)}|\mathcal{I}^+}^{i_0}(x)$ and $x \mapsto \psi_{X^{n}|\mathcal{I}^-}^{i_0}(x)$ are strictly positive on the intervals $(a_{i_0}, b_{i_0})$ and $(c_{i_0}, d_{i_0})$ respectively (see Equations~\eqref{eq:conditional-densities}, \eqref{eq:Density-RandomVecttor} and \eqref{eq:transition-density-assouad}). Then, the functions $\Psi_{X^{(n)}|\mathcal{I}^+}^{i_0}: x \in (a_{i_0}, b_{i_0}) \in  \mapsto \int_{a_{i_0}}^{x}\psi_{X^{(n)}|\mathcal{I}^+}^{i_0}(u)du \in [0,1]$
and $\Psi_{X^{(n)}|\mathcal{I}^-}^{i_0}: x \in (c_{i_0}, d_{i_0}) \in  \mapsto \int_{c_{i_0}}^{x}\psi_{X^{(n)}|\mathcal{I}^-}^{i_0}(u)du \in [0,1]$ are continuous and strictly increasing, that is, $\Psi_{X^{(n)}|\mathcal{I}^+}^{i_0}$ and $\Psi_{X^{(n)}|\mathcal{I}^-}^{i_0}$ are bijective functions. Thus, there exist $\lambda_1^{i_0}, \lambda_2^{i_0}, \gamma_1^{i_0}, \gamma_2^{i_0} \in [0,1]$ such that $\lambda_1^{i_0} < \lambda_2^{i_0}$, $\gamma_1^{i_0} < \gamma_2^{i_0}$, $(\lambda_1^{i_0}, \lambda_2^{i_0}) \cap (\gamma_1^{i_0}, \gamma_2^{i_0}) = \emptyset$, and
\begin{align*}
    \Psi_{X^{(n)}|\mathcal{I}^+}^{i_0}((a_{i_0}, b_{i_0})) = (\lambda_1^{i_0}, \lambda_2^{i_0}), ~~ \Psi_{X^{(n)}|\mathcal{I}^-}^{i_0}((c_{i_0}, d_{i_0})) = (\gamma_1^{i_0}, \gamma_2^{i_0}).
\end{align*}
Let $w_0 \in (0,1)$ close enough to $0$, and set $q_1^{i_0} = \lfloor \frac{\lambda_2^{i_0} - \lambda_1^{i_0}}{w_0} \rfloor$ and  $q_2^{i_0} = \lfloor \frac{\gamma_2^{i_0} - \gamma_1^{i_0}}{w_0} \rfloor$. Consider the respective discrete subsets $I_{i_0}^+ = \{\lambda_1^{i_0} + kw_0, ~~ k = 1, \ldots, q_1^{i_0}\}$ and $I_{i_0}^- = \{\gamma_1^{i_0} + jw_0, ~~ j = 1, \ldots, q_2^{i_0}\}$ of the intervals $(\lambda_1^{i_0}, \lambda_2^{i_0})$ and $(\gamma_1^{i_0}, \gamma_2^{i_0})$ respectively. Since the functions $\Psi_{X^{(n)}|\mathcal{I}^+}^{i_0}$ and $\Psi_{X^{(n)}|\mathcal{I}^-}^{i_0}$ are bijective, there exist unique $x_1^{i_0}, \ldots, x_{q_1^{i_0}}^{i_0} \in (a_{i_0}, b_{i_0})$ and unique $y_1^{i_0}, \ldots, y_{q_2^{i_0}}^{i_0} \in (c_{i_0}, d_{i_0})$ such that $x_1^{i_0} < \ldots < x_{q_1^{i_0}}^{i_0}$, $y_1^{i_0} < \ldots < y_{q_2^{i_0}}^{i_0}$ and
\begin{equation}\label{eq:partition}
    \forall ~ (k,j) \in [\![1, q_1^{i_0}]\!] \times [\![1, q_2^{i_0}]\!], ~~ \Psi_{X^{(n)}|\mathcal{I}^+}^{i_0}(x_k^{i_0}) = \lambda_1^{i_0} + kw_0 \in I_{i_0}^+ ~~ \mathrm{and} ~~ \Psi_{X^{(n)}|\mathcal{I}^-}^{i_0}(y_j^{i_0}) = \gamma_1^{i_0} + jw_0 \in I_{i_0}^-.
\end{equation}
Consider the subsets $\mathcal{X}_0^{\prime (n)}, \mathcal{X}_1^{\prime (n)}, \ldots, \mathcal{X}_{q_1^{i_0}}^{\prime (n)}, \mathcal{X}_{q_1^{i_0} + 1}^{\prime (n)}, \ldots, \mathcal{X}_{q_1^{i_0} + q_2^{i_0}}^{\prime (n)}$ of $\mathcal{X}^{(n)} = \mathbb{R}^{n}$ given for all $(k,j) \in [\![1, q_1^{i_0}]\!] \times [\![1, q_2^{i_0}]\!]$ by 
\begin{align*}
   &~ \mathcal{X}_k^{\prime (n)} := (a_1, b_1) \times \ldots \times (a_{i_0-1}, b_{i_0-1}) \times (x_{k-1}^{i_0}, x_k^{i_0}) \times (a_{i_0+1}, b_{i_0+1}) \times \ldots \times (a_{n}, b_{n}),\\\\
   &~ \mathcal{X}_{q_1^{i_0} + j}^{\prime (n)} := (c_1, d_1) \times \ldots \times (c_{i_0-1}, d_{i_0-1}) \times (y_{j-1}^{i_0}, y_j^{i_0}) \times (c_{i_0+1}, d_{i_0+1}) \times \ldots \times (c_{n}, d_{n}),
\end{align*}
and $\mathcal{X}_0^{\prime (n)} := \mathcal{X}^{(n)} \setminus \bigcup_{j=1}^{q_1^{i_0} + q_2^{i_0}}{\mathcal{X}_j^{\prime (n)}}$, where $x_0^{i_0} = a_{i_0}$ and $y_0^{i_0} = c_{i_0}$.\\
Then, the subsets $\mathcal{X}_0^{\prime (n)}, \ldots, \mathcal{X}_{q_1^{i_0} + q_2^{i_0}}^{\prime (n)}$ constitute a partition of $\mathcal{X}^{(n)}$ and, on the one hand, from Equation~\eqref{eq:partition} and for all $k \in [\![1, q_1^{i_0}]\!]$,
\begin{align*}
    \mathbb{P}_X\left(X^{(n)} \in \mathcal{X}_k^{\prime (n)}\right) = &~ \int_{a_1}^{b_1} \ldots \int_{a_{i_0-1}}^{b_{i_0-1}} \int_{x_{k-1}^{i_0}}^{x_k^{i_0}} \int_{a_{i_0+1}}^{b_{i_0+1}} \ldots \int_{a_n}^{b_n}\psi_{X^{(n)}}(u_1, \ldots, u_n)du_n \ldots du_1\\
    = &~ \mathbb{P}_X(X^{(n)} \in \mathcal{I}^+)\int_{x_{k-1}^{i_0}}^{x_k^{i_0}}\psi_{X^{(n)}|\mathcal{I}^+}^{i_0}(x)dx\\
    = &~ \mathbb{P}_X(X^{(n)} \in \mathcal{I}^+)\left[\Psi_{X^{(n)}|\mathcal{I}^+}^{i_0}(x_k^{i_0}) - \Psi_{X^{(n)}|\mathcal{I}^+}^{i_0}(x_{k-1}^{i_0})\right]\\
    = &~ w_0\mathbb{P}_X\left(X^{(n)} \in \mathcal{I}^+\right),
\end{align*}
and, on the other hand, from Equation~\eqref{eq:partition} and for all $j \in [\![1, q_2^{i_0}]\!]$,
\begin{align*}
    \mathbb{P}_X\left(X^{(n)} \in \mathcal{X}_{q_1^{i_0} + j}^{\prime (n)}\right) = &~ \int_{c_1}^{d_1} \ldots \int_{c_{i_0-1}}^{d_{i_0-1}} \int_{y_{j-1}^{i_0}}^{y_j^{i_0}} \int_{c_{i_0+1}}^{d_{i_0+1}} \ldots \int_{c_n}^{d_n}\psi_{X^{(n)}}(u_1, \ldots, u_n)du_n \ldots du_1\\
    = &~ \mathbb{P}_X(X^{(n)} \in \mathcal{I}^-)\int_{y_{j-1}^{i_0}}^{y_j^{i_0}}\psi_{X^{(n)}|\mathcal{I}^-}^{i_0}(x)dx\\
    = &~ \mathbb{P}_X(X^{(n)} \in \mathcal{I}^-)\left[\Psi_{X^{(n)}|\mathcal{I}^-}^{i_0}(y_j^{i_0}) - \Psi_{X^{(n)}|\mathcal{I}^-}^{i_0}(y_{j-1}^{i_0})\right]\\
    = &~ w_0\mathbb{P}_X\left(X^{(n)} \in \mathcal{I}^-\right).
\end{align*}
In addition, from Equation~\eqref{eq:Proba-I+I-}, set $w = w_0\mathbb{P}_X\left(X^{(n)} \in \mathcal{I}^+\right) = w_0\mathbb{P}_X\left(X^{(n)} \in \mathcal{I}^-\right) \in (0,1)$
and $q = \min\left\{w^{-1}, q_1^{i_0} + q_2^{i_0}\right\}$. 
We deduce that for all $(k,j) \in [\![1, q_1^{i_0}]\!] \times [\![1, q_2^{i_0}]\!]$,
\begin{equation*}
    \mathbb{P}_X\left(X^{(n)} \in \mathcal{X}_k^{\prime (n)}\right) = \mathbb{P}_X\left(X^{(n)} \in \mathcal{X}_{q_1^{i_0} + j}^{\prime (n)}\right) = w.
\end{equation*}
We assume that $w_0$ is close enough to $0$ so that $q > m$ is large enough with respect to $m$. Then, from partition $\mathcal{P}^{\prime (n)} = \left\{\mathcal{X}_0^{\prime (n)}, \ldots, \mathcal{X}_q^{\prime (n)}\right\}$, we deduce a new partition $\mathcal{P}^{(n)} = \left\{\mathcal{X}_0^{(n)}, \mathcal{X}_1^{(n)} \ldots, \mathcal{X}_m^{(n)}\right\}$ of $\mathcal{X}^{(n)}$ such that: 
\begin{itemize}
    \item for all $j \in \{1, \ldots, m\}$, $\mathcal{X}_j^{(n)} \in \mathcal{P}^{\prime (n)}$ and $\mathbb{P}_X\left(X^{(n)} \in \mathcal{X}_j^{(n)}\right) = w$,
    \item for all $j \in \{1, \ldots, m\}$, $\mathcal{X}_j^{(n)} \subset \mathcal{X}^{(n)-} ~~ \mathrm{or} ~~ \mathcal{X}_j^{(n)} \subset \mathcal{X}^{(n)+}$ and $\mathcal{X}_0^{(n)} = \mathcal{X}^{(n)} \setminus \bigcup_{j=1}^{m}{\mathcal{X}_j^{(n)}}$.
\end{itemize}
Finally, since for all $X^{(n)} \in \mathcal{X}^{(n)}$, $\Phi_f(X^{(n)}) = \mathbb{P}_{X,Y}\left(Y=1 | X^{(n)}\right) = \dfrac{1 + \widetilde{\xi}^{(n)}(X^{(n)})}{2}$, 
for any $j \in \{1, \ldots, m\}$, if $\mathcal{X}^{(n)}_j \subset \mathcal{X}^{(n)-}$, then for all $X^{(n)} \in \mathcal{X}_j^{(n)}$, $\widetilde{\xi}^{(n)}(X^{(n)}) \leq 0$, and for $\sigma_j = -1$, we obtain for all $X^{(n)} \in \mathcal{X}_j^{(n)}$, $\widetilde{\xi}^{(n)}(X^{(n)}) = \sigma_j\xi^{(n)}(X^{(n)})$,
where $\xi^{(n)} = \left|\widetilde{\xi}^{(n)}\right|: \mathcal{X}^{(n)} \longrightarrow [0,1]$, which leads to 
\begin{align*}
    \Phi_{\vec{\sigma}, f}(X^{(n)}) = \mathbb{P}_{X,Y}(Y=1 | X^{(n)}) = \dfrac{1 + \sigma_j\xi^{(n)}(X^{(n)})}{2} = 1 - \mathbb{P}_{X,Y}(Y=1 | X^{(n)}). 
\end{align*}
For the case $\mathcal{X}_j^{(n)} \subset \mathcal{X}^{(n)+}$, a similar reasoning is applied with $\sigma_j = +1$.
\end{proof}

\subsection*{Proof of Lemma~\ref{lm:Bounds-b-bprime}}

\begin{proof}
Recall that the diffusion process $X$, unique strong solution of Equation~\eqref{eq:model2}, admits a transition density that is strictly positive. Since $\mathbb{P}_X \circ X^{-1} \sim \mathcal{W}$, $\mathbb{P}_X(X \in \mathcal{X}_1) = w > 0$ implies that $\mathcal{W}(\mathcal{X}_1) > 0$. As a result, conditional on $\{X \in \mathcal{X}_1\}$, for all $t \in [0,T]$, we have $\mu(\left\{X_t(\omega), ~ \omega \in \Omega\right\}) > 0$, which means that for all $t \in [0,T]$, the random variable $X_t$ takes values in a continuous subset of $\mathbb{R}$. We have
\begin{align*}
    \mathfrak{b} = \sqrt{1 - \left(\mathbb{E}_X\left[\sqrt{1 - \xi^2(X)} \biggm\vert X \in \mathcal{X}_1\right]\right)^2}, ~~~ \mathfrak{b}^{\prime} = \mathbb{E}_X\left[\xi(X) | X \in \mathcal{X}_1\right],
\end{align*}
where $\xi(X) = |\widetilde{\xi}(X)|$ and $\widetilde{\xi}(X)$ given by Equation~\eqref{eq:xitilde}. For all $f = \kappa D^{-\beta} + \sum_{k=1}^{D}{\theta_k\phi_k} \in \Sigma^M$, $ F_f(X) = \int_{0}^{T}{f(X_s)dX_s} - \dfrac{1}{2}\int_{0}^{T}{f^2(X_s)ds} = D^{-\beta}Q_{\widetilde{f}}(X)$, where 
\begin{equation}\label{eq:Q-f}
    Q_{\widetilde{f}}(X) = \int_{0}^{T}{\widetilde{f}(X_s)dW_s} + \dfrac{D^{-\beta}}{2}\int_{0}^{T}{\widetilde{f}^2(X_s)ds},
\end{equation} 
and the function $\widetilde{f}$ is given for all $x \in [0, 1]$ by $\widetilde{f}(x) = \kappa + R\sum_{k=1}^{D}{\theta_kK\left(\frac{x-x_k}{D^{-1}}\right)}$, and satisfies the requirements:
\begin{equation}\label{eq:Bounds-f-tilde}
    x \in [0, 1], ~ \kappa \leq \widetilde{f}(x) \leq R\|K\|_{\infty}.
\end{equation}
Then, using the Taylor-Young expansion, when $D \rightarrow \infty$:
\begin{equation}\label{eq:Equiv-xi}
    \begin{aligned}
        \exp\left(F_f(X)\right) - 1 = &~ D^{-\beta}Q_{\widetilde{f}}(X) + \dfrac{D^{-2\beta}Q_{\widetilde{f}}^2(X)}{2} + {o}\left(D^{-2\beta}Q_{\widetilde{f}}^2(X)\right),\\
        \dfrac{1}{\exp\left(F_f(X)\right) + 1} = &~ \dfrac{1}{2} - \dfrac{D^{-\beta}Q_{\widetilde{f}}(X)}{4} +{o}\left(D^{-2\beta}Q_{\widetilde{f}}^2(X)\right).
    \end{aligned}
\end{equation}
We deduce that $\xi(X) \underset{D \rightarrow \infty}{\sim} 2D^{-\beta}\left|Q_{\widetilde{f}}(X)\right|$. Thus, there exist constants $c_{\xi}, C_{\xi} > 0$ such that 
\begin{equation}\label{eq:equivalence}
     \forall ~ X \in \mathcal{X}, ~ c_{\xi}D^{-\beta}\left|Q_{\widetilde{f}}(X)\right| \leq \xi(X) \leq C_{\xi}D^{-\beta}\left|Q_{\widetilde{f}}(X)\right| ~~ a.s.
\end{equation}

\subsection*{Lower bound of $\mathfrak{b}^{\prime}$}

By Equation~\eqref{eq:equivalence},
\begin{equation}\label{eq:LB-bprime}
    \mathfrak{b}^{\prime} = \mathbb{E}_X\left[\xi(X) | X \in \mathcal{X}_1\right] \geq c_{\xi}D^{-\beta}\mathbb{E}_X\left[\left|Q_{\widetilde{f}}(X)\right| \biggm\vert X \in \mathcal{X}_1\right].
\end{equation}
It remains to show that the quantity $\mathbb{E}_X\left[\left|Q_{\widetilde{f}}(X)\right| \biggm\vert X \in \mathcal{X}_1\right]$ is bounded from below by a strictly positive constant that does not depend on $N$. For this purpose, let $C_0 > 0, c_0 > 0$ be two numerical constants to be chosen later so that $C_0$ is large enough with respect to $c_0$, and set
\begin{equation*}
    \mathcal{B} = \left\{\underset{t \in [0,T]}{\sup}{\left|\int_{0}^{t}\widetilde{f}(X_s)dW_s\right|} \leq C_0\right\}.
\end{equation*}
From Equation~\eqref{eq:Q-f}, we have the following.
\begin{align*}
    &~ \mathbb{E}_X\left[\left|Q_{\widetilde{f}}(X)\right| \biggm\vert X \in \mathcal{X}_1\right]\\
    &~ \geq c_0\mathbb{P}_X\left(\left|\int_{0}^{T}\widetilde{f}(X_s)dW_s + \dfrac{D^{-\beta}}{2}\int_{0}^{T}\widetilde{f}^2(X_s)ds\right| \geq c_0 \biggm\vert \mathcal{B} \cap \{X \in \mathcal{X}_1\}\right)\mathbb{P}_X\left(\mathcal{B} \biggm\vert X \in \mathcal{X}_1\right)\\
    &~ \geq c_0\mathbb{P}_X\left(\left|\int_{0}^{T}\widetilde{f}(X_s)dW_s\right| \geq c_0 + \dfrac{1}{2}T\left\|\widetilde{f}\right\|_{\infty}^2D^{-\beta} \biggm\vert \mathcal{B} \cap \{X \in \mathcal{X}_1\}\right)\mathbb{P}_X\left(\mathcal{B} \biggm\vert X \in \mathcal{X}_1\right).
\end{align*}
Since $T\|\widetilde{f}\|_{\infty}^2D^{-\beta}/2 \rightarrow 0$ as $D \rightarrow \infty$, for $D$ large enough, $T\|\widetilde{f}\|_{\infty}^2D^{-\beta}/2 \leq c_0$. We deduce that
\begin{equation}\label{eq:LB-b}
    \mathbb{E}_X\left[\left|Q_{\widetilde{f}}(X)\right| \biggm\vert X \in \mathcal{X}_1\right]
    \geq c_0\mathbb{P}_X\left(\left|\int_{0}^{T}\widetilde{f}(X_s)dW_s\right| \geq 2c_0 \biggm\vert \mathcal{B} \cap \{X \in \mathcal{X}_1\} \right)\mathbb{P}_X\left(\mathcal{B} \biggm\vert X \in \mathcal{X}_1\right).
\end{equation}
Focusing on the second factor on the right-hand side of Equation~\eqref{eq:LB-b}, we obtain the following result.
\begin{align*}
    &~ \mathbb{P}_X\left(\left|\int_{0}^{T}\widetilde{f}(X_s)dW_s\right| > 2c_0 \biggm\vert \mathcal{B} \cap \{X \in \mathcal{X}_1\}\right)\\
    &~  = 1 - \mathbb{P}_X\left(-\left[\left(\int_{0}^{T}\widetilde{f}(X_s)dW_s\right)^2 - \int_{0}^{T}\widetilde{f}^2(X_s)ds\right] \geq \int_{0}^{T}\widetilde{f}^2(X_s)ds - 4c_0^2 \biggm\vert \mathcal{B} \cap \{X \in \mathcal{X}_1\}\right).
\end{align*}
We choose $c_0 = \kappa T^{1/2}/4$. Then, from Equation~\eqref{eq:Bounds-f-tilde}, we obtain
\begin{equation}\label{eq:LB-b-2}
    \mathbb{P}_X\left(\left|\int_{0}^{T}\widetilde{f}(X_s)dW_s\right| > c_0 \biggm\vert \mathcal{B} \cap \{X \in \mathcal{X}_1\}\right) \geq 1 - \mathbb{P}_X\left(-M_T > \dfrac{3T\kappa^2}{4} \biggm\vert \mathcal{B} \cap \{X \in \mathcal{X}_1\}\right),
\end{equation}
where $\mathcal{F}_t -$Martingale $M = (M_t)_{t \in [0,T]}$ is given by $M_t := \left(\int_{0}^{t}\widetilde{f}(X_s)dW_s\right)^2 - \int_{0}^{t}\widetilde{f}^2(X_s)ds, ~~ t \in [0,T]$. In event $\mathcal{B} \cap \{X \in \mathcal{X}_1\}$ and from Equation~\eqref{eq:Bounds-f-tilde}, the quadratic variation of $M$ satisfies the following:
\begin{align*}
    \forall ~ t \in [0,T], ~~ \left<M,M\right>_t = &~ 4\int_{0}^{t}\left(\int_{0}^{s}\widetilde{f}(X_u)dW_u\right)^2\widetilde{f}^2(X_s)ds \leq 4T\left\|\widetilde{f}\right\|_{\infty}^2C_0^2 \leq 4TR^2\|K\|_{\infty}^2C_0^2.
\end{align*}
Then, from \cite{van1995exponential}, \textit{Lemma 2.1}, we obtain from Equation~\eqref{eq:LB-b-2} that
\begin{equation}\label{eq:lbound}
    \mathbb{P}_X\left(\left|\int_{0}^{T}\widetilde{f}(X_s)dW_s\right| > c_0 \biggm\vert \mathcal{B} \cap \{X \in \mathcal{X}_1\}\right) \geq 1 - \exp\left(-\dfrac{9T\kappa^4}{128R^2\|K\|_{\infty}^2C_0^2}\right) > 0.
\end{equation}
Focusing on the third factor on the right-hand side of Equation~\eqref{eq:LB-b}, Doob's $L^2 -$inequality gives:
\begin{align*}
    \mathbb{P}_X\left(\mathcal{B} \biggm\vert X \in \mathcal{X}_1\right) = &~ 1 - \mathbb{P}_X\left(\underset{t \in [0,T]}{\sup}{\left|\int_{0}^{t}\widetilde{f}(X_s)dW_s\right|} \geq C_0 \biggm\vert X \in \mathcal{X}_1\right)\\
    \geq &~ 1 - \dfrac{1}{C_0^2}\underset{t \in [0,T]}{\sup}{\mathbb{E}_X\left[\left(\int_{0}^{t}\widetilde{f}(X_s)dW_s\right)^2 \biggm\vert X \in \mathcal{X}_1\right]} \geq 1 - \dfrac{TR^2\|K\|_{\infty}^2}{C_0^2}.
\end{align*}
The numerical constant $C_0 > 0$ is chosen so that $C_0 > \max\{5c_0, 2T^{1/2}R\|K\|_{\infty}\}$, which implies that
\begin{equation}\label{eq:LB-b-3}
    \mathbb{P}_X\left(\mathcal{B} \biggm\vert X \in \mathcal{X}_1\right) \geq 1 - \dfrac{TR^2\|K\|_{\infty}^2}{C_0^2} > 0.
\end{equation}
From Equations~\eqref{eq:LB-b-3}, \eqref{eq:lbound}, and \eqref{eq:LB-b}, we obtain the following.
\begin{equation}\label{eq:LB-b-4}
    \mathbb{E}_X\left[\left|Q_{\widetilde{f}}(X)\right| \biggm\vert X \in \mathcal{X}_1\right]
    \geq \left(1 - \exp\left(-\dfrac{9T\kappa^4}{128R^2\|K\|_{\infty}^2C_0^2}\right)\right)\left(1 - \dfrac{TR^2\|K\|_{\infty}^2}{C_0^2}\right) > 0.
\end{equation}
Finally, we deduce from Equations~\eqref{eq:LB-b-4} and \eqref{eq:LB-bprime} that there exists a constant $c^{\prime}>0$ such that
\begin{equation*}
    \mathfrak{b}^{\prime} = \mathbb{E}_X\left[\xi(X) | X \in \mathcal{X}_1\right] \geq c^{\prime}D^{-\beta}.
\end{equation*}

\subsection*{Upper bound of $\mathfrak{b}$}

We deduce from Equation~\eqref{eq:Equiv-xi} that $\sqrt{1-\xi^2(X)} =  1 - \dfrac{1}{8}D^{-2\beta}Q_{\widetilde{f}}^2(X) + o\left(D^{-2\beta}Q_{\widetilde{f}}^2(X)\right)$ which implies $\mathfrak{b} = \sqrt{1 - \left(\mathbb{E}_X\left[\sqrt{1 - \xi^2(X)} \vert X \in \mathcal{X}_1\right]\right)^2} = \mathcal{O}\left(D^{-\beta}\sqrt{\mathbb{E}_X[Q_{\widetilde{f}}^2(X)|X \in \mathcal{X}_1]}\right)$. To conclude the proof, we refer to Equation~\eqref{eq:Q-f} and remark that for $D \rightarrow \infty$, $\mathbb{E}_X[Q_{\widetilde{f}}^2(X)|X \in \mathcal{X}_1] \leq w^{-1}\mathbb{E}_X[Q_{\widetilde{f}}^2(X)] \leq 3T\|\widetilde{f}\|_{\infty}^2$.
\end{proof}

\end{document}